\documentclass[11pt,reqno]{amsart}
\usepackage{hyperref}
\usepackage{amsmath,amsfonts,amsthm}
\usepackage[mathscr]{euscript}
\usepackage{libertine}

\usepackage{enumitem}
\usepackage{tikz}
\usetikzlibrary{shapes,arrows}

\newif\ifarch
\archtrue
\newif\ifjour
\jourfalse
\newif\ifdraft
\drafttrue

\ifdraft
\else
\fi

\newtheorem{thm}{Theorem}[section]
\newtheorem{prop}[thm]{Proposition}
\newtheorem{lemma}[thm]{Lemma}
\newtheorem{cor}[thm]{Corollary}
\theoremstyle{definition}
\newtheorem{defn}[thm]{Definition}
\theoremstyle{remark}
\newtheorem{remark}[thm]{Remark}

\newtheorem{example}{Example}

\newcommand{\examoreref}{}

\def\rankD{2d} 
\def\rankDplus{d} 
\def\dimM{m} 
\def\dimG{n} 
\def\rankVinf{q}
\def\dimN{p} 
\def\rankH{h} 

\newcommand{\myvf}[1]{{\,\mathrm{#1}}}
\newcommand{\vD}{\myvf D}
\newcommand{\vN}{\myvf N}
\newcommand{\vS}{\myvf S}
\newcommand{\vT}{\myvf T}
\newcommand{\vU}{\myvf U}

\newcommand{\vW}{\myvf W}
\newcommand{\vX}{\myvf X}
\newcommand{\vY}{\myvf Y}
\newcommand{\vZ}{\myvf Z}
\newcommand{\vZR}{\vZ^R} 
\newcommand{\vZL}{\vZ^L}  

\newcommand\ra{{k}}
\newcommand\rb{\ell}
\newcommand\Ug{{u}}

\newcommand{\pib}[1]{\overline{\pi^{#1}}}
\newcommand\vtheta{\boldsymbol{\theta}}
\newcommand\vpi{\boldsymbol{\pi}}
\newcommand\vpibar{\overline{\vpi}}
\newcommand\veta{\boldsymbol{\eta}}
\newcommand\vetabar{\overline{\veta}}
\newcommand\vsigma{\boldsymbol{\sigma}}
\newcommand\sigmabar{\overline{\sigma}}
\newcommand\vsigmabar{\overline{\vsigma}}
\newcommand\pibar{\overline{\pi}}
\newcommand\etabar{\overline{\eta}}
\newcommand\thetau{\vartheta}
\newcommand\thetaubi{\overline{\vartheta^i}}%
\newcommand\thetaubj{\overline{\vartheta^j}}%
\newcommand\thetaubk{\overline{\vartheta^k}}%
\newcommand\thetaubl{\overline{\vartheta^l}}%

\newcommand\npt{\mathrm n}

\newcommand\tfA{\mathscr{A}} 

\newcommand\dD{\mathcal D} 
\newcommand\DD{\mathcal D}
\newcommand\Dplus{\dD_+} 
\newcommand\Dminus{\dD_-}
\newcommand\Dplusinf{\dD_+^{(\infty)}}
\newcommand\Dminusinf{\dD_-^{(\infty)}}
\newcommand\Deltaplus{\Delta_+} 
\newcommand\Deltaminus{\Delta_-}

\newcommand\dDC{\dD\otimes \C}
\newcommand\dS{\mathscr S}
\newcommand{\JJ}{\mathsf J} 

\newcommand\rH{H}
\newcommand\rE{E}
\newcommand\rI{I}
\newcommand\rV{V}
\newcommand\rVb{\overline{V}}

\newcommand\VB{Vert}

\newcommand\vv{\mathsf v} 

\newcommand\sysE{\mathcal E}
\newcommand\sysI{\mathcal I}
\newcommand\IC{I}

\newcommand\sysH{\mathcal H}
\newcommand\subalg{{}_{\operatorname{alg}}}
\newcommand\subdiff{{}_{\operatorname{diff}}}
\newcommand{\Jhol}[1]{J^{#1}_{\operatorname{hol}}}
\newcommand\Vinf{V^{(\infty)}} 
\newcommand\Invts{{\mathscr H}(V)}
\newcommand\calI{\mathcal I} 

\newcommand\calE{\mathcal E}

\newcommand\calH{\mathcal H}
\newcommand\calV{\mathcal V}
\newcommand\calW{\mathcal W}

\newcommand\Upset{\scrV} 
\newcommand\Udown{\scrU}
\newcommand\Wdown{\scrW}
\newcommand\EDelta{\Sigma}

\newcommand{\zbar}{\overline{z}}
\newcommand{\zbars}[1]{\overline{z^{#1}}}
\newcommand{\Ubar}{\overline{U}}
\newcommand{\Wbar}{\overline{W}}
\newcommand{\vUbars}[1]{\overline{\vU_{#1}}}
\newcommand{\pibars}[1]{\overline{\pi^{#1}}}
\newcommand{\thetabars}[1]{\overline{\theta^{#1}}}
\newcommand{\etabars}[1]{\overline{\eta^{#1}}}
\newcommand{\sigmabars}[1]{\overline{\sigma^{#1}}}

\newcommand{\bq}{\mathbf{q}}

\newcommand{\bqg}{\bq_G}

\newcommand\piG{\pi_G} 
\newcommand\piK{\pi_K}
\newcommand\qG{{\mathbf{q}}_{G}} 

\newcommand{\ws}[1]{w_{#1}}
\newcommand{\wbars}[1]{\overline{\ws{#1}}} 

\newcommand{\Id}{\operatorname{Id}}

\newcommand{\fg}{\mathfrak{g}}
\newcommand{\fk}{\mathfrak{k}}

\newcommand\muL{\mu_L} 
\newcommand\muR{\mu_R}
\newcommand{\Gam}{\boldsymbol\Gamma}
\newcommand{\gam}{\mbox{\raisebox{.43ex}{$\boldsymbol\gamma$}}} 
\newcommand{\Slice}{\mathcal{S}} 
\newcommand{\Sup}{\widehat\Slice}

\newcommand\C{\mathbb C}
\newcommand\R{\mathbb R}
\newcommand{\ri}{\mathrm i} 

\newcommand{\mpt}{{\mathrm m}}
\newcommand\spanc{\operatorname{span}_{\C}}
\newcommand\spanr{\operatorname{span}_{\R}}
\renewcommand\Re{\operatorname{Re}}
\renewcommand\Im{\operatorname{Im}}
\newcommand\rk{\operatorname{rank}}
\newcommand\rank{\operatorname{rank}}
\newcommand{\di}{\partial}
\newcommand{\dib}[1]{\partial_{{#1}}} 

\newcommand\dx{\,dx} 
\newcommand\dy{\,dy}
\newcommand\dz{\,dz}

\newcommand\DI{{w}} 

\newcommand\scrU{\mathscr{U}} 
\newcommand\scrV{\mathscr{V}} 
\newcommand\scrW{\mathscr{W}}

\newcommand\tC{\widetilde{C}}  

\newcommand\intprod{\mathbin{\raisebox{.4ex}{\hbox{\vrule height .5pt width
5pt depth 0pt %
         \vrule height 3pt width .5pt depth 0pt}}}}
\newcommand{\hook}{\intprod}

\DeclareFontFamily{U}{mathx}{}
\DeclareFontShape{U}{mathx}{m}{n}{<-> mathx10}{}
\DeclareSymbolFont{mathx}{U}{mathx}{m}{n}
\DeclareMathAccent{\widehat}{0}{mathx}{"70}
\DeclareMathAccent{\widecheck}{0}{mathx}{"71}

\newcommand{\defenumstyle}{{[\roman*]}}
\newcommand{\thmenumstyle}{{(\arabic*)}}
\newcommand{\alfenumstyle}{{(\alph*)}}

\newcommand\defem{\sl} 

\newcommand\restr{\big\vert}
\def\realform{real version}

\begin{document}
\ifarch
\title{The Geometry of Darboux-Integrable Elliptic Systems}
\fi

\author{Mark E. Fels}
\address{Department of Mathematics and Statistics, Utah State University}
\email{mark.fels@usu.edu}
\author{Thomas A. Ivey}
\address{Dept. of Mathematics, College of Charleston}
\email{iveyt@cofc.edu}
\date{\today}

\begin{abstract} We characterize real elliptic differential systems whose solutions can be expressed
in terms of holomorphic solutions to an associated holomorphic Pfaffian system $\sysH$ on a complex manifold. In particular, 
these elliptic systems arise as quotients by a group $G$ of the real differential system generated by the real and imaginary parts of $\sysH$, such that $G$ is the real form of a complex Lie group $K$ which is a symmetry group of $\sysH$. 
Subject to some mild assumptions, we show that such elliptic systems are characterized by a property known as Darboux integrability.  Examples discussed include first- and second-order elliptic PDE and PDE systems in the plane.
\end{abstract}
\maketitle

\ifarch
\tableofcontents
\fi


\newcommand\Vhh{\widehat{V}}
\newcommand\Vhhinf{\Vhh^{(\infty)}}
\newcommand\Vc{\widecheck V} 
\newcommand\Vcinf{\Vc^{(\infty)}} 
\newcommand\Nh{\hat{N}}
\newcommand\Nc{\check{N}}
\newcommand\Eh{\hat{E}}
\newcommand\Ec{\check{E}}
\newcommand\sysEh{\widehat{\sysE}}
\newcommand\sysEc{\widecheck{\sysE}}
\newcommand\bpi{\mathbf{\pi}}
\def\intem{\bf}

\section{Introduction}
Classically, a Darboux integrable (DI)  
second-order partial differential equation  (PDE) in the plane (i.e., for one function of two independent variables) is a hyperbolic PDE compatible with a secondary system  of PDE 
which together make a system of total differential equations satisfying the Frobenius integrability condition. 
This in turn can lead  to a closed form general solution to the original PDE (see, e.g., the survey \cite{Iantalk}). 
In the pioneering work \cite{Vessiot} Ernest Vessiot 
found an alternative method
for finding the general solution to Darboux integrable equations through a generalization of equations of Lie type on homogeneous spaces (see Chapter 12 in \cite{Stormark}). 
This idea was expanded upon and formalized by Anderson, Fels and Vassiliou \cite{AFV} to a class of hyperbolic exterior differential systems whose solutions could be obtained using the theory of Lie group quotients of exterior differential systems. The quintessential example here is the Liouville equation $u_{xy}=2e^u$, for which the general solution can be expressed in terms of the two generic 1-variable functions $f(x),g(y)$ as
\begin{equation}
u(x,y)= \ln \dfrac{f'(x) g'(y)}{(f(x) + g(y))^2}
\label{Hliou}
\end{equation}
This solution can be derived by representing the Liouville equation as the quotient of the product of two jet spaces with their contact system  by an action of $SL(2,\R)$. The 
details of this construction can be found in Example 3.1 of \cite{AFBA}. The article \cite{AFV} shows how the Lie algebra of $SL(2,\R)$ arises as an invariant of the Liouville equation and that it allows for the construction of the general solution to the equation through this quotient construction.  In recognition of the work of Vessiot this algebra was termed  {\intem the Vessiot algebra for  the Liouville equation} in \cite{AFV}.


In this paper, our ultimate goal is to characterize elliptic differential equations whose general solution admit formulas similar to that in \eqref{Hliou}, and
to show how the theory of group quotients plays a role similar to the hyperbolic case. For example, there are two inequivalent versions of the elliptic Liouville equation, namely $u_{xx}+u_{yy} \pm 2 e^u=0$.  
As is well known (see \S13.3 in \cite{Dubrovin}), these equations arises for $e^{u}$ to be the conformal factor modifying a flat metric in the plane to have constant positive or negative Gauss curvature.  
The  general solutions for these two equations have a form similar to the one above for the hyperbolic Liouville equation and are given in terms of a holomorphic function $f(z)$ by
\begin{equation}
u_{+}(x,y)=\ln \frac{ 4 f(z)' \overline {f(z)'} }{(1+f (z)\overline{f(z)})^2}, \qquad
u_{-}(x,y)=\ln \frac{ f(z)' \overline {f(z)'} }{(\Im f(z))^2}.
\label{ELpm}
\end{equation}
The two versions of the elliptic Liouville equation are Darboux integrable and are distinguished in this work by an invariant we again call the  {\intem Vessiot algebra} which is the direct analogue of the terminology in \cite{AFV}. 
The Vessiot algebras for these equations consist of the two real forms of $\mathfrak{sl}(2,\C)$, and the corresponding subgroups of $SL(2,\C)$ are used to find the solutions \eqref{ELpm} using a quotient theory of differential systems which we develop below in analogy to the hyperbolic case.  
In Example \ref{Liou12} below, the $u_{+}(x,y)$ solution is derived  using an $SU(2)$  quotient while the $u_{-}(x,y)$ solution is derived from an $SL(2,\R)$ quotient.

A striking consequence of the invariance of the Vessiot algebra is the fact that while in the hyperbolic Liouville equation $u_{xy} =2 e^{u}$ a minus sign on the right hand side may be disposed of by a trivial change of coordinates, the two versions $\Delta u = 2e^u$ and $\Delta u=-2e^{u}$ of the elliptic Liouville equation are not equivalent, since their Vessiot algebras are not isomorphic; see Corollary \ref{NonCont}.  Another application of the Vessiot algebra  can be found in \cite{FelsIveyCMC}, where we uncover the surprising fact that the PDE  for constant mean curvature (CMC) one surfaces in hyperbolic 3-space (which are Darboux integrable) and the `plus' elliptic Liouville equation are the same equation up to a change of variables. This equivalence was discovered from the fact that  these two equations have the same Vessiot algebra. A similar relation holds between the `minus' Liouville equation and CMC-1 surface equation in de Sitter space.

\smallskip

In order to generalize the notion of hyperbolic differential systems formulated in \cite{AFV},
and based on the algebraic structure of the Pfaffian systems occurring in second-order elliptic PDE in the plane (analyzed in Example \ref{introE1} below), we define an {\intem elliptic decomposable} exterior differential system on a manifold $M$ to be a system $\sysI$ equipped with a bundle $V \subset T^*M \otimes \C$ such that $V \cap \overline{V}$ spans the 1-forms of $\sysI \otimes \C$ and the 2-form generators of $\sysI \otimes \C$ may be chosen to lie in $\Lambda^2 V$ or in $\Lambda^2 \overline{V}$. (Further technical requirements are given in Definition \ref{EDdef}.). The bundle $V$ is called the singular bundle due to its relationship with singular integral elements of $\sysI\otimes \C$, as described in Corollary \ref{Dsingular} and Example \ref{introE1}.  Our definition implies that the real distribution $\dD$ annihilated by the 1-forms of $\sysI$ is even-dimensional and carries an `almost sub-complex structure', i.e., an endomorphism $\JJ$ satisfying $\JJ^2 = -\Id$ (see \S\ref{Bestone} for details). 

We then give a definition for Darboux integrability in Definition \ref{EDIdef} which is a straightforward analogue of the notion of Darboux integrability for hyperbolic systems used in \cite{AFV}.  
Essentially, Darboux integrability means that the singular bundle $V$ has sufficiently many independent first integrals,  i.e., complex-valued functions on $M$ whose differentials are sections of $V$.  
We refer to these functions as {\intem holomorphic Darboux invariants}; the justification for this term is that, when $\sysI$ is Darboux integrable, $\JJ$ restricts to be a genuine complex structure on any integral manifold whose tangent spaces are $\JJ$-invariant, 
and the Darboux invariants restrict to be $\JJ$-holomorphic functions.  
In the case of elliptic PDE in the plane, for integral manifolds of real dimension two, any two holomorphic Darboux invariants must be functionally related to each other on a solution; in fact, imposing a functional relationship between independent Darboux invariants amounts to restricting the system $\sysI$ to a submanifold of $M$ on which it becomes a Frobenius system 
and this takes us back to the original idea of Darboux integrability. Examples demonstrating this property are given in  \S \ref{ClosedForm}. 

As with the case of scalar PDE in the plane (and hyperbolic systems \cite{AFV}), we also associate to a decomposable Darboux integrable elliptic systems the singular differential system $\calV$ in Definition \ref{SSdef}.  To obtain a structure theory for Darboux integrable systems, we limit ourselves to the case where the singular system is Pfaffian.

\subsection*{Comparison and Summary of Results}

We state below a fundamental theorem of Darboux integrability, which is proved in the hyperbolic case in \cite{AFV} and in the elliptic case here, and can be described in terms of quotients of differential systems by a symmetry group $G$ and  the notion of a {\it symmetric pair} which appears in the classification of symmetric spaces \cite{Helg}. 

A group $G$ acting on $M$ is a symmetry group of a differential system $\sysI$  if $g^*\sysI=\sysI$ for all $g\in G$. We will assume that all actions are regular in the sense that the orbit space $M/G$ is a smooth manifold and $\piG:M \to M/G$ is a smooth submersion. In this case the quotient  $\sysI/G$ is the EDS on $M/G$  defined by 
\begin{equation*}
\sysI/G = \{\ \theta \in \Omega^*(M/G) \ \ | \ \piG^*\, \theta \in \sysI \ \},
\end{equation*}
while  $\pi_G^*(\sysI/G)$ are the $G$-basic forms in $\sysI$. 
The action of the symmetry group $G$ is transverse to $\sysI$ if  $I_1^\perp \cap \ker {\piG}_* = 0$, 
in which case $\sysI/G$ is a constant rank EDS. 
The system $\sysI/G$ is the group reduction of $\sysI$, while $\sysI$ is an {\defem integrable extension}  of $\sysI/G$ (see Definition \ref{DefIE}). 
The transversality condition on the action of $G$ implies that $\piG$ maps integral manifolds of $\sysI$ to integral manifolds of $\sysI/G$, 
while the integrable extension property implies that
every integral manifold of $\sysI/G$ may be lifted (locally) to an integral manifold of $\sysI$ by an application of the Frobenius theorem. 
Further useful facts about $\sysI/G$ and its computation can be found in Section 2.2 of \cite{AFB}.

 A pair of Lie groups $(K,G)$ is called a symmetric pair if there exists an involutive automorphism  $\sigma :K \to K$ such that
 the subgroup $G\subset K$ is the fixed point set of $\sigma$. 
Somewhat analogous to symmetric spaces, these play a key role in the characterization of Darboux integrable systems described in the following fundamental theorem.  However, the symmetric pairs that appear here are of a special type where either $K = G \times G$ and the automorphism fixes the diagonal subgroup, or $K$ is the complexification
of $G$ and the automorphism is complex conjugation; we refer to such symmetric pairs as being of {\em Darboux type}.

\begin{thm} \label{FTDI} A differential system $(\sysI,M)$ is a Darboux integrable decomposable (hyperbolic or elliptic) differential system where the singular systems are Pfaffian,  
if and only if there exists and EDS $\sysE$ on a manifold $N$ and a symmetric pair $(K,G)$ of Darboux type such that

\begin{enumerate}[label=\thmenumstyle,ref=\thmenumstyle]
\item $\sysE$ is a maximally Darboux integrable\footnote{This means the number of independent Darboux invariants is as large as possible; see Definition \ref{maxmindef}} Pfaffian system,
\item  $K$ acts freely on $N$, as a transverse symmetry group of $\sysE$,
\item  $\sysI$ is the quotient of $\sysE$ by the subgroup $G$, that is $\sysI =\sysE/G$ on $M=N/G$.
\end{enumerate}
\end{thm}

The proof of the necessary part of the Theorem \ref{FTDI} is a local  construction of  $\sysE$ and $(K,G)$ from $\sysI$ using the Frobenius theorem for complex vector fields. 
The sufficiency part of Theorem \ref{FTDI} allows us to easily construct Darboux integrable systems by group quotients, while the
necessary part allows us to find closed form general solutions to partial differential equations represented by $\sysI$ using only the Frobenius theorem and the integral manifolds of $\sysE$. In particular the fact that $\sysI$ is the quotient of $\sysE$ by a transverse symmetry group $G$ implies that any integral manifold of $\sysI$ is the image, under the quotient map $\bpi_G:N \to N/G$, of an integral manifold of $\sysE$  (see Theorem 2.1 and Section 2.1 in \cite{AFB}). In both cases of elliptic and hyperbolic systems we can identify families of systems $\calI$ where the integral manifolds of $\sysE$ can be determined  in closed form (due to the maximal number of Darboux invariants for $\calE$), which then produces a closed form general solution to the partial differential equation represented by $\sysI$. See the examples in  \S \ref{SecondEx} and  \S \ref{MoreEx} in this article, or the many other hyperbolic examples in \cite{AA}, \cite{AFB}, \cite{AFV} demonstrating this theorem.

The difference between hyperbolic and elliptic systems in Theorem \ref{FTDI} manifests itself in the structure of $\sysE$ and the symmetric pair $(K,G)$ which we describe in the next two propositions.
 First, Theorem 1.4 in \cite{AFV} shows the following for hyperbolic systems:

\begin{prop}\label{HDIT}  A decomposable {\bf hyperbolic} differential system
is Darboux integrable  with Pfaffian singular systems if and only if
there exists  $\sysE$ and $(K,G)$ in Theorem \ref{FTDI} where
\begin{enumerate}[label=\thmenumstyle,ref=\thmenumstyle]
\item $\sysE=\sysE_1\oplus \sysE_2$ is the direct sum of two Pfaffian systems $(\sysE_1,N_1)$, $(\sysE_2,N_2)$ on $N=N_1\times N_2$,
\item $K=G\times G$ and the symmetric pair is $(K,G_{\operatorname{diag}})$, where $G_{\operatorname{diag}}$ denotes the diagonal subgroup,
\item $G$ acts freely on $N_i$ as a transverse symmetry group of  $\sysE_i$.
\end{enumerate}
\end{prop}

By comparison, one of our main results given below states that  all  Darboux integrable elliptic decomposable systems (with Pfaffian singular systems) arise as quotients of a system $\sysE$ which is the `realification' of a holomorphic Pfaffian system $\calH$ on a complex manifold.  
(By realification we  mean simply that as a Pfaffian system $\sysE$ is generated by the real and imaginary parts of the holomorphic 1-forms generating $\calH$.)

\begin{prop} \label{MTE} A real analytic {\bf elliptic} decomposable differential system 
$\sysI$ is  Darboux integrable with Pfaffian singular system if and only if there exists  $\sysE$ and 
$(K,G)$ as in Theorem \ref{FTDI} where
\begin{enumerate}[label=\thmenumstyle,ref=\thmenumstyle]
\item $\sysE$ is the \realform \ of a holomorphic Pfaffian system $\calH$ on a complex manifold $N$, 
\item the symmetric pair is $(K,G)$ where $K$ is a complex Lie group and $G\subset K$ is a real form,
\item $K$ acts freely and holomorphically on $N$ as a transverse symmetry group of $\calH$. 
\end{enumerate}
\end{prop}
In the forward direction, the construction of $\sysE$ uses the Frobenius theorem for complex vector fields 
(stated below in Theorem \ref{CFrob})
and is only local.  
This application of the Frobenius theorem requires that $\sysI$ be real analytic; 
in the other direction, $\sysI$ is necessarily real analytic since $\calH$ and the action of $K$ are holomorphic.
The Lie algebra of $G$ which appears in Theorem \ref{HDIT} and Proposition\ref{MTE} is shown in \S\ref{TVA} to be an invariant of the system, 
and we call this the Vessiot algebra because its role is analogous to that for hyperbolic systems. 

In the hyperbolic case \cite{AFV}  the quotient map $\bpi_G$ is called the {\intem superposition map} as it combines integral manifolds of the two auxiliary systems $\sysE_1$ and $\sysE_2$ to produce integral manifolds of $\sysI$. In the elliptic case the map $\bpi_G$ is more analogous to 
the classical Weierstrass representation for minimal surfaces in $\R^3$, as it allows solutions to $\sysI$ to be expressed in terms of holomorphic data,
namely integral manifolds of the auxiliary holomorphic system $\calH$.\footnote{In fact, as pointed out by Ian Anderson, the minimal surface equation itself can be 
represented by an elliptic system which is Darboux integrable, and the Weierstrass representation can be derived from the quotient map $\bpi_G$ for that system.} 

In this paper, we split the proof of Proposition \ref{MTE} into the sufficiency and necessary parts.  
We start with a maximally Darboux integrable system and a symmetric pair, 
and prove in Theorem \ref{C34} that the quotient is a Darboux integrable elliptic decomposable  system. This provides a mechanism for constructing such DI systems which we demonstrate with examples in \S \ref{SecondEx}. Then in \S \ref{coframes} we prove the necessary part by starting with 
a DI elliptic system $\sysI$ with Pfaffian singular system, and construct $\sysE$ using the Frobenius theorem in Theorems \ref{extendexist} and \ref{loopreduction}. 
One of the key steps in establishing Theorem \ref{loopreduction} is associating to a DI elliptic system $\sysI$ a real Lie algebra $\fg$, the Lie algebra of $G$ used in constructing the extension $\sysE$.  
In \S \ref{TVA} we show that, up to Lie algebra isomorphism, the Lie algebra $\fg$ is an invariant of $\sysI$. 
This requires an alternative approach compared to the case of hyperbolic systems due to the need to track the real isomorphism class of the algebra.

The proofs of Theorem \ref{extendexist} and Theorem \ref{loopreduction} 
have a number of similarities to that of Theorem 5.1 in \cite{AFV} but also some critical differences.  
One main difference is that we prove the existence of the extension $\sysE$ more directly than was done in \cite{AFV}; 
in particular, compare Theorem \ref{extendexist} here  with  Theorem 5.1 in \cite{AFV} which utilizes the manifolds $M_1$ and $M_2$ which we have not.  
The proofs of Theorem \ref{extendexist} and Theorem \ref{loopreduction}  would provide an alternative approach  to  proving Theorem 5.1 in \cite{AFV}. 
A second important difference is the role in which real analyticity plays in the construction of $\sysE$, 
both through the application of the Frobenius theorem as well as Step 3 of Theorem \ref{fouradaptedexist}.

\medskip

While we have generalized the construction in \cite{AFV} to elliptic systems, we have not pursued the resulting applications, such
as to the existence of B\"acklund transformations as in \cite{AFB}, \cite{AFBA}, \cite{ClellandIvey}. Of particular interest, it was shown in \cite{AFCP} that the initial value problem for Darboux integrable hyperbolic PDE in the plane with solvable Vessiot algebra 
admits a closed-form solution expressed in terms of quadratures.  
This  not only generalizes d'Alembert's solution to the wave equation to other equations, but it also generalizes Lie's approach to closed-form solutions to ordinary differential equations in that it determines a class of equations where the initial value problem can be solved by quadrature (i.e., the hyperbolic Darboux integrable systems with solvable Vessiot algebras). 
It is an interesting problem to extend these results to boundary value problems for DI elliptic equations. 

We would like to thank our good friend and colleague Ian Anderson for many useful discussions and suggestions that helped initiate this project.
We are also grateful to Anderson and his co-author Brandon Ashley for providing us a copy of their preprint \cite{AA} which supplied some motivation and perspective for this article. We are  indebted to  Andreas Malmendier for his help that allowed us to finish the proof of Theorem \ref{fouradaptedexist}.
The second author is grateful to Robert Bryant for introducing him to the subject of Darboux integrability many years ago, and would also like to thank the Department of Mathematics and Statistics at Utah State University for their hospitality and support during the winter semester of 2023.

\subsection*{Layout of the Paper}

In \S\ref{Bsect} we define elliptic decomposable systems and Darboux integrability, beginning with the dual notion of elliptic distributions.   We also discuss elliptic systems generated by the spans of the real and imaginary part of a holomorphic Pfaffian system $\sysH$ on complex manifolds (the \realform \ of $\sysH$ described above), and show that they are maximally Darboux integrable. In the examples, we discuss systems arising from first-order and second-order PDE in the plane, where the elliptic structure is revealed through considerations of singular integral elements.

In \S\ref{reduction} we consider actions of a complex Lie group $K$ on a complex manifold that preserve a  holomorphic Pfaffian system.  We show, in Theorem \ref{C34}, that under the assumptions of regularity and transversality, the quotient of the realification of $\sysH$ by the action of the real form of $K$ is elliptic and Darboux integrable.   This provides the sufficiency proof of Theorem \S\ref{MTE}. We also show how to construct a local coframe for the quotient system  which leads to a normal form for the structure equations.  Examples are given at the end of the section demonstrating the theory.

In \S\ref{coframes} we begin with showing that every DI elliptic decomposable system $\sysI$ admits a local coframe of precisely the same normal form as given in \S\ref{reduction}. The construction of this coframe is, up to a point, analogous to what happens in \cite{AFV}, but here we also use analytic continuation in an essential way.  In \S\ref{CONSIE} and \S\ref{ConvTheorem} we finish the proof of the local necessary part of Theorem \ref{MTE} by constructing the integrable extension $\sysE$  on an appropriate complex manifold along  with and a group action $K$ so that $\sysI =\sysE/G$, with $G$ a real form of $K$. Again, examples are given demonstrating this construction.

In remarks at the end of this article, we briefly discuss the classic theory of Darboux integrability and its relationship to closed form solutions. 

\subsection*{EDS Conventions}
We will study PDE and PDE systems by working with their generalization as exterior differential systems. 
Recall that an {\defem exterior differential system} (EDS) on a smooth manifold $M$ is a graded ideal $\sysI$ 
(with respect to wedge product) inside the ring of differential forms on $M$, 
such that $\sysI$ is also closed under exterior differentiation.  
We assume systems do not include 0-forms, and we assume that for $1\le k \le \dim M$ 
the $k$-forms of $\sysI$ span a vector sub-bundle $I_k$ of constant rank inside $\Lambda^k T^*M$.  
The systems we consider will, in fact, be generated algebraically by 1-forms and 2-forms, i.e., sections of $I_1$ and $I_2$.
Most of these will be {\defem Pfaffian systems}, which are generated by sections of a given sub-bundle $I_1 \subset T^*M$ and their exterior derivatives.  

For a Pfaffian EDS $\sysI$ or $\sysE$
we will by convention use the corresponding roman letter $I$ or $E$ respectively to denote the bundle of 1-forms which generates it in this way, omitting the subscript $1$ when no confusion is possible; sometimes $I$ or $E$ will itself be referred to as a Pfaffian system. 
We also let  $I^\perp \subset TM$ denote the annihilator sub-bundle of $I \subset T^*M$; 
similarly, given a distribution $\DD\subset TM$ then $\DD^\perp \subset T^*M$ denotes its annihilator.  

For a Pfaffian system generated by $I \subset T^*M$, the {\defem first derived system} $I'$ is spanned by sections $\theta$ of $I$ which satisfy $d\theta \in C^\infty( I \wedge T^*M)$.  Applying this construction repeatedly
generates the {\defem derived flag} $I \supset I^{(1)} \supset \ldots$ where $I^{(1)}=I'$.  This sequence (the members of which we will always assume to be constant rank sub-bundles of $T^*M$) stabilizes at a 
terminal derived system $I^{(\infty)}$ which may have rank zero.  Any local section of $I$ that is a closed 1-form lies in $I^{(\infty)}$, and by the Frobenius Theorem $I^{(\infty)}$ is locally generated by exact 1-forms. 
We apply these standard constructions to the complexification $\sysI \otimes \C$ and utilize the  version of the Frobenius theorem for complex vector fields found in \cite{StormarkFrob} and stated in Theorem \ref{CFrob} below.  
This version of the  Frobenius theorem requires real analyticity and therefore
we will assume all manifolds and differential systems are {\it real analytic}.

An immersion $s:\Sigma \to M$ is said to define an {\defem integral submanifold} of $\sysI$ if $s^* \varphi = 0$ for all differential forms $\varphi \in \sysI$.  
When $\sysI$ is coupled with a differential $n$-form $\Omega$ giving an {\defem independence condition}, 
it is customary to consider only integral submanifolds which 
satisfy the independence condition, in the sense that $s^* \Omega$ is a volume form on $\Sigma$. 

As is well-known (see, e.g., \cite{BCG3} or \cite{CfB2}) a system $\mathcal R$ of $p$ partial differential equations of order $k$ for $s$ functions of $n$ variables defines a codimension-$p$ submanifold $M$ inside the space $J^k(\R^n, \R^s)$ of $k$-jets of 
functions from $\R^n$ to $\R^s$, and solutions to system $\mathcal R$ are in one-to-one correspondence with
integral submanifolds of a Pfaffian system $\sysI$ on $M$ generated by the pullbacks to $M$ of the canonical contact forms
on the jet space.  (One usually also requires the integral submanifolds to satisfy an independence condition
that implies they are graphs over $\R^n$.)  Examples \ref{ex:Laplace}, \ref{introE1}, \ref{benexample1} and \ref{BHE} below
show how the system $\sysI$ is constructed in several cases.

\section{Preliminaries}\label{Bsect}

The definition of decomposable elliptic systems, which are the differential systems we consider in this paper,  
is analogous to that of a decomposable hyperbolic differential systems formulated in \cite{AFV}. 
This definition is also based on the algebraic structure of the Pfaffian systems occurring in second-order elliptic PDE in the plane, discussed in Example \ref{introE1} of \S \ref{basicexamplesec}.
In this section we will define an elliptic decomposable exterior differential system on a manifold $M$ to be 
a system $\sysI$ equipped with a bundle $V \subset T^*M \otimes \C$ such that $V \cap \overline{V}$ spans the 1-forms of $\sysI \otimes \C$ 
and the 2-form generators of $\sysI \otimes \C$ may be chosen to lie in $\Lambda^2 V$ and in $\Lambda^2 \overline{V}$; 
further technical requirements are given in Definition \ref{EDdef}. 
The bundle $V$ is called the singular bundle due to its relationship with singular integral elements of $\sysI\otimes \C$%
\ifarch, as described Corollary \ref{Dsingular} and Example \ref{introE1}. \else; see Remark \ref{TS10}. \fi  
Our definition implies that the real distribution $\dD$ annihilated by the 1-forms of $\sysI$ is even-dimensional and carries an `almost sub-complex structure', i.e., an endomorphism $\JJ$ satisfying $\JJ^2 = -\Id$ (see \S\ref{Bestone} for details). 

The definition of Darboux integrability given in \ref{EDIdef} is a straightforward analogue of the notion of Darboux integrability for hyperbolic systems used in \cite{AFV}.  
Essentially, Darboux integrability means that the singular bundle $V$ has sufficiently many independent first integrals,  i.e., complex-valued functions on $M$ whose differentials are sections of $V$.  
We refer to these functions as holomorphic Darboux invariants; the justification for this term is that, when $\sysI$ is Darboux-integrable $\JJ$ restricts to be a genuine complex structure on any integral manifold whose tangent spaces are $\JJ$-invariant, and the Darboux invariants restrict to be $\JJ$-holomorphic.  
Again, as in the hyperbolic case \cite{AFV}, we  associate to a decomposable Darboux-integrable elliptic system the singular differential system $\calV$ in Definition \ref{SSdef}. 
We limit ourselves  to DI systems $\sysI$ where this singular system is Pfaffian, as was done in \cite{AFV}, and which is always satisfied when $\sysI$ is itself a Pfaffian system. 
The case where the number of independent Darboux invariants is maximal plays a key role.  
In \S \ref{SmaxDI} a maximal DI system is shown to be equivalent to a holomorphic system on a complex manifold, 
and the invariants of such systems used in the proof of Theorem \ref{MTE} are then determined.
Examples which demonstrate the theory are presented in \S\ref{basicexamplesec}.

\subsection{Elliptic Distributions and Elliptic Decomposable Systems}\label{Bestone}
We begin by giving the elliptic analogue of the definition of a hyperbolic distribution from \cite{AA}. 

\begin{defn}\label{ED} An {\em elliptic distribution} on a manifold $M$ of dimension $\dimM$ is a pair $(\dD, \JJ)$ where 
\begin{enumerate}[label=\defenumstyle,ref=\defenumstyle]
\item\label{Ddefeven} $\dD\subseteq TM$ is a smooth distribution of even rank $\rankD <\dimM$, 
\item\label{Ddefchow} $\dD$ is bracket-generating, i.e., $\dD^{(\infty)}= TM$,
\item\label{DdefJ} $\JJ$ is a sub-complex structure on $\dD$ (i.e., $\JJ: \dD \to \dD$ is a bundle isomorphism satisfying $\JJ^2 = -\Id$) ,
\item\label{Ddefeigens} when $\JJ$ is extended $\C$-linearly to $\dD \otimes \C$, the terms of the decomposition 
\[
\dD \otimes \C = \Dplus\oplus \Dminus
\]
as a sum of $+\ri$ and $-\ri$ eigenspaces of $\JJ$ respectively, satisfy
\begin{equation}
[\Dplus, \Dminus] \subseteq \dD\otimes \C.
\label{DDB}
\end{equation}

\end{enumerate}
\end{defn}
\def\JJSig{\JJ_{\scriptscriptstyle\Sigma}}

The non-degeneracy of $\JJ$ on the elliptic distribution $\DD$ implies that 
if $s:\Sigma \to M$ is an integral manifold of $\DD$ which is $\JJ$ invariant, so that  $\JJ$ preserves $s_* T_x \Sigma$, 
then  $\JJ$ restricts to give an almost complex structure  ${\JJSig}$ on $\Sigma$.


\begin{remark}\label{JfromDplus}
The definition of $\Dplus, \Dminus$ as eigenspaces of $\JJ$
implies that $\Dminus = \overline{\Dplus}$.  Conversely, a decomposition of
 $\dD \otimes \C$ as an (ordered) direct sum $\Dplus \oplus \Dminus$ of complex subspaces which are conjugates of each other uniquely determines a sub-complex structure on $\dD$ with these eigenspaces
\ifjour
\else.  
To see why, suppose $\vW_a$ is a local  basis of sections of $\Dplus$, and define a complex endomorphism on $\dD \otimes \C$ by $\JJ \vW_a = \ri \vW_a$ and
$\JJ \overline{\vW_a}  = -\ri \overline{\vW_a}$.  Then $\JJ^2 = -\Id$, and $\JJ$ restricts to be an endomorphism of $\dD$, since
\begin{equation}
\JJ( \vW_a+\overline{\vW_a}) = \ri (\vW_a -\overline{\vW_a}) 
\qquad \JJ\left(\ri( \vW_a-\overline{\vW_a})\right)  = -(\vW_a + \overline{\vW_a}).
\label{IJmap}
\end{equation}
(In this section we use index ranges $1 \le a,b \le \rankDplus$ and $1 \le i,j \le \dimM - \rankD$.)
Moreover, this $\JJ$ is independent of the choice of basis.
\fi
\end{remark}

It is easy to see that condition \eqref{DDB} implies the following:
\begin{lemma}\label{CDER} For an elliptic distribution $\dD$, the derived systems of $\dD, \Dplus, \Dminus$ satisfy
$$
\dD^{(k)} \otimes \C=(\dD \otimes \C)^{(k)} = \Dplus^{(k)}+ \Dminus^{(k)}
$$
\end{lemma}

Many of our later results will be derived using local coframes adapted to elliptic distributions.  
The following lemma gives a first version of the structure equations satisfied by such coframes.

\newcommand\npi{{\sigma}}
\newcommand\npibar{\overline{\sigma}}
\newcommand{\npibars}[1]{\overline{\sigma^{#1}}}
\begin{lemma}\label{BasicCoframe}
Let $(\dD,\JJ)$ be an elliptic distribution on $M$.  Near any point of $M$ there is a coframe $\{ \theta^i, \npi^a, \npibars{a} \}$ consisting of complex-valued 1-forms such that 
$I=\dD^\perp$ is spanned by the $\theta^i$, $\Dplus^\perp$ is spanned by the $\theta^i$ and $\npibars{a}$, and the $\theta^i$ satisfy structure equations 
\begin{equation}
d\theta^i \equiv \tfrac12 A^i_{ab} \npi^a \wedge \npi^b + \tfrac12 B^i_{ab} \npibars{a} \wedge \npibars{b} \qquad \mod \ \theta^1\ldots \theta^{\dimM - \rankD}.
\label{PED}
\end{equation}
In addition, on the domain of this coframe there is an invertible matrix $P$, whose entries are complex-valued functions, such that $\theta^i = P^i_j \overline {\theta^j}$.  It follows that
$P^{-1}=\overline{P}$ and 
\begin{equation}\label{PED2}
B^i_{ab}=P^i_j \overline {A^j_{ab}}.
\end{equation}
\end{lemma}
\begin{proof}Choose a local frame  $\{\vT_i, \vW_a, \overline{\vW_a} \}$  for $TM\otimes \C$ such that  $ \Dplus=\{ \vW_a \} $ and $\Dminus=\{ \overline{\vW_a} \}$, and let $\{ \theta^i, \npi^a, \npibar^{a} \}$ be the  dual local coframe for $T^*M\otimes \C$.  
(Thus, the $\theta^i$ form a local basis of sections of $I \otimes \C$.)
Let $\vX$ and $\vY$ be arbitrary local sections of $\Dplus$ and $\Dminus$ respectively.  Then  \eqref{DDB} 
implies that
$$
d \theta^i( \vX, \vY) = -\theta^i([\vX,\vY]) = 0.
$$
It follows that the expression for $d \theta^i$ in terms of our coframe contains no $\npi^a \wedge \npibars{b}$ terms, so has the form \eqref{PED}.

Since $I\otimes \C$ is closed under conjugation, the $\thetabars{i}$ also form a local basis of sections of $I \otimes \C$.  Thus, 
$\theta^i = P^i_j \overline{ \theta^j }$ where $P^i_j$ are the components of an invertible matrix-valued function. 
Substituting this equation into its complex conjugate then gives $P^{-1}=\overline{P}$.
Taking the exterior derivative of $\theta^i = P^i_j \overline{ \theta^j }$, using the conjugate of equation \eqref{PED}, and comparing the left-hand side with that of \eqref{PED} gives $B^i_{ab} = P^i_j \overline{ A^j_{ab}}$ .
\end{proof}


\ifarch
Condition \eqref{DDB} implies that the subspaces $\Dplus$ and $\Dminus$ 
are related to the singular integral elements of the Pfaffian system generated by the 1-forms annihilating $\dD$ as follows.

\begin{cor}\label{Dsingular} Let $(\dD,\JJ)$ and $\sysI$ be an elliptic distribution on $M$ and $\sysI$ be the Pfaffian system generated by $I=\dD^\perp$. Then any one-dimensional complex subspace of $\Dplus$ or of $\Dminus$ is a singular integral element of $\sysI \otimes \C$.
\end{cor}
\begin{proof} Let $\{ \theta^i, \npi^a, \npibar^{a} \}$ be the local coframe given by Lemma \ref{BasicCoframe}.

Since $\dD$ is bracket-generating, there exist local sections $\vX_1, \vX_2$ of $\Dplus$ with $[\vX_1,\vX_2]\not \in \dD \otimes \C$.  
Thus there exists an index $i_0$ such that $d\theta^{i_0}(\vX_1,\vX_2) = - \theta^{i_0}([\vX_1,\vX_2])\neq 0$, and therefore the 2-form $A^{i_0}_{ab} \npi^a \wedge \npi^b \neq 0$.  
Similarly, since $[\overline{\vX_1},\overline{\vX_2}] \notin \dD\otimes \C$ then $B^{i_0}_{ab} \npibars{a} \wedge \npibars{b}\neq 0$.

Singular integral elements are those integral elements for which the polar equations associated to the integral element have lower rank when compared to those associated to nearby integral elements. 
(The `polar equations' refer to a basis for the annihilator of the polar space of the integral element; see Chapter 8 in \cite{CfB2} for further details.)
If $E$ is the one-dimensional integral element for $\sysI \otimes \C$ spanned by $X^+= \xi^a \vW_a\restr_\mpt \in \Dplus\restr_\mpt$, 
then the polar space associated to $E$ is annihilated by $\theta^i$ and by 
$$X^+ \hook d\theta^i = A^i_{ab}  \xi^a  \npi^b$$
for each $i$.  
Because  $B^{i_0}_{ab} \npibars{a} \wedge \npibars{b}\neq 0$ these equations are not maximal rank. In particular, an integral element spanned by a generic $X \in \dDC$ has polar equations
which include nonzero terms involving the $\npibars{b}$. Therefore, $X^+\in \Dplus$ spans a singular integral element, and similarly any nonzero $X^-\in \Dminus$ spans a singular integral element.
\end{proof}
\fi

\ifarch
The following converse to Corollary~\ref{Dsingular} 
characterizes ellipticity in terms of Pfaffian systems and their singular integral elements. (See Lemma 2.1 and Theorem 2.3 in~\cite{AFV} for a similar characterization for hyperbolic systems.) 

\begin{prop}  \label{PS1}
Let $\sysI$ be a Pfaffian system generated by $I\subset T^*M$ and let $\dD =I^\perp$.  Suppose that 
\begin{enumerate}[label=\thmenumstyle,ref=\thmenumstyle]
\item $I^{(\infty)}=0$;
\item\label{PS1decomp}
there are complex sub-bundles $\Dplus, \Dminus \subset TM \otimes \C$ such that $\Dminus = \overline{\Dplus}$ and  $\dD \otimes \C = \Dplus\oplus \Dminus$;

\item\label{PS1integral} for any $X^+\in \Dplus$ and $X^- \in \Dminus$,  
$\spanc \{\, X^+,\, X^-\, \}$ is an integral 2-plane of $\calI\otimes \C$.
\end{enumerate}
Then $(\dD,\JJ)$ is an elliptic distribution on $M$, where $\JJ :\dD \to \dD$ is determined by the 
decomposition $\dD \otimes \C = \Dplus\oplus \Dminus$ as in Remark \ref{JfromDplus}.
\end{prop}

%



\begin{proof} 
The assumption $I^{(\infty)}=0$ is equivalent to $\dD$ being bracket-generating, thus satisfying
condition \ref{Ddefchow} in Definition \ref{ED}.
Assumption \ref{PS1decomp} implies that $\rk_\R \dD = 2 \rk_\C \Dplus$, so condition \ref{Ddefeven} is satisfied. 
Assumption \ref{PS1integral} implies that for any 1-form $\theta$ in $\sysI$ and any 
sections $\vX^+, \vX^-$ of $\Dplus$ and $\Dminus$ respectively,
$$
\theta([\vX^+,\vX^-]) = -d\theta(\vX^+,\vX^-)  = 0.$$
Therefore $[\vX^+,\vX^-] \in  \dD\otimes \C$ and equation \eqref{DDB} in condition \ref{Ddefeigens} is satisfied.  Lastly, it is easy to check that the $\JJ: \dD\to \dD$  defined by equation \eqref{IJmap} is a sub-complex structure, so all conditions in Definition \ref{ED} are satisfied. 
\end{proof}

\begin{remark} \label{PfaffED} In order to determine if a Pfaffian system $\sysI$ satisfies the conditions in Proposition \ref{PS1}, it is necessary to identify the 1-dimensional complex integral elements that are singular.  For, Corollary \ref{Dsingular} shows that every non-zero vector in $\Dplus$ or $\Dminus$ is in fact a singular integral element.   Determining appropriate candidate spaces $\Dplus,\Dminus$ is illustrated in Example \ref{benexample1} of  \S\ref{basicexamplesec} below.
\end{remark}

By reformulating the hypotheses of Proposition \ref{PS1} in terms of a local coframe, we obtain the following.

\begin{cor}\label{CF1}  Let $\calI$ be a Pfaffian system such that $I^{(\infty)}=0$.
Suppose that on $\scrU \subset M$ there exists a coframe $\{\, \theta^i, \, \npi^a, \npibars{a} \, \}$ for   $TM\otimes \C$ such that $I \otimes \C = \spanc \{ \theta^i\}$ 
and the $\theta^i$ satisfy \eqref{PED}.  
Let $\dD =I^\perp$, $\Dplus =\{\ \theta^i\ , \npibars{b}\ \}^\perp$, $\Dminus =\{\ \theta^i\ , \npi^a \ \}^\perp$, and 
define a sub-complex structure $\JJ$ on $\DD$ associated to the splitting $\dD \otimes \C = \Dplus \oplus \Dminus$ as in Remark \ref{JfromDplus}.  
Then $(\DD, \JJ)$ is an elliptic distribution on $\scrU$.
\end{cor}
\fi

\medskip



Given an elliptic distribution $(\dD,\JJ)$ we define its associated {\defem singular bundles}, which are complex-conjugate sub-bundles of $T^*M \otimes \C$ dual to the splitting of $\dD \otimes \C$:
\begin{equation}\label{DEFV}
\rV = (\Dminus )^\perp  \subset T^*M \otimes \C, \qquad \rVb=  (\Dplus)^\perp \subset T^*M \otimes \C.
\end{equation}
(For simplicity, we will refer to $V$ as {\em the} singular bundle, since it determines $\rVb$.)
These bundles have complex rank $\dimM - \rankDplus$, and satisfy $\rV \cap \rVb =  I \otimes \C$ where $I = \dD^\perp$. 

\begin{remark}\label{TS10}
Note that $\rV, \rVb$ determine, and are determined by, the sub-complex structure on $\dD$.  In particular, the bundle $\rV$ is spanned by 1-forms that restrict to $\dD$ to be of type $(1,0)$:
\[
\rV = \{ \psi \in T^*M \otimes \C \mid \psi(\JJ \vv) = \ri \psi(\vv), \forall \ \vv \in \dD \}.
\]
Note also that in terms of the local coframe in Lemma \ref{BasicCoframe}, we have $\rV=\spanc \{\theta^i, \npi^a\}$. 
Consequently, for a $\JJ$-invariant integral manifold $s:\Sigma \to M$ of $\DD$ with induced almost complex structure~${\JJSig}$,
$s^*(\rV)\subset T^*\Sigma \otimes \C$ is the bundle of $(1,0)$-forms on $\Sigma$ relative to ${\JJSig}$.
Then $T^*_{1,0} \Sigma$ will equal the pointwise span of the 1-forms $s^* \npi^a$, since $s^* \theta^i=0$. 
\end{remark}

\smallskip
The exterior differential systems we will consider in the rest of this paper include Pfaffian systems whose dual distribution is elliptic, but which satisfy the following additional assumptions:

\begin{defn}\label{EDdef}
A triple $(\sysI,\dD,\JJ)$ is an {\defem elliptic decomposable system} on $M$ if
\begin{enumerate}[label=\defenumstyle,ref=\defenumstyle]

\item\label{EDdefsys} $\sysI$ is an exterior differential system on $M$ generated algebraically by 1-forms and 2-forms;
\item\label{EDdefJ} $\dD = (I_1)^\perp$ and $\JJ$ is a sub-complex structure on $\dD$ such that 
$(\dD, \JJ)$ is elliptic in the sense of Defn. \ref{ED} above;
\item\label{EDdefV} on a neighborhood of any point in $M$, there exist a set of 2-forms $\Omega^\nu$ which are sections of $\Lambda^2 \rV$ such that $\sysI \otimes \C$ is algebraically generated by the $\Omega^\nu$, their complex conjugates $\overline{\Omega^\nu}$, and sections of $I_1$,
where $\rV, \rVb$ denote the singular bundles associated to $(\dD,\JJ)$.
\end{enumerate}
\end{defn}
 
Definition \ref{EDdef} is analogous to the notion of a (hyperbolic) {\defem decomposable system} given in~\cite{AFV}.
Note that condition \ref{EDdefV} is independent of the others.  For example, if $(\dD, \JJ)$ is an elliptic distribution and $\sysI$ is the Pfaffian
system generated by $\dD^\perp$, then \eqref{PED} implies that its generator 2-forms are sections of $\Lambda^2 \rV + \Lambda^2 \rVb$, whereas
\ref{EDdefV} implies that each 2-form $A^i_{ab} \npi^a \wedge \npi^b$ belongs to $\sysI$.
More generally, we note that for Pfaffian systems the decomposability condition \ref{EDdefV} is equivalent to the coefficients in structure equations \eqref{PED} satisfying certain linear-algebraic conditions; see Theorem 2.3 in~\cite{AFV} for details in the hyperbolic case.

However, in certain important low-dimensional cases, condition \ref{EDdefV} follows from \ref{EDdefJ}.
When $\sysI$ is a Pfaffian system arising from a PDE in the plane (i.e., $\sysI$ encodes a second-order PDE for one unknown function of two variables, or is a prolongation of such a system) the ellipticity of the PDE implies that there are complex-conjugate singular bundles such that the structure
equations satisfy condition \ref{EDdefV}; then $\dD = (I_1)^\perp$ becomes an elliptic distribution once one designates one of these as the annihilator of $\Dminus$.  
The details of the construction are given in Example \ref{introE1}. 
Likewise, for determined first-order systems for two functions in the plane, ellipticity of the symbol relations implies that the corresponding Pfaffian system is an elliptic decomposable EDS (see, e.g., the analysis in \S7.1 of \cite{BCG3}).

\subsection{Darboux Integrability}

\begin{defn}\label{DarbouxInv} Given an elliptic decomposable system $(\sysI,\dD,\JJ)$ with singular bundle $\rV$, a function $f: M \to \C$ such that $df$ is a section of $\rV^{(\infty)}$ is called a {\defem Darboux invariant}. We will let $\Invts$ denote the set of Darboux invariants.

\end{defn}

As noted in Remark \ref{TS10},  for $\JJ$-invariant integral submanifolds of $\dD$,  $s^*\rV$ defines the bundle of $(1,0)$-forms.
If the almost complex structure ${\JJSig}$ is integrable, then the restriction of any Darboux invariant to such submanifolds is a holomorphic function.  Accordingly, we will sometimes refer to functions in $\Invts$ as {\defem holomorphic} Darboux invariants. 
Since $\Dminus^{(\infty)}=(\Vinf)^\perp$ then $f:M \to \C$ is a Darboux invariant if and only if $\vX(f)=0$ 
for every section $\vX$ of  $\Dminus^{(\infty)}$.  However, since $\Dminus^{(\infty)}$ is generated by brackets of 
vector fields in $\Dminus$, this condition can be simplified as follows (see Lemma 2.8 in~\cite{AFV} for the hyperbolic case).

\begin{lemma}\label{DIX}  A function $f: M \to \C$  is a {\defem Darboux invariant} if and only if
 $\vX(f)=0$ for every section $\vX$ of $\Dminus$.
\end{lemma}
Classically, the existence of a sufficient number of Darboux invariants allows one to construct integral manifolds 
using techniques from ordinary differential equations; see~\cite{AFV} or Chapter 7 in~\cite{CfB2} for the hyperbolic case. 
This leads to the following definition of Darboux integrability.

\begin{defn}\label{EDIdef}
Let $M$ be a real-analytic manifold and $(\dD,\JJ)$ an real-analytic elliptic distribution on $M$. 
Then $(\dD,\JJ)$ is {\defem Darboux integrable} if its associated singular bundle satisfies
\begin{equation}\label{VDIcond}
\Vinf + \rVb=T^*M \otimes \C, 
\end{equation}
where the sum need not be direct. Equivalently, by taking the annihilator of \eqref{VDIcond},  $(\dD,\JJ)$ is Darboux integrable if 
\begin{equation}\label{DDIcond}
\Dplus^{(\infty)} \cap \Dminus = 0.
\end{equation}
We similarly define an elliptic decomposable system $(\sysI,\dD,\JJ)$ on $M$ to be Darboux integrable if $\sysI$ is a real-analytic EDS and 
$(\dD,\JJ)$ is a Darboux integrable elliptic distribution.  
\end{defn}


Analyticity is required in Definition \ref{EDIdef} in order to guarantee we can find a local basis for the Darboux invariants. 
In particular, a local basis is provided by the next lemma, which relies on the Frobenius Theorem for complex vector fields, 
which in turn requires the vector fields to be real analytic. 

For what follows, it will be convenient to introduce the following abbreviations for the dimensions and ranks of the objects we are considering:
\begin{equation}\label{rankdims}
\dim M = \dimM, \qquad \rank_{\R} \DD = \rankD, \qquad \rank_{\C} \Vinf =\rankVinf, \qquad n=m-2q.
\end{equation}

\begin{lemma} \label{Vinfdz} Let $(\DD, \JJ)$ be Darboux integrable.  Near any point in $M$ there exist real local coordinates $(x^a, y^a, t^i)$, where $1\le a \le q$ and $1\le i \le n$, such that
$$
\Vinf=\{ dz^a\}
$$
where $z^a =x^a+\ri y^a$.
\end{lemma}
We call coordinates $(t^i, z^a,\overline {z^a})$ {\defem adapted coordinates}; 
these have the property that the $z^a$ give a complete set of independent holomorphic Darboux invariants.
\ifjour
\else 
(attributed to Nirenberg \cite{StormarkFrob}), which we now state.

\begin{thm}[Complex Frobenius Theorem]\label{CFrob}   
Let $M$ be a real-analytic manifold of dimension $\dimM$.
Let $\dS \subset TM \otimes \C$ be a real-analytic complex sub-bundle of 
rank $s$ which is closed under Lie bracket.
Let $\widetilde{\dS} = \dS + \overline{\dS}$.  Assume that $\widetilde{\dS}$ has 
rank $s+\ell$ and is also
closed under Lie bracket.  Then near any point in $M$ there is a real-analytic coordinate system 
$(x^a, y^a, t^i, u^\alpha)$, where
$1 \le a \le \ell$, $1 \le i \le s-\ell$ and $1 \le \alpha \le \dimM-(s+\ell)$, such that $\dS$ is generated by
$$\dib{x^a} + \ri \dib{y^a} \quad \text{and} \quad \dib{t^i}.$$
\end{thm}

\begin{proof}[Proof of Lemma \ref{Vinfdz}]
The Frobenius Theorem is applied as follows. By construction, $\Dminusinf=(\Vinf)^\perp$ is closed under Lie bracket and  $\rk  \Dminusinf= \dimM - \rankVinf \ge \rankDplus$ because $\rk  \Vinf = \rankVinf $.  Since $\dD$ is bracket-generating, by Lemma \ref{CDER} we have
$$T^*M \otimes \C = \Dminusinf+\Dplusinf.$$ Hence $\dS = \Dminusinf$ satisfies the hypotheses of the Theorem \ref{CFrob} with $s=\dimM - \rankVinf$ and $\ell=\rankVinf$
 (i.e., no $u$-coordinates).  Thus, on a neighborhood of any given point in $M$ there are real coordinates $(x^a, y^a, t^i)$,  where $1 \le a \le \rankVinf$ and $1\le i \le  \dimG=\dimM - 2\rankVinf$, such that
if $z^a = x^a + \ri y^a$ then the $dz^a$ are a local basis for sections of $\Vinf$. 
\end{proof}
\fi

\bigskip
The Darboux integrability condition \eqref{VDIcond} leads to bounds on the number of independent Darboux invariants, given by the following proposition.

\begin{prop}\label{NumDI}  Let $(\sysI, \dD,\JJ)$ be a Darboux integrable elliptic decomposable system and let $\rankDplus, \rankVinf, \dimM$ be as in \eqref{rankdims}. 
Then
\begin{equation}
\rankDplus \leq \rankVinf \leq \lfloor  \dimM/2 \rfloor.
\label{rankqinfo}
\end{equation}
\end{prop}
\begin{proof} Since $\rank \Dplus=d$ then $\rank  \rVb=m-d$ and \eqref{VDIcond} implies that $q \geq d$, giving the lower bound in \eqref{rankqinfo}. 
Equation \eqref{VDIcond} also implies that $m = \rank \Vinf+ \rank \rVb -\rank (\rVb \cap \Vinf)$, and so
\begin{equation}
\rank(\rVb \cap \Vinf) = \rank \Vinf+ \rank \rVb -m = q -d.
\label{numeta}
\end{equation}
By Lemma \ref{Vinfdz} there are local coordinates
such that $\Vinf=\{ dz^a\} $, $1\le a \le q$.  This implies that
$\Vinf \cap \overline{\Vinf} = 0$ and hence $ (\rVb \cap \Vinf) \cap ( \rV \cap \rVb^{(\infty)} )=0$.  Since 
\begin{equation}\label{containment}
(\rVb \cap \Vinf )\oplus(\rV \cap \rVb^{(\infty)} )\subset I\otimes \C
\end{equation}
then $2\, \rank (\rVb \cap \Vinf ) \leq \rank I = m-2d$.
Substituting this into equation \eqref{numeta} and simplifying leads to the upper bound in \eqref{rankqinfo}.
\end{proof}


\begin{defn}\label{maxmindef}
We will say that a Darboux integrable system 
\ifarch is {\defem minimal} if $\rank \Vinf = \rankDplus$, its minimal value in \eqref{rankqinfo}.
(In other words, $\Vinf$ is no larger than necessary to fulfill condition \eqref{VDIcond}, and the sum in \eqref{VDIcond} is direct.)
On the other hand, we will say the system 
\fi
is {\defem maximal} if $m=\dim M$ is even and $\rank \Vinf$ attains its maximum value $m/2$ in \eqref{rankqinfo}.
Maximally Darboux integrable systems are characterized in the next section.
\end{defn}


We now find a local coframe for $T^*M \otimes \C$ adapted to the condition in equation \eqref{VDIcond}. 

\begin{thm}\label{oneadaptedexist}  Let $(\DD, \JJ)$ be  Darboux integrable elliptic distribution on a manifold $M$ of dimension $m$ and let $I=\DD^\perp$.
Let $\rank I=m- 2d$ and let $q = \rank \Vinf$.  Then near any point in $\mpt \in M$ there exists an open set $\Udown\subset M$ about $\mpt$ and a complex coframe $(\theta^i, \eta^r, \sigma^u, \etabars{r}, \sigmabars{u})$ on $\Udown$ 
 such that 
\begin{enumerate}[label=\thmenumstyle,ref=\thmenumstyle]
\item 
$\Vinf\vert_\Udown = \spanc \{ \eta^r, \sigma^u\}$, where $1 \le r \le \rankVinf - \rankDplus, \ 1\le u \le  \rankDplus$;

\item
$(\IC \otimes \C)\vert_\Udown= \spanc \{\theta^i,\eta^r,  \etabars{r}\}$, where $1 \leq i \leq n=m-2q$;

\item the coframe satisfies structure equations
\begin{subequations}\label{onedapt}
\begin{align}
d\sigma^u&=0, \label{onedsigma} \\
d\eta^r &=\tfrac12 E^r_{u v} \sigma^u \wedge \sigma^v +
F^r_{u s} \sigma^u \wedge \eta^s,\label{onedeta} \\
d\theta^i &\equiv \tfrac12 A^i_{ab} \pi^a \wedge \pi^b + \tfrac12 P^i_j \overline{A^j_{ab}} \pibars{a}\wedge \pibars{b} \qquad 
\mod \ \{\theta^i\}, \  \label{onedtheta}
\end{align}
\end{subequations}
where in \eqref{onedtheta} we amalgamate $(\eta^r, \sigma^u)$ into the vector $\vpi= (\pi^a)$, $1\le a \le \rankVinf$, and $P^i_j$ are entries in a nonsingular matrix of functions such that 
\begin{equation}\label{Pmoddef}
\theta^i\equiv P^i_j \overline{\theta^j} \ \mod \eta^r,\etabars{r};
\end{equation}
\item the coefficients $A^i_{ab}, E^r_{uv}, F^r_{u s}$ are Darboux invariants.
\end{enumerate}
\end{thm}

The coframe $(\theta^i, \eta^r, \sigma^u, \etabars{r}, \sigmabars{u})$ in Theorem \ref{oneadaptedexist} is the analogue of the `1-adapted' coframe in Theorem 2.9 in \cite{AFV} (compare \eqref{dtheta1} with (2.14) in \cite{AFV}).  We will likewise call this coframe 1-adapted. 

%

\begin{proof}
Let $(x^a, y^a, t^i)$ be adapted coordinates near $\mpt$.  Note that
by equation \eqref{numeta} the rank of $(\IC \otimes \C) \cap \Vinf$ equals $\rankVinf - \rankDplus$.
We may choose 1-forms $\eta^r$, $1 \le r \le \rankVinf - \rankDplus$, forming a basis of sections of 
$(\IC \otimes \C) \cap \Vinf$, and complete them with 1-forms $\sigma^u$, $1\leq u \leq d$, to form a basis of sections of $\Vinf$, so that
(after suitably re-numbering the $z^a$)
\begin{align}
\sigma^u &=dz^{u},  & & 1\le u \le  \rankDplus  \label{zucoords} \\
\eta^r &=dz^{r+\rankDplus} + R^r_u dz^{u}, & &1 \le r \le \rankVinf - \rankDplus.  \notag
\end{align}
Because $d\eta^r$ contains no terms of the form $d\zbars{b} \wedge dz^a$ or $dt^i \wedge dz^a$, then $dR^r_u \in V$, and hence $R^r_u \in \Invts$.  As argued in the proof of Prop. 10.2.7 in \cite{CfB2}, the 1-forms $\eta^r$ then satisfy structure equation \eqref{onedeta} and moreover $E^r_{uv}$ and $F^r_{us}$ also belong to $\Invts$. 

We may now choose 1-forms $\theta^i$, $1\le i \le n= m-2q$, defined on a possibly a smaller neighborhood of $\mpt$, such that $(\theta^i, \eta^r, \etabars{r})$ form a basis of sections for $\IC \otimes \C$. 
Since $\IC\otimes \C$ is preserved by complex conjugation, a nonsingular matrix $P$ exists satisfying \eqref{Pmoddef}.
Moreover, since $\IC \otimes \C = \rV \cap \rVb$ and $\Vinf \subset \rV$, then $(\theta^i, \eta^r, \sigma^u, \etabars{r})$ is
a local basis of sections for $V$.

Let $(\vT_i, \vN_r, \vS_u, \overline{\vN_r}, \overline{\vS_u})$ be the dual frame to $(\theta^i, \eta^r, \sigma^u, \etabars{r}, \sigmabars{u})$. 
Note that $\Dplus$ is spanned at each point by the vector fields $\vS_u$, and $\Dplusinf$ by the vector fields $\vT_i, \vN_r, \vS_u$. The forms $\theta^i$ satisfy $\theta^i([\vS_u, \overline{\vS_v}]) = 0$ on account of \eqref{DDB},  so that
the structure equations for $\theta^i$ have the form
\begin{equation}
d \theta^i \equiv \tfrac12 A^i_{ab} \pi^a \wedge \pi^b + \tfrac12 B^i_{ab} \overline{\pi^a} \wedge \overline{\pi^b} \qquad {\rm mod} \quad \ \{\, \theta^i \,\} .
\label{dtheta1}
\end{equation}
Applying the argument in the proof of Lemma \ref{BasicCoframe}, we have
$B^i_{ab}= P^i_j \overline{A^j_{ab}}$. We have now established that \eqref{onedsigma},\eqref{onedeta}, \eqref{onedtheta} hold.

Finally, taking the exterior derivative of equation \eqref{onedtheta}, the coefficient of 
$\overline{\sigma^u}\wedge \pi^a \wedge \pi^b $ is $\overline{\vS_u}( A^i_{ab})$ which must vanish.
Lemma \ref{DIX} then implies $A^i_{ab} \in \Invts$.
\end{proof}

\ifarch
\begin{remark}
In Remark \ref{TS10} it was noted that, for a $\JJ$-invariant integral manifold $s:\Sigma \to M$
of  an elliptic distribution $(\dD,\JJ)$ with induced almost complex structure $ \JJSig$, $ T^*_{1,0}\Sigma =s^* (V)$ and
is spanned pointwise by the forms $s^* \sigma^u$.  Equation  \eqref{onedsigma} shows that when $(\dD,\JJ)$ is Darboux integrable, ${\JJSig}$ is integrable, and the local coordinates from \eqref{zucoords} provided by the Frobenius Theorem will induce
holomorphic coordinates on $\Sigma$. 
The Darboux invariants on $M$ then restrict to be holomorphic functions on $\Sigma$.  While Darboux integrability provides a sufficient condition that $ \JJSig$ be integrable, 
Example 7.4.20 in \cite{CfB2} gives a representation of the minimal surface equation by an elliptic decomposable system on a 5-manifold; 
while the system is not Darboux integrable (but its prolongation is) the induced almost complex structure $ \JJSig$ on  any $\JJ$-invariant integral submanifold is integrable for dimensional reasons.
\end{remark}
\fi

\begin{cor}\label{Abelong}  Let $(\sysI, \dD, \JJ)$ be a Darboux integrable elliptic decomposable system, and 
$(\theta^i, \eta^r, \sigma^u, \etabars{r}, \sigmabars{u})$ a 1-adapted coframe satisfying structure equations \eqref{onedapt}.  Then the 2-forms defined by
\begin{equation}\label{tfAdef}
\tfA^i := \tfrac12 A^i_{ab} \pi^a \wedge \pi^b
\end{equation}
lie in $\sysI \otimes \C$, along with their complex conjugates.
\end{cor}
\begin{proof} Let $\Omega^\nu, \overline{\Omega^\nu}$ be the local 2-form generators of $\sysI\otimes \C$ provided by part \ref{EDdefV} of Definition \ref{EDdef}.
Then structure equation \eqref{onedtheta} implies that 
\begin{equation}\label{fonzie}
\tfA^i + P^i_j \overline{\tfA^j} \equiv S^i_\nu \Omega^\nu + T^i_\nu \overline{\Omega^\nu}\qquad \mod I_1\otimes\C
\end{equation}
for some functions $S^i_\nu, T^i_\nu$.  The splitting 
$$(T^*M /I_1) \otimes \C \cong \rV / (I_1 \otimes \C) \oplus \rVb / (I_1 \otimes \C)$$
induces a splitting of $\Lambda^2 (T^*M /I_1) \otimes \C$.  By applying the latter splitting to both sides of \eqref{fonzie}
we see that $\tfA^i \equiv S^i_\nu \Omega^\nu$ and $P^i_j \overline{\tfA^j} \equiv T^i_\nu \overline{\Omega^\nu}$ mod $I_1$.
Since $P$ is non-singular, then $\tfA^i,\overline{\tfA^i} \in \sysI \otimes \C$.
\end{proof}

\begin{defn} \label{SSdef} Let $(\sysI, \dD, \JJ)$ be a Darboux integrable elliptic decomposable system with associated singular bundle $V$.  We define its {\defem singular system} $\calV$ as the algebraic ideal generated (over $\C$) by sections of $V$ and the 2-forms of $\sysI$.  
\end{defn}


\begin{prop}
The singular system $\calV$ associated to a Darboux integrable elliptic decomposable system $(\sysI, \dD, \JJ)$ 
is a differential ideal and contains $\sysI \otimes \C$.
\end{prop}
\begin{proof}  Let $(\theta^i, \eta^r, \sigma^u, \etabars{r}, \sigmabars{u})$ be a 1-adapted coframe on an open set $\Udown$ about $\mpt \in M$ as in Theorem \ref{oneadaptedexist}, and let $\Omega^\nu$ be the set of generator 2-forms in $\Lambda^2 V$
referred to in condition \ref{EDdefV} in Definition \ref{EDdef}.  Then 
\begin{equation}
(\calI \otimes \C)\vert_\Udown = \{\ \theta^i,\ \eta^r,\ \overline {\eta^r},\ \Omega^\nu ,\ \overline{\Omega^\nu}\ \}\subalg .
\label{calI1a}
\end{equation}
On the other hand, 
\begin{equation}
\calV \vert_\Udown=  \{\ \theta^i,\ \eta^r,\  \overline {\eta^r},  \sigma^u, \ \overline{\Omega^\nu}\ \}\subalg
\label{calV1a}
\end{equation}
showing that $\sysI\otimes \C \subset \calV$.  Since $\sysI$ is differentially closed (by definition) and 
$\theta^i,\eta^r, \overline {\eta^r}, \overline{ \Omega^\nu} \in \sysI\otimes \C$
we need only observe from the structure equation \eqref{onedsigma} that trivially $d \sigma^u \in \calV$.  Thus $\calV$ is differentially closed.
\end{proof}

We now restrict our attention to the class of elliptic Darboux integrable systems whose singular system $\calV$ is Pfaffian. This is analogous to the class of hyperbolic systems whose singular systems form a Darboux pair (see Definition 1.3 in \cite{AFV}).



\begin{lemma}\label{LocIV} Let $(\sysI, \dD,\JJ)$ be a  Darboux integrable elliptic decomposable system whose singular system $\calV$  is Pfaffian, and let $\mpt \in M$ and 
$(\theta^i,\eta^r,\etabars{r}, \sigma^u, \sigmabars{u})$ be a 1-adapted coframe on the open set $\Udown$ about $\mpt$ satisfying equations \eqref{onedsigma}, \eqref{onedeta}, and \eqref{onedtheta} of Theorem \ref{oneadaptedexist}. Then
\begin{subequations}
\begin{align}
\calV\vert_\Udown&=\{\ \theta^i,\ \eta^r,\  \etabars{r},\ \sigma^u ,\overline{d\eta^{r}} , \ 
\overline{\tfA^i}\ \}\subalg, \label{LocalV}\\
(\sysI\otimes \C)\vert_\Udown &= \{\ \theta^i,\ \eta^r,\ \etabars{r}, \ d \eta^r ,\ \overline{d\eta^{r}} ,\ \tfA^i
,\ \overline{\tfA^i}\ \}\subalg, \label{LocalI}
\end{align}
\end{subequations}
where $\tfA^i$ is defined as in \eqref{tfAdef}.
\end{lemma}

\begin{proof}  
Since $\overline{\tfA^i} \in \sysI \otimes \C$ by Corollary \ref{Abelong}, $\sysI \otimes \C \subset \calV$ and $\calV$ is a differential ideal, then 
$$\{\ \theta^i,\ \eta^r,\  \etabars{r},\ \sigma^u ,\overline{d\eta^{r}} , \ 
\overline{\tfA^i}\ \}\subalg \subset \calV.$$ 
On the other hand, the fact that $\calV$ is Pfaffian
and $\overline{\Omega^\nu}\in \Lambda^2 \rVb$  
together imply that there exist functions $S^\nu_i, T^\nu_r$ such that
\begin{alignat}{2}
\overline{\Omega^\nu}&\equiv S^\nu_i d\theta^i+T^\nu_r \overline{d\eta^{r}} & \mod \
&\theta^i, \eta^r, \etabars{r}, \sigma^u \notag \\
&\equiv S^\nu_i P^i_j \overline{\tfA^j}+T^\nu_r \overline{d\eta^{r}}& \mod \
&\theta^i , \eta^ r, \etabars{r}, \sigma^u\ .
\label{Om2}
\end{alignat}
This shows that $\calV$, as given by the expression \eqref{calV1a}, is included the right-hand side of \eqref{LocalV}.

Similarly, Corollary \ref{Abelong} implies that  $\{\ \theta^i,\ \eta^r,\ \etabars{r}, \ d \eta^r ,\ d \etabars{r} ,\ \tfA^i
,\ \overline{\tfA^i}\ \}\subalg \subset \sysI \otimes \C$.  Then the expression \eqref{calI1a}, together with \eqref{Om2} and its complex conjugate, imply the inclusion
in the opposite direction so that \eqref{LocalI} holds.
\end{proof}

We note that the following proposition implies in particular that, if the standard representation of a second-order PDE in the plane by an EDS on a 7-dimensional manifold 
\ifarch (see Example \ref{introE1} below) \else (see Example 2 in \cite{FelsIveyArchive}) \fi 
is Darboux integrable, then its singular system is Pfaffian and thus Theorem \ref{MTE} applies.

\begin{prop} \label{Normdef} If $(\sysI,\dD,\JJ)$ is a  Darboux integrable elliptic decomposable system which is Pfaffian then 
its singular system $\calV$ is Pfaffian.
\end{prop}
\begin{proof} In terms of a 1-adapted coframe provided by Theorem \ref{oneadaptedexist},
the condition that $\calI$ is Pfaffian implies by equation \eqref{calI1a} that
$$ 
{\overline \Omega^\nu}= R^\nu_i d \theta^i+S^\nu_rd{\eta^r}\ +T^\nu_r  \ d \overline{\eta^r}\mod \
\theta^i, \eta^r, \overline{\eta^r} \, .
$$
This in turn shows $\calV$ is Pfaffian by equation \eqref{calV1a}.
\end{proof}


\subsection{Maximal Darboux Integrability} \label{SmaxDI}
\newcommand{\sysEM}{{\mathcal I}}
\newcommand{\WM}{{\rV}}
\newcommand{\EM}{{I}}
\newcommand{\DeltaM}{{\dD}}

\newcommand{\MDeltaplus}{{\dD_+}}
\newcommand{\MDeltaminus}{{\dD_-}}
\newcommand{\MDeltaplusinf}{{\dD_+^{(\infty)}}}
\newcommand{\MDeltaminusinf}{{\dD_-^{(\infty)}}}

\ifjour
\else
Maximally Darboux integrable systems play a pivotal role in the structure theory of elliptic  Darboux integrable systems, as highlighted in Theorem \ref{FTDI}.
\fi
In particular, the next proposition shows that the system $\sysE$ in Theorem \ref{MTE} belongs to a generic class of maximally Darboux integrable systems.  
Furthermore, in Theorem \ref{C34} we will show that quotients of these maximally integrable systems give rise to Darboux integrable systems whose singular systems are Pfaffian. 
(The converse is given by Theorems \ref{extendexist}, \ref{loopreduction} and Corollary \ref{CircleD}.)
Examples of maximally Darboux integrable system used in the quotient construction are given in \S\ref{SecondEx}. 

Let $(\sysEM, \DeltaM, \JJ)$ be a Darboux integrable elliptic decomposable system with singular bundle $\WM$ on a manifold $N$.  Suppose that $\sysI$ is maximally Darboux integrable, so that $\dim N= 2\, q$ 
where $q=\rank \WM^{(\infty)}$.  Thus,  $\rank  \EM = \dim N - 2d = 2(q-d)$ and the containment in \eqref{containment} becomes 
\begin{equation}
\EM\otimes \C =
 ( \overline{\WM}\cap \WM^{(\infty)} )  \oplus (\WM \cap  \overline{\WM^{(\infty)}}) \ . 
\label{maxDI}
\end{equation}
This decomposition arises naturally if $N$ is a complex manifold and $\sysI$ is 
the \realform\ of a holomorphic Pfaffian system (see Remark \ref{MaxR} below).  Such systems are
generated by sections of a holomorphic vector bundle $\rH \subset T^*_{1,0} N$.
Since $\rH \cap \overline{\rH} = 0$, then $I \otimes \C = \rH \oplus \overline{\rH}$.
(For example, if $\{\theta^j\}$ is a local basis of holomorphic sections of $\rH$ then $\theta^j +\overline{\theta^j}$ and $\ri (\theta^j - \overline{\theta^j})$ are pointwise linearly independent local sections of $\EM$.) 
By way of an {\it abus de langage} we say that $I$ is the {\defem \realform\ } of $H$, and write $I=H_{\R}$.

Since $I\otimes \C = \rH \oplus \overline{\rH}$, any integral manifold of $\sysI$ is an integral manifold of 
the 1-forms in $\rH$.  Conversely, since $\EM$ is spanned by real and imaginary parts of 
sections of $\rH$, then any integral manifold of $\rH$ is an integral manifold of $\sysI$. Note that
the $\JJ$-invariant integral manifolds of $\sysI$ with induced complex structure ${\JJSig}$ are holomorphic integral manifolds of $\sysH$ and conversely.

The following result, which is the elliptic analogue of Theorem 3.1 in \cite{AFV}, shows that 
$I=H_{\R}$ is maximally Darboux integrable. 

\newcommand{\maxdimN}{{p}}
\begin{prop}\label{HDI}  Let $N$ be a complex manifold of complex dimension $\maxdimN$ with complex structure $\JJ$, and let $\rH \subset T^*_{1,0}N$ be a rank $\rankH$ holomorphic Pfaffian system with  $\rH^{(\infty)}= 0$.  Let $\EM=H_{\R}$ be the \realform\ of $H$ and let $\sysEM$ be the corresponding Pfaffian system generated by $\EM$.  Then:
\begin{enumerate}[label=\thmenumstyle,ref=\thmenumstyle]

\item\label{maxhED} The distribution $\DeltaM=\EM^\perp$ is $\JJ$-invariant, $(\DeltaM, \JJ)$ is an elliptic distribution, where
$\DeltaM \otimes \C = \MDeltaplus \oplus \MDeltaminus$ and
\begin{subequations}\label{DefDD}
\begin{alignat}{2}
\MDeltaplus &= T_{1,0}N  \cap \rH^\perp, \qquad & \MDeltaminus &= T_{0,1}N \cap \overline{\rH}^\perp, \label{DefDD1} \\
\MDeltaplusinf &=T_{1,0}N, &
\MDeltaminusinf &= T_{0,1}N. \label{DefDD2}
\end{alignat}
\end{subequations}

\item\label{maxhspliv} 
The singular bundle $ \WM= (\MDeltaminus)^\perp$ satisfies
\begin{equation}
\begin{aligned}
\WM= (\MDeltaminus)^\perp &=& &  T^*_{1,0}N\oplus \overline{\rH}, \qquad  \WM^{(\infty)}  = T^*_{1,0}N. 
\end{aligned}
\label{DefW}
\end{equation}

\item\label{maxhmax} The system  $(\sysEM, \DeltaM, \JJ)$ is a maximally Darboux integrable elliptic decomposable system whose singular system $\calV$ is Pfaffian.
\end{enumerate}
\end{prop}

\ifjour
\else
\begin{proof}  {\bf Part \ref{maxhED}}: We check the conditions of Definition \ref{ED}.  First note that $\EM = \rH_\R$ has real rank $2\rankH$ (as shown above), so $\DeltaM= \EM^\perp$ has rank $2(\maxdimN-\rankH)$ and Definition \ref{ED}.\ref{Ddefeven} is satisfied. Moreover,
$$
\DeltaM \otimes \C = (H \oplus \overline{H})^\perp = H^\perp \cap \overline{H} ^\perp = (\MDeltaplus \oplus T_{0,1}N)\cap (  T_{1,0}N \oplus \MDeltaminus)
= \MDeltaplus\oplus \MDeltaminus
$$
with $\MDeltaplus$ and $\MDeltaminus$ being given by equation \eqref{DefDD1}. 

We next check Definition \ref{ED}.\ref{DdefJ} by showing that $\DeltaM$ is $\JJ$-closed. 
Fix a point $\npt\in N$ and let $\{ X_a\}$ be a local basis of sections of $\MDeltaplus$ near $\npt$, so that $\{ \overline{X_a}\}$ is a  basis of sections of $\MDeltaminus$. The real vector fields $T_a=X_a + \overline{X_a}$, $S_a= \ri(X_a - \overline{X_a})$ then 
give a local basis $\{S_a, T_a\}$ of sections for $\DeltaM$ near $\npt$. Since $X_a, \overline{X_a}$ take value in $T_{1,0}N$ and $T_{0,1}N$ respectively, then the $\C$-linear extension of $\JJ$ to $TN\otimes \C$ satisfies $\JJ(X_a) =  \ri X_a$ and $\JJ(\overline{X_a}) = -\ri \overline{X_a}$. Therefore $\JJ(T_a) = S_a$, 
$\JJ(S_a)=-T_a$, so that $\DeltaM$ is invariant under $\JJ$.   

We now verify condition Definition \ref{ED}.\ref{Ddefeigens}. By construction, $\MDeltaplus$ and $\MDeltaminus$ are precisely the $+\ri$ and $-\ri$ eigenspaces for $\JJ$ on $\DeltaM\otimes \C$, so it remains to check equation \eqref{DDB}. Since $H$ is holomorphic, we may assume that the local basis $\{ \vX_a\} $ of sections of $\MDeltaplus\subset T_{1,0}N$ consists of holomorphic vector fields, and hence $[\vX_a, \overline{\vX_b}] =0$.
If $\vX^+, \vY^-$ are arbitrary locally defined sections of $\MDeltaplus$ and $\MDeltaminus$ respectively, then 
$$
X^+= f^a X_a, \qquad Y^- = g^a \overline{ X_a}
$$
for locally defined functions $f^a, g^a$, and 
$$
[ X^+, Y^-]=
[f^a X_a, g^a \overline{X_b}] =X^+(g^a)  \overline{X_a} - Y^{-}(f^a) X_a  \in \MDeltaplus \oplus \MDeltaminus = \DeltaM \otimes \C.
$$
This shows that equation \eqref{DDB} holds and hence Definition \ref{ED}.\ref{Ddefeigens} is satisfied.
 
Finally we check Definition \ref{ED}.\ref{Ddefchow} along with verifying equations \eqref{DefDD2}. Since  $H^{(\infty)}=0$,
$$ 
TN \otimes \C=(H^{(\infty)})^\perp=(H^{\perp})^{(\infty)}=  (\MDeltaplus \oplus T_{0,1}N)^{(\infty)}= \MDeltaplusinf \oplus T_{0,1}N, 
$$
so that $\MDeltaplusinf = T_{1,0}N$; a similar computation finishes the proof of \eqref{DefDD2}.  Then by Lemma \ref{CDER},
$$
\DeltaM^{(\infty)} \otimes \C= \MDeltaplusinf+\MDeltaminusinf = TN\otimes \C
$$
and $\DeltaM^{(\infty)}= TN$. This shows  Definition \ref{ED}.\ref{Ddefchow} is satisfied, so that  $(\DeltaM, \JJ)$ is an elliptic decomposable distribution with decomposition given by equation \eqref{DefDD1}.

\smallskip
\noindent
{\bf Part \ref{maxhspliv}}: Equations \eqref{DefW} follow from the definition of $\MDeltaminus$ in \eqref{DefDD1} by taking the annihilator. 

\smallskip
\noindent
{\bf Part \ref{maxhmax} }: We  first show that $(\sysEM, \DeltaM, \JJ)$ is elliptic and decomposable (Definition \ref{EDdef}). 
From part \ref{maxhED} and the fact that $\sysEM$ is Pfaffian, Definition \ref{EDdef}.\ref{EDdefsys} and \ref{EDdefJ} are satisfied.
We now check the condition in Definition \ref{EDdef}.\ref{EDdefV}.  Choose a local basis $\{\theta^i\}$ of holomorphic sections of $H$,
extended so that  $\{ \theta^i, \pi^a\}$ is a local basis of holomorphic sections of $T^*_{1,0}N$.  Then we have
$$
d \theta^i = C^i_{jk} \theta^j \wedge \theta^k + A^i_{ab} \pi^a \wedge \pi^b +  M^i_{ja} \theta^i \wedge \pi^a
$$
for some locally defined holomorphic functions $  C^i_{jk}, A^i_{ab}, M^i_{ja} $ on $N$.
By equation \eqref{DefW} $\pi^a$ and $\theta^i$  are local sections of the singular system $\WM$,
and so $d\theta^i$ is a local section of $\Lambda^2 \WM$. Since $\EM\otimes \C=H\oplus \overline{H}$, a similar computation with $\overline {\theta^i}$ shows that the Pfaffian system
$\sysEM \otimes \C $ is generated (locally) by $\{\theta^i, d \theta^i, \overline {\theta^i} , d \overline{\theta^i} \}\subalg$, and hence Definition \ref{EDdef}.\ref{EDdefV} holds. 

We now check that is $\sysEM$ is maximally Darboux integrable. From equation \eqref{DefW}, 
$$
\WM^{(\infty)} + \overline{\WM}  =T^*_{1,0} N +   H\oplus T^*_{0,1} N= T^*N\otimes \C,
$$
so equation \eqref{VDIcond} is satisfied and  $(\sysEM, \DeltaM,\JJ)$ is Darboux integrable. From equation \eqref{DefW}, 
$\dim_{\R}N = 2\, \rank(\WM^{(\infty)}) $, so that $(\sysEM, \DeltaM,\JJ)$ is maximally Darboux integrable; moreover equation \eqref{maxDI} holds since $\WM^\infty \cap  \overline{\WM}=H$. Finally the singular system for  $(\sysEM, \DeltaM,\JJ)$ is Pfaffian by Proposition \ref{Normdef}.
\end{proof}
\fi

\begin{remark}\label{MaxR}
Equation \eqref{DefW} shows that the Darboux invariants for $(\sysI,\dD,\JJ)$ in this case (Definition \ref{DarbouxInv}) consist of the holomorphic functions on $N$ and that  equation \eqref{maxDI} holds where $ \overline{\WM}\cap \WM^{(\infty)}= H$. On account of the singular system $\calV$ being Pfaffian, equation \eqref{DefW} shows that $\calV=\{\Omega^{1,0}(M) \oplus \overline{\sysH}\}\subalg$.
\end{remark}



%

We also have a converse to Proposition \ref{HDI}.

\begin{prop} \label{MaxPfaff} If an elliptic decomposable Pfaffian system $(\sysEM, \dD,\JJ)$ on $N$ is maximally Darboux integrable then 
$\JJ$ extends to give an integrable complex structure on $N$ (so that
$N$ is a complex manifold), and there exists a holomorphic Pfaffian system $H$
such that $\sysEM$ is the \realform\ of $H$.
\end{prop}

\begin{proof} Let $\rV$ be the singular bundle of $(\dD,\JJ)$.  Since $\dim N =2 \rank \Vinf$ then $T^*N \otimes \C = \Vinf \oplus \overline{\Vinf}$.
This splitting defines an almost complex structure $\JJ$ on $N$ 
such that $(\Vinf)^\perp$ is the $-\ri$-eigenspace of $\JJ$, and which is an extension of that on $\dD$.
Because $\Vinf$ is a Frobenius system the almost complex structure is integrable (see, e.g., Theorem 2.8 in Chapter IX of \cite{KN}).  We set
$$H = \overline{\rV} \cap \Vinf.$$
Then $H \cap \overline{H}=0$ and by equation \eqref{maxDI}, $I\otimes \C= H\oplus \overline{H}$.
In terms of the adapted coframe provided by Theorem \ref{oneadaptedexist}, $H$ is locally spanned by
the 1-forms $\eta^r$.  It follows that from \eqref{onedeta} that these are holomorphic 1-forms.
\end{proof}

\subsection{Examples}\label{basicexamplesec}

The examples below demonstrate the basic Definitions \ref{ED}, \ref{EDdef} and \ref{EDIdef} 
for elliptic distributions, elliptic decomposability and Darboux integrability respectively,
focusing on second-order scalar PDE in the plane 
and first-order systems of PDE in two dependent and two independent variables.

\begin{example}[Laplace's equation]\label{ex:Laplace}
Let $x^i$, $u$, $u_i$, $u_{ij}$ be standard coordinates on the 8-dimensional jet space $J^2(\R^2, \R)$.
(In terms of these coordinates, $u$ is a function of independent variables $x^1, x^2$, and
$u_i$, $u_{ij}=u_{ji}$ are its first- and second-order partial derivatives.)
In these coordinates Laplace's equation is $u_{11}+u_{22}=0$. If $M \subset J^2(\R^2, \R)$ is 
the 7-dimensional submanifold defined by this equation, then a Pfaffian EDS $\sysI$ encoding Laplace's equation is generated by the pullbacks to $M$ of the standard contact forms:
\begin{equation}\label{generateplanePDE}
\sysI=\{ \vartheta_0:=\iota^*(du-u_1 dx^1 - u_2 dx^2),\  
\vartheta_1:=\iota^*(du_1- u_{11}dx^1 -u_{12} dx^2), \ 
\vartheta_2:=\iota^*(du_2- u_{12}dx^2 -u_{22} dx^2\}\subdiff,
\end{equation}
where $\iota$ is the inclusion $\iota:M \hookrightarrow J^2(\R^2, \R)$.
We will use the restrictions of $x^1,x^2,u,u_1,u_2, u_{11}, u_{12}$ as coordinates on $M$.
In terms of the coordinate vector fields, the distribution annihilated by the 1-forms of $\sysI$ is
$$\dD = \{ D_1 := \di_{x^1} + u_1 \di_u + u_{11} \di_{u_1} + u_{12} \di_{u_2}, D_2 := \di_{x^2}+u_2 \di_u + u_{12} \di_{u_1} - u_{11} \di_{u_2}, \di_{u_{11}}, \di_{u_{12}})\}.$$
The derived flag of this distribution satisfies
$$\dD^{(1)} = \{ D_1, D_2, \di_{u_1}, \di_{u_2}, \di_{u_{11}}, \di_{u_{12}}\}, \quad \dD^{(2)} = TM,$$
so $\dD$ is bracket-generating.

We will define a sub-complex structure on $\dD$ which makes $\sysI$ an elliptic system.  To do this (and illustrate Prop. \ref{ED}) we identify
two subspaces of $\dD \otimes \C$, each spanned by singular 1-dimensional integral elements.
Such singular integral elements are determined by the generator 2-forms of the EDS, which in this case are
$$I_2 = \{ du_{11} \wedge dx^1 + du_{12} \wedge dx^2, du_{12} \wedge dx^1 - du_{11} \wedge dx^2 \} \mod I_1.$$
Let $X=\xi^1 D_1 + \xi^2 D_2 + \xi^3 \di_{u_{11}} + \xi^4 \di_{u_{12}}$ be a nonzero vector in $\dD$.
The polar equations of $X$ are computed by taking the interior product of $X$ with the generator 2-forms.  The resulting 1-forms fail to be linearly independent if and only if 
\begin{equation}
0=(\xi^1)^2+(\xi^2)^2 = (\xi^3)^2+(\xi^4)^2 = \xi^1 \xi^3-\xi^2 \xi^4 = \xi^1\xi^4 +\xi^2 \xi^3.
\label{Lapxi}
\end{equation}
In particular, the 1-forms are linearly independent for all nonzero real $X$, and thus there are no real singular integral 1-planes.   When we allow the coefficients $\xi^a$ in equation \eqref{Lapxi} to be complex, the complex singular integral 1-planes belong to one of two subspaces of $\dD\otimes \C$,
$$\Dplus = \{ D_1 - \ri D_2, \di_{u_{11}} + \ri \di_{u_{12}} \}, \quad \Dminus = \overline{\Dplus} ,$$
which have the property that $\dD \otimes \C = \Dplus \oplus \Dminus$.

Matters being so, we define a sub-complex structure $\JJ$ on $\dD$ so that 
$\Dplus$ and $\Dminus$ are respectively the $+\ri$ and $-\ri$ eigenspaces for the extension of $\JJ$ to $\dD\otimes \C$ (see also Remark \ref{JfromDplus}).  It follows that 
$$\JJ: D_1 \mapsto D_2, \quad \JJ: D_2 \mapsto -D_1, \quad \JJ: \di_{u_{11}}\mapsto -\di_{u_{12}}, \quad\JJ: \di_{u_{12}} \mapsto \di_{u_{11}}.$$
It is easy to check that our decomposition of $\dD\otimes\C$ satisfies $[\Dplus, \Dminus] \subset \dD\otimes\C$;
Thus, $\sysI$ is an elliptic decomposable system. 

The terminal derived flag of $\Dplus$ is given by 
\begin{equation}
\Dplus^{(\infty)} =\Dplus^{(2)} = \{ D_1 -\ri D_2, \di_u, \di_{u_1} + \ri \di_{u_2}, \di_{u_{11}} + \ri \di_{u_{12}} \},
\label{LapDPinf}
\end{equation}
and comparing with $\Dminus =  \{ D_1 +\ri D_2, \di_{u_{11}} - \ri \di_{u_{12}}\}$ shows that the criterion \eqref{DDIcond} for 
Darboux integrability is satisfied.  Notice that in this case the integrability is not minimal in the sense of Definition \ref{maxmindef}, as the sum in \eqref{VDIcond} is not direct:
$\Vinf = (\Dminusinf)^\perp$ has rank 3 while $\overline{V}$ has rank 5, so these sub-bundles have a non-trivial intersection
inside $TM \otimes \C$.  The Darboux invariants 
are found by taking the conjugate of \eqref{LapDPinf} and 
solving for functions invariant under $\Dminusinf$; one choice of basis for $\Invts$ is 
the set $\{x^1+\ri x^2, u_1 - \ri u_2, u_{11}-\ri u_{12}\}$.
\end{example}

\bigskip
\begin{example}\label{introE1} Consider the general case of a second-order PDE  in the plane of the form
\begin{equation}\label{Feqn}
F(x^i,u,u_i,u_{ij})=0, \qquad 1\le i,j \le 2,
\end{equation}
where $x^i, u, u_i, u_{ij}$ are coordinates on $J^2(\R^2, \R)$ as in Example \ref{ex:Laplace}.
We assume that \eqref{Feqn} defines a smooth hypersurface $M_F \subset J^2(\R^2, \R)$, and that this hypersurface submerses onto $J^1(\R^2, \R)$.  
As in \eqref{generateplanePDE} we let $\sysI$ be the Pfaffian system generated
by the pullbacks to $M_F$ of standard contact 1-forms.

\def\Volmf{\operatorname{Vol}_{M_F}}
\def\bilin{\langle \,,\,\rangle}
Following \cite{GKsecond}, we define a symmetric bilinear form on sections $\phi,\psi$ of $I_1$ by
\begin{equation}
\langle \phi, \psi \rangle \Volmf = d\phi \wedge d \psi \wedge  \theta^0 \wedge  \theta^1 \wedge  \theta^2,
\label{GKform}
\end{equation}
where $\Volmf$ is a chosen volume form on $M_F$ and $\theta^0, \theta^1, \theta^2$ is any choice of local basis of 1-form generators for $\sysI$.  Note that since $\bilin$ is $C^\infty$-linear, it is well-defined pointwise.
However, with different choices of volume form or basis, the bilinear form changes
by multiplication by a non-vanishing smooth function on $M_F$, so it should be regarded as a {\em conformal} symmetric bilinear form.  

As the form $\bilin$ is only well-defined up to a multiple, its only invariants are its rank and the relative signs of its eigenvalues.  
The first derived system $I'$ is in the kernel of $\bilin$, so the rank of $\bilin$ is at most 2.
In \cite{GKsecond} $\sysI$ is defined to be elliptic if and only if the form $\bilin$ has rank 2 and its nonzero eigenvalues have the same sign at each point.  This implies the classical result that the PDE is elliptic if and only if 
$$
 \dfrac{\di F}{\di {u_{11}}}\dfrac{\di F}{\di {u_{22}}} - \frac{1}{4}\left(\dfrac{\di F}{\di u_{12}}\right)^2 > 0
$$
at each point of $M_F$; similarly, the PDE is hyperbolic if $\bilin$ has rank 2 and this discriminant is everywhere negative
 (see equations (2.21)-(2.24) in \cite{GKsecond} for the details).
 
Assume that $\sysI$ is elliptic.  By choosing generator 1-forms $\theta^0, \theta^1, \theta^2$  that diagonalize $\langle \,,\,\rangle$ over $\R$, we see that there is a choice of local coframe $(\theta^0, \theta^1, \theta^2, \omega^1, \pi^1, \omega^2, \pi^2)$ on $M_F$
which satisfies structure equations
\begin{equation}
\left.
\begin{aligned}
d\theta^0 &\equiv 0 \\
d\theta^1 &\equiv \omega^1 \wedge \pi^1 - \omega^2 \wedge \pi^2,\\
d\theta^2 &\equiv \omega^1 \wedge \pi^2 + \omega^2 \wedge \pi^1
\end{aligned}
\right\} \mod \theta^0, \theta^1, \theta^2.
\label{SEDE1}
\end{equation}
The two sub-bundles of $T^*M\otimes \C$ defined by 
$$
{\mathcal M}  = \spanc\{ \theta^0, \theta^1+ \ri \theta^2 \}, \quad \overline{\mathcal M}  =\spanc \{ \theta^0, \theta^1- \ri \theta^2 \}
$$
are the maximally isotropic subspaces of the $\C$-linear extension of the bilinear form $\langle \,,\,\rangle$ at each point of $M_F$.  
Setting $\beta=\theta^1+\ri \theta^2$, $\omega=\omega^1 + \ri \omega^2$ and $\pi=\pi^1 + \ri \pi^2$, we have from equations \eqref{SEDE1},
\begin{equation}
d\beta \equiv \omega \wedge \pi \quad \mod \theta^0, \beta, {\overline \beta}.
\label{dpde1}
\end{equation}
It follows that
\begin{equation}
\sysI \otimes \C= \{ \theta^0, \beta, {\overline \beta}, \omega \wedge \pi, {\overline \omega} \wedge {\overline \pi} \}\subalg.
\label{decomEPDE}
\end{equation}

We can use the maximally isotropic subspaces to give $\sysI$ the structure of an elliptic decomposable system.  
Letting $\dD = I^\perp$, in terms of the dual frame $( \partial_{\theta^0}, \partial_{\beta},\partial_{\overline{\beta}}, \partial_{\omega},\partial_{\overline{\omega}}, \partial_{\pi},\partial_{\overline{\pi}})$ we have
$$\dD \otimes \C = \spanc\{ \partial_{\omega},\partial_{\overline{\omega}}, \partial_{\pi},\partial_{\overline{\pi}}\}.$$
On the other hand, following (3.18) in \cite{GKsecond} we define the characteristic vector field system
$$\operatorname{Char}(I\otimes \C, d{\mathcal M}) := \{ X \in \dD\otimes \C \mid X \hook d{\mathcal M} \subset I\otimes \C \ \}.$$
and compute that $\operatorname{Char}(I\otimes \C, d {\mathcal M} ) = \spanc\{ \partial_{\overline{\omega}}, \ \partial_{\overline{\pi}}\}$.
Furthermore it is easy to check that $X\in \dD\otimes \C$ spans a singular integral 1-plane if and only if
it lies in $\operatorname{Char}(I\otimes \C, d {\mathcal M} )$ or its conjugate.
Thus, setting 
\begin{equation}\label{EPDEDminus}
\Dminus=\operatorname{Char}(I\otimes \C, d {\mathcal M}), \qquad \Dplus= \overline{\Dminus},
\end{equation}
we have $\DD\otimes\C=\Dplus \oplus \Dminus$ and we define $\JJ:\DD \to \DD $ as in Remark \ref{JfromDplus}.  
(Note that if in \eqref{EPDEDminus} we had instead set $\Dplus=\operatorname{Char}(I\otimes \C, d {\mathcal M} )$ then
$\JJ$ is modified by a minus sign.)
The singular bundle is computed to be
\begin{equation}
V=(\Dminus)^\perp=\spanc\{\, \omega, \pi, \theta^0,\beta,\overline{\beta}\, \}.
\label{VEPDE2}
\end{equation}
It now follows from \eqref{dpde1}, \eqref{decomEPDE}, \eqref{EPDEDminus}, and \eqref{VEPDE2}
that $(\calI, \DD, \JJ)$ satisfies Definition \ref{EDdef} and is an elliptic decomposable system. 
Finally it also follows from the structure equations \eqref{SEDE1} that any 2-dimensional integral manifold 
$s:\Sigma\to M$ satisfying the independence condition $\omega^1 \wedge \omega^2\neq 0$ is $\JJ$-invariant, 
so that $\JJ$ restricts to give an almost complex structure on $\Sigma$ which is integrable if $\sysI$ is Darboux integrable.


The objects ${\mathcal M}, \Dminus, \rV$ in this example are easily shown to be (up to conjugation) differential invariants of $\sysI$. 
If a diffeomorphism $\Phi:M_F\to M_F$ is a symmetry transformation, so that $\Phi^*\sysI=\sysI$, then since pullback and $d$ commute, 
$\langle \Phi^* \phi,  \Phi^* \psi\rangle = \lambda \langle \phi, \psi\rangle$ for some non-vanishing function $\lambda$.
Hence the sub-bundle $\Phi^*{\mathcal M}$ is again a 2-dimensional isotropic subspace for $\bilin$. From this and the form of $\Dminus$ we conclude that either
\begin{equation}\label{PDECont}
\begin{matrix}
& \Phi^*{\mathcal M} ={\mathcal M}, & \Phi_*{\Dminus} ={\Dminus},& \Phi^*{V} ={V},& \Phi_* \circ \JJ = \JJ \circ \Phi_*\\
\text{or\qquad}\\
&\Phi^*{\mathcal M} ={\overline{\mathcal M}},&  \Phi_*{\Dminus} =\Dplus,& \Phi^*{V} ={\overline{V}},& \Phi_* \circ \JJ =-\JJ \circ \Phi_*\ .
\end{matrix}
\end{equation}
\end{example}

\bigskip
\begin{example}\label{benexample1}
Consider the following first-order partial differential equation for a complex-valued function $W$ of a complex variables $z$ and conjugate $\zbar$,
\begin{equation}\label{beneqncomplex}
\dfrac{\di W}{\di \zbar} = \tfrac12 |W|^2,
\end{equation}
which appeared in \cite{McKay} and is further discussed in \cite{Eendebak}.
If we introduce real variables $x,y,u,v$ such that $W=u+\ri v$ and $z=x+\ri y$, then the real and imaginary parts of this equation give a first-order system for $u,v$:
\[
\tfrac12\left(\dfrac{\di}{\di x} + \ri \dfrac{\di}{\di y}\right)(u+\ri v) = \tfrac12 (u^2+v^2) 
\quad
\Longleftrightarrow
\quad
\left\{
\begin{aligned}\dfrac{\di u}{\di x} - \dfrac{\di v}{\di y} &= u^2+v^2, \\ 
\dfrac{\di u}{\di y} +\dfrac{\di v}{\di x} &= 0.
\end{aligned}\right.
\]
Let $x,y,u,v,u_i,v_i$ be standard coordinates on jet space $J^1(\R^2, \R^2)$.  We encode this first-order
system as a Pfaffian EDS by restricting the standard contact forms of the jet space to the 6-manifold $M$
defined by $u_1 - v_2 = u^2 + v^2$ and $u_2 + v_1 = 0$.  Retaining $u_1$ and $v_1$ as coordinates on $M$, this results in a Pfaffian system
\[
\sysI = \{ du - u_1 \dx +v_1 \dy, dv - v_1 \dx - (u_1 - u^2-v^2)\dy\}\subdiff.
\]
Note that while $u_1+\ri v_1$ gives the value of $\di W/\di z$ when $W$ is a holomorphic function of $z$, this is not the case for a solution of system $\sysI$; in fact, 
\begin{equation}\label{deeWmaudi}
dW \equiv \left( u_1 + \ri v_1 -\tfrac12(u^2+v^2)\right)dz + \tfrac12 (u^2+v^2)d\zbar 
\mod \sysI.
\end{equation}

In terms of the coordinate vector fields on $M$, the distribution annihilated by the 1-forms of $\sysI$
is
$$\dD = \{ D_x := \di_x + u_1 \di_u + v_1 \di_v, D_y := \di_y - v_1 \di_u +(u_1 -u^2-v^2)\di_v, \di_{u_1}, \di_{v_1}\}.$$
It is easy to check that $\dD^{(1)}=TM$ at points of $M$ where $u,v$ are not both zero.  So, we will restrict our attention to such points, in order that $\dD$ be bracket-generating with a regular derived flag.

Again, we will derive a sub-complex structure on $\dD$ by decomposing $\dD \otimes \C$ into spaces spanned by singular integral elements.
The generator 2-forms of the EDS are
$$\sysI_2 \equiv \{ du_1 \wedge dx -dv_1 \wedge dy, (dv_1 + 2 v v_1 \dy) \wedge dx + (du_1 - 2u u_1 \dx) \wedge \dy \} \mod \sysI_1.$$
Taking the interior product of $X =\xi^1 D_x + \xi^2 D_y + \xi^3 \di_{u_1} + \xi^4 \di_{v_1} \in \dD$ with these 2-forms yields a pair of 1-forms which are linearly independent
for all nonzero real vectors $X$.  When we allow the coefficients $\xi^a$ to be complex, the polar equations drop rank if and only if $X$ lies in either the rank 2 sub-distribution
$$
\Dplus = \{ 
D_x - \ri D_y + 2(u u_1 + v v_1) \di_{u_1},
\di_{u_1} -\ri \di_{v_1} \} \subset \dD \otimes \C
$$
or its complex conjugate $\Dminus = \overline{\Dplus}$.  As in Remark \ref{JfromDplus}, this determines a sub-complex structure on $\dD$.

One can check that the decomposability condition in equation 
\eqref{DDB} also holds here; for example,
\begin{align*}
[D_x - \ri D_y + 2(uu_1 + vv_1) \di_{u_1},  D_x + \ri D_y + 2(uu_1 + vv_1) \di_{u_1}] &= 4\ri (u^2+v^2) v_1 \di_{u_1},
\\
[D_x - \ri D_y + 2(uu_1 + vv_1) \di_{u_1}, \di_{u_1} +\ri \di_{v_1}] &=-2(u+\ri v)\di_{u_1}.
\end{align*}
Thus, $(\sysI, \JJ)$ is an elliptic decomposable Pfaffian system.
(More generally, any quasilinear first-order complex PDE of the form
$\dfrac{\di W}{\di \zbar} = F(z,\zbar,W,\Wbar)$
leads in the same fashion to an elliptic decomposable EDS.)

The terminal derived flag of $\Dplus$ is
\begin{equation}
\Dplus^{(\infty)} = \Dplus^{(2)} = \{ D_x - \ri D_y + 2(u u_1 + v v_1) \di_{u_1}, \di_{u_1} -\ri \di_{v_1}, 
\di_u-\ri \di_v +(u-\ri v) \di_{u_1}, \ri(u-\ri v) \di_v +(u_1-\ri v_1) \di_{u_1} \},
\label{BenDinf}
\end{equation}
and since this has zero intersection with $\Dminus$ the system is Darboux integrable by Defn. \ref{EDIdef}.  In this case the integrability is minimal,
since in equation \eqref{VDIcond} the sum is direct:  $\Vinf=(\Dminusinf)^\perp$ has rank 2 and $\overline{V}=(\Dplus)^\perp$ has rank 4.
A basis for the (holomorphic) Darboux invariants is found, using the conjugate of equation \eqref{BenDinf} and Lemma \ref{DIX}, to be the two functions $x+\ri y$ and 
\begin{equation}
\xi=\dfrac{u_1+\ri v_1}{u+\ri v}-u.
\label{BenRinv}
\end{equation}
\end{example}

\bigskip
\begin{example} \label{CRstandard}  
Let $N  = \Jhol{1}(\C,\C)$ with holomorphic contact system 
$H = \spanc \{ \, dw -w _z dz\, \}$ and let 
$ (x,y,u,v, u_1,v_1)$ be real coordinates on $N$ where $z=x+\ri y, w= u+ \ri v, w_z=u_1 + \ri v_1$. Then the \realform\ $I= H_\R$ in Proposition \ref{HDI} is found by writing
$$
dw -w_z dz =  du - u_1 dx +v_1 dy + \ri( dv -v_1 dx - u_1 dy)
$$
and
$$
I = H_\R = \spanr\{\ du - u_1 dx +v_1 dy, \ dv -v_1 dx - u_1 dy \ \},
$$
which is a standard representation of the Cauchy-Riemann equations by a Pfaffian system on a 6-dimensional manifold. By equation \eqref{DefW} the singular system is
$$
\WM= {\rm span }_\C \{\  dz,\ dw,\ dw_z, \ d\overline{w} -\overline{w _z} d\overline{z} \ \}
$$
while the coordinate functions $z, w, w_z$ are a basis for the holomorphic Darboux invariants 
and $\sysI$ is maximally Darboux integrable.
\end{example}


\def\Ad{\operatorname{Ad}}
\def\Vert{\operatorname{Vert}}
\def\VB{\operatorname{Vert}}

\section{Symmetry Reduction of Maximally Darboux Integrable Elliptic Systems}\label{reduction}

In this section we show that the quotient of a maximally Darboux integrable system (as given in Proposition \ref{HDI}) by a symmetry group
satisfying a natural transversality condition is again a Darboux integrable system. 
This will prove the sufficiency part of Theorem \ref{MTE} from the Introduction.  The full technical details are given in Theorem \ref{C34} of this section, which is the elliptic analogue of Theorem 6.1 in \cite{AFB} (see also Corollary 3.4 in \cite{AFV}).  Examples which demonstrate the quotient process in Theorem \ref{C34} are given in \S \ref{SecondEx}.

\subsection{Group Action Conventions}\label{conv}

Let $K$ be a connected complex Lie group of complex dimension $\dimG$ admitting a real form $G\subset K$. This implies there exists a basis $\{ Z^L_i\}$ of the complex Lie algebra $\fk$ of left-invariant holomorphic vector fields on $K$ such that $\{ X^L_i= Z^L_i+ \overline{ Z^L_j} \}$ is a basis for the real Lie algebra $\fg$ of left-invariant vector fields on $G$. These bases satisfy
\begin{equation}
[Z^L_i , Z^L_j] = C^k_{ij} Z^L_k, \quad [X^L_i , X^L_j] = C^k_{ij} X^L_k. 
\label{RC}
\end{equation}
for real structure constants $C^i_{jk}$.

Let $Z^R_i$ be the right-invariant holomorphic vector fields coinciding with $Z^L_i$ at the identity, and
let $\mu_R^i$ and  $\mu^i_L$ be the right- and left-invariant holomorphic 1-forms on $K$ dual to $\{ Z^L_i\}$ and $\{ Z^R_i\}$ respectively.
Let  $\Ad^*:K \to \operatorname{Aut}(\fk^*)$ be the co-adjoint representation of $K$, and $\Lambda$ be the matrix-valued function on $K$
such that $\Lambda(a)= [\Ad^*(a)]^T$, i.e., 
\begin{equation}
\Ad^*(a)\mu^i_L= \Lambda^i_j (a) \mu_L^j.
\label{OM1}
\end{equation}
Note that $\Lambda$ and $\Omega= \Lambda^{-1}$ satisfy
\begin{equation}
 \mu^i_L= \Lambda^i_j \mu_R ^j,\quad R_a^* \left( \mu^i_L\vert_{ba}\right) = \Lambda^i_j(a) \mu^j_L \vert_b  \ , 
 \quad \mu^i_R= \Omega^i_j \mu_L ^j,\quad
\Lambda^i_k( a b ) = \Lambda^i_j(b) \Lambda^j_k(a),
\label{OM2}
\end{equation}
where $R_a:G \to G$ is right-multiplication by $a\in G$.
The structure equations $d\mu_L^i = -\tfrac12 C^i_{jk} \mu_L^j \wedge \mu_L^k$ of the Lie group $K$ imply that 
\begin{equation}
d \Omega^i_j = \Omega^i_k C^k_{\ell j}\mu_L^\ell, \quad d \Lambda^i_j = -C^i_{kl} \Lambda^\ell_j \mu_L^k
\label{dLam}
\end{equation}
and 
\begin{equation*}
\Omega^i_j C^j_{\ell m} = C^i_{jk} \Omega^j_\ell \Omega^k_m.
\end{equation*}
Let $\qG:K \to K/G$ be the quotient map; on the homogeneous space $K/G$ the real vector fields
$$
\widetilde X_i ={\qG}_*(Z^R_i+ \overline{ Z^R_j} ),\quad \widetilde Y_i =\ri\, {\qG}_*(Z^R_i- \overline{ Z^R_j}) \ ,
$$
form a basis for the Lie algebra of infinitesimal generators of the action of $K$ on $K/G$, while
\begin{equation*}
\widetilde Z_i = \frac{1}{2} (\widetilde X_i - \ri \widetilde Y_i) = {\qG}_* {\overline Z_ i^R}
\end{equation*}
form a global basis of sections of $T(K/G)\otimes \C$ and satisfy 
$$
[\widetilde Z_i, \widetilde Z_j] = -C^k_{ij} \widetilde Z_k.
$$

We now assume $\rho:N\times K \to N$ is a right action of $K$  on a complex manifold $N$ of dimension $\dimN$ with the properties that it acts freely and 
holomorphically,
and that $\piK:N \to N/K$ is a holomorphic principal $K$-bundle over the complex manifold $N/K$.
We also assume the quotient $N/G$ by the action of $G$ is a differentiable manifold (of dimension $2\dimN - \dimG$) with $\piG:N \to N/G$ a principal $G$-bundle. This implies there is a submersion $\phi: N/G \to N/K$ giving $N/G$ the structure of a $K/G$-bundle over complex manifold $N/K$, such that the following diagram commutes:
\begin{equation} 
\begin{gathered}
\tikzstyle{line} = [draw, -latex', thick]
\begin{tikzpicture}
\node(B) {$N$};
\node[below of=B, node distance=20mm](I2) {$N/G$};
\node[right of=I2, node distance=25mm](I1) {$N/K$};
\path[line](B) -- node[right, label={$\piK$}] {} (I1);
\path[line](B) -- node[left] {$\piG$} (I2);
\path[line](I2) -- node[below] {$\phi$} (I1);
\end{tikzpicture}
\end{gathered}
\label{CD1}
\end{equation}


Let $\Gamma_{\fk}\subset T_{1,0}N $ be the $\dimG$ dimensional complex Lie algebra of infinitesimal generators of the action of $K$ on $N$ and denote by $Z_i$  the basis of infinitesimal generators corresponding
to the choice of basis $\{ Z^L_i\}$ of $\fk$ satisfying \eqref{RC}.  The distribution of complex  rank $\dimG$ given by the pointwise span of $Z_i$ will be denoted 
by ${\bf \Gamma_{\fk}}$.  
The real vector fields $X_j=Z_j + \overline{Z_j}$ are a basis for the $\dimG$-dimensional real Lie algebra, denoted by $\Gamma_{\fg}$, of infinitesimal generators of the action of $G$ on $N$, and correspond to the choice of left-invariant vector fields $X_i^L$ satisfying equation \eqref{RC}.  Let  ${\bf \Gamma_{\fg}}\subset TN$ be the real rank $\dimG$ distribution defined by the pointwise span of the $X_j$. These vector fields on $N$ satisfy
the bracket relations
$$
[Z_i , Z_j] = C^k_{ij} Z_k ,\quad [Z_i, \overline{ Z_j}] = 0,  \quad [X_i, X_j] = C_{ij}^k X_k. 
$$
We also have pointwise on $N$,
\begin{equation}
\Vert_{\piG}=\operatorname{ker}\, {\piG}_*= \spanr\{ X_i \}={\bf \Gamma_{\fg}},
\label{GIG}
\end{equation}
for  the bundle of vertical vectors for $\piG$.


\subsection{The Quotient Theorem}\label{QDI}




\ifjour
\else
The main result proven in this section is the sufficiency part of Theorem \ref{FTDI} (or Proposition \ref{MTE}) in the Introduction.
\fi
\begin{thm}\label{C34}  (Quotient Theorem) Let $N$ be a complex manifold of complex dimension $\dimN$, and let $\rH \subset T^*_{1,0} N$ be a rank $\rankH$ holomorphic Pfaffian system with  $\rH^{(\infty)}= 0$. Let $\rE= H_{\R}$ be the rank $2\rankH$ \realform\ of $H$ and let $\Delta=\rE^\perp \subset TN$ be the real rank $2(\dimN - \rankH)$ distribution annihilated by $\rE$.

Let $\rho:N \times K \to N$ be a right action of the $\dimG$-dimensional complex Lie group $K$ with real form $G\subset K$, such that $K$ acts  freely on $N$ with $\piK:N \to N/K$ being a principal bundle as in Section \ref{conv}. Furthermore assume that $K$ acts by symmetries of $\rH$, and transversely to $H$, i.e., $H^\perp \cap {\bf \Gamma_{\fk}}  =0$.  Suppose that $G$ acts regularly on $N$ with principal bundle  $\piG:N \to M$ where $M=N/G$.   Then the following hold:

\begin{enumerate}[label=\thmenumstyle,ref=\thmenumstyle]

\item \label{qcrank} The Pfaffian system $\sysE$ generated by $\rE$ is $K$-invariant, the $K-$ (and hence $G-$) action is transverse to $E$, and
\begin{equation}
\sysI = \sysE/G,\qquad \DD=I_1^\perp={ \piG}_*(\Delta) 
\label{QDD}
 \end{equation}
are of constant rank, with $\rank \DD= 2(\dimN - \rankH)$.

\item \label{qJ}  The complex structure $\JJ$ on $N$ induces a pointwise linear map $\JJ_0 : \DD \to \DD$ defined by
\begin{equation}
\JJ_0( X) = {\piG}_* \,  \JJ( Y) , \qquad \text{where } {\piG}_* \,(\,  Y \, ) = X \in \DD
\label{defJJ0}
\end{equation}
such that $( \DD, \JJ_0)$ is an elliptic distribution.  Moreover $\DD \otimes \C = \Dplus \oplus \Dminus$ where
\begin{equation}
\Dplus={\piG}_* \, (\, \Deltaplus \,)
,\quad  \Dminus={\piG}_* \, (\,  \Deltaminus \, )
\label{defDD}
\end{equation}
is the $\JJ_0$-eigenspace decomposition, and the decomposition $\Delta \otimes \C= \Deltaplus\oplus \Deltaminus$ is given by 
equation \eqref{DefDD} (with $\DD$ replaced by $\Delta$).

\item \label{Qsing} With $\phi$ being given in diagram \eqref{CD1}, the  singular  bundle $V=\Dminus^\perp $ satisfies
\begin{align}
\rV &= (T^*_{1,0}N\oplus \overline{H})/G, \label{quotientV} \\
\Vinf
&=(T^*_{1,0}N) /G = \phi^*(T^*_{1,0}(N/K)) \label{quotientVinf},
\end{align}
and $(\DD,\JJ_0)$ is a Darboux integrable elliptic distribution.

\item \label{qE} The system $(\sysI,\DD, \JJ_0) $ is a Darboux integrable elliptic decomposable system whose singular system $\calV$ is Pfaffian and is the quotient of the singular system for $\calE$.


\end{enumerate}
\end{thm}



\begin{remark}\label{KDINV} Equation \eqref{quotientV} shows that the singular bundle $V$ of $\sysI=\sysE/G$ is the quotient of the singular bundle $T^*_{1,0}N\oplus \overline{H}$ of $\sysE$ (given by \eqref{DefW}). Then equation \eqref{quotientVinf} shows that the Darboux invariants for $\sysI$ are the pullbacks from $N/K$ of holomorphic functions on $N/K$. Since the holomorphic functions on $N/K$ are precisely the $K-$invariant holomorphic functions on $N$, we will use this to determine the Darboux invariant for $\sysI$ in the examples; 
in particular, see equations  \eqref{QDI2}  and \eqref{QDI3}.
\end{remark}

\begin{proof}   {\bf Part  \ref{qcrank}}:  Let $\thetau^\alpha$ be a local basis of sections of $H$ defined on a $K$-invariant open set $\Upset \subset N$ and let $a\in K$ . Since $H$ is $K$-invariant, $\rho_a^* \thetau^\alpha = R^\alpha_\beta \thetau^\beta$ for some holomorphic functions 
$R^\alpha_\beta$, 
and therefore 
$$
\rho_a^* (\thetau^\alpha + \overline{ \thetau^\alpha}) = A^\alpha_\beta (\thetau^\beta + \overline{ \thetau^\beta}) + B^\alpha_\beta (  {\ri}( \thetau^\beta - \overline{ \thetau^\beta} ) )
$$
where $2 A^i_j=  ( R^i_j + \overline{R^i_j}) $ and $ 2 B^i_j = -\ri  ( R^i_j + \overline{R^i_j}) $ are real matrices. A similar computation with $ {\ri}( \thetau^\alpha - \overline{ \thetau^\alpha} )$ show that $K$ is a symmetry group of the real Pfaffian system $\sysE$. The subgroup $G\subset K$ is then also a symmetry group.

We now prove that $K$ (and hence $G$) acts transversely to $E$. 
Using the initial transversality condition $ {\bf \Gamma_{\fk}}\cap H^\perp=0$, we have that near each point $\npt$ in $N$ there exist local sections $\thetau^i$ of $H$ such that $\thetau^i(Z_j) = \delta^i_j$ where $Z_i$ are the infinitesimal generators of the right $K$-action as in section \ref{conv}. Therefore, 
$$
(\thetau^i+ \overline{\thetau^i}) (X_j) = 2 \delta^i_j ,\  (\thetau^i+ \overline{\thetau^i}) (Y_j) = 0,\ 
{\ri} (\thetau^i- \overline{\thetau^i}) (X_j) = 0  ,\  {\ri} (\thetau^i+ \overline{\thetau^i}) (Y_j) = 2 \delta^i_j
$$
where $X_j= Z_i + \overline{Z_i}$, $Y_j= {\ri} ( Z_i - \overline{Z_i})$. This shows that $K$, as a real $2 \dimN$-dimensional Lie group, acts transversely to $E$, and hence so does $G$. Therefore by Corollaries 7.1 and 7.4 in \cite{AF2005} the quotient differential system $\sysE/G$ is constant rank and is algebraically generated by 1-forms and 2-forms.  We also have
$$
\DD= \IC^\perp = {\piG}_*(E^\perp) ={ \piG}_*(\Delta)
$$
which by transversality $(E^\perp\cap \Vert_{\piG} = 0$) implies that $\rk \DD = \rk \Delta = 2(\dimN - \rankH)$.

\smallskip
\noindent
{\bf Part \ref{qJ}}:   We begin by noting that the complex structure $\JJ$ is $K$-invariant, as is $\Delta=E^\perp$ as shown in  part \ref{qcrank} above. Therefore  $\JJ_0:\DD\to \DD$ in equation \eqref{defJJ0} is well defined, linear and satisfies $\JJ_0^2 = -\Id$. At this point conditions 
\ref{Ddefeven} and \ref{DdefJ} in Definition \ref{ED} hold, and it remains to check the other conditions in Definition \ref{ED} to verify ellipticity.

Definition \ref{ED}.\ref{Ddefchow}  for $\DD$ in equation \eqref{QDD} is verified by first recalling that  $\Delta$ is elliptic so $\Delta^{(\infty)} = TN $ (see part \ref{maxhED} in Proposition \ref{HDI}), and therefore by Theorem 2.2 of \cite{AFB},
$$
\DD^{(\infty)} = {\piG}_* \, (\,  \Delta^{(\infty)}\, )= {\piG}_* \,  (\, T N\, ) = T(N/G)= TM.
$$

Definition \ref{ED}.\ref{Ddefeigens} is verified by first noting that the $K$-invariant decomposition $\Delta\otimes \C =\Deltaplus \oplus \Deltaminus$, along with equation \eqref{defDD} defining $\Dplus, \Dminus$, gives
\begin{equation}
\DD \otimes \C = {\piG}_*( \Delta \otimes \C) =  {\piG}_*( \Deltaplus \oplus \Deltaminus) ={\piG}_*( \Deltaplus) + {\piG}_*( \Deltaminus)= \Dplus + \Dminus.
\label{eqc4}
\end{equation}
Since $\rk \Delta\otimes\C =\rk \DD\otimes \C= \rk \Deltaplus+ \rk \Deltaminus$, the sum in equation \eqref{eqc4} must be direct. 
We now check the eigenspace condition for $\JJ_0$. Suppose that $X\in \Dplus$ and $Y\in \Deltaplus$ with ${\piG}_*Y=X$.  Then 
$$
\JJ_0(X)={\piG}_*(\JJ\,Y)={\piG}_* ( \ri Y) = \ri X,
$$
so that $\Dplus$ as defined in equation \eqref{defDD} is the $\ri$-eigenspace of $\JJ_0$, and similarly  $\Dminus$ is the $-\ri$-eigenspace of $\JJ_0$. 

Lastly we check equation \eqref{DDB} in Definition \ref{ED}.\ref{Ddefeigens}. Consider vector fields $X_+\in \Dplus$, $X_-\in \Dminus$, $Y_+\in \Deltaplus$ and $Y_-\in \Deltaminus$ such that $X_+= {\piG}_* Y_+$ and $X_-= {\piG}_* Y_-$.  Since $[Y_+, Y_-]\in \Delta\otimes \C$,
$$
[X_+, X_-] = [{\piG}_* Y_+, {\piG}_* Y_-] = {\piG}_*[ Y_+, Y_-] \in \DD\otimes\C.
$$
All conditions in Definition \ref{ED} are thus satisfied and $(\DD, \JJ_0)$ is an elliptic distribution.  

\smallskip
\noindent
{\bf Part \ref{Qsing}}:  The formula \eqref{quotientV} for the singular bundle $V= (\Dminus)^\perp$ follows from  equation \eqref{DefW}, giving
$$
V = (\Dminus)^\perp= \left( {\piG}_*( \Deltaminus) \right)^\perp = (\Deltaminus)^\perp /G=
\left(  T^*_{1,0}N\oplus \overline{\rH} \right)/G .
$$ 
By equation \eqref{DefW},  $((\Deltaminus)^\perp)^{(\infty)}=T^*_{1,0}N$, and the action of $G$ is transverse to this.  Thus, 
the first equality in \eqref{quotientVinf} follows from \eqref{quotientV} by fact {\bf [d]} on page 1921 in \cite{AFV} 
(i.e.,  if a symmetry group $G$ acts transversely to $I^{(\infty)}$, for some Pfaffian system $I$, then $(I/G)^{(\infty)} = I^{(\infty)}/G$\,).
The final part of
equation \eqref{quotientVinf} follows from the commutative diagram \eqref{CD1}; see Section 3.1 in \cite{AFB}.

We now check that $(\DD, \JJ_0)$ satisfies equation \eqref{DDIcond} and is Darboux integrable.
Since $\Deltaplus^{(\infty)}= T_{1,0}N$ from \eqref{DefDD}, we have $\Deltaplus^{(\infty)} \cap (\Vert_{\piG}\otimes \C)=0$,
so that by the dual of the above-cited fact,
$$
\Dplus^{(\infty)} = {\piG}_*( \Deltaplus^{(\infty)} ) , \quad \Dminus^{(\infty)} = {\piG}_* (\Deltaminus^{(\infty)}).
$$
Since $(\Vert_{\piG}\otimes \C) \cap (T_{1,0} N \oplus \Deltaminus) = 0$ we have
$$
\Dplus^{(\infty)}  \cap \Dminus = 
{\piG}_*\left(T_{1,0} N \right)  \cap {\piG}_*\left(\Deltaminus\right)
=  {\piG}_*( T_{1,0} N \cap \Deltaminus ) = 0.
$$
Therefore $(\DD,\JJ_0)$ is Darboux integrable.

\bigskip
\noindent
{\bf Part \ref{qE}}:   We begin by checking that $(\sysI,\DD,\JJ_0)$ is an elliptic decomposable differential system. As noted in the proof of part \ref{qcrank}, the transversality hypothesis implies that  $\sysI=\sysE/G$ is a constant rank EDS generated by 1-forms and 2-forms which is part \ref{EDdefsys} of Definition \ref{EDdef}. Part \ref{qJ} of this theorem shows that $(\sysI, \DD,\JJ_0)$ satisfies Definition \ref{EDdef}.\ref{EDdefJ}, so we need only check the decomposability condition \ref{EDdefV} in Definition \ref{EDdef}.

Let $Z_i$ be a basis for the infinitesimal generators for the action of $K$ on $N$.  
\ifjour
\else
Again, by the transversality of the $K$  action, Theorem \ref{HICF} below implies
\fi  
that we have a local basis $(\thetau^i, \eta^r, \sigma^u)$ of holomorphic
sections of $T^*_{1,0} N$ defined on some $K$-invariant open set $\Udown \subset N$  such that $H=\spanc\{ \thetau^i, \eta^r\}$, $\thetau^i(Z_j) = \delta^i_j$,
and $\eta^r, \sigma^u$ are $K$-basic. 
\ifjour
\else
Then from the structure equations \eqref{HSE} below we have
\fi
$$
\sysH\otimes\C = \{\ \thetau^i,\ \eta^r,\ \tfA^t \ \}\subalg,\quad \sysE\otimes\C=\{\ \thetau^i,\ \eta^r,\ \tfA^t, \ \overline{\thetau^i},\ \overline{ \eta^r},\overline{ \tfA^t }\  \}\subalg
$$
\ifarch
where the $\tfA^t = \tfrac12 A^t_{uv} \sigma^u \wedge \sigma^v$ denote the 2-forms appearing on the right-hand sides
of \eqref{HSE} involving only wedge products of the $\sigma$'s.
\else
where the $\tfA^t = \tfrac12 Q^t_{ab} \pi^a \wedge \pi^b$ denote the 2-forms appearing on the right-hand sides
of \eqref{HSE} involving only wedge products of the $\pi$'s.
\fi
These 2-forms are $K$-semibasic and take value in $\Lambda^2 T^*_{1,0} N$. Thus, the $\piG$-semibasic forms are
$$
(\sysE\otimes\C)_{G,sb}=\{\ \thetau^i-\overline{\thetau^i},\ \eta^r,\ \tfA^t, \ \overline{ \eta^r},\overline{ \tfA^t }\  \}\subalg
$$
If $s:\Udown/G \to \Udown$ is a local cross-section for the action of $G$ on $N$, then  a local set of algebraic generators for $\sysI\otimes \C$  on $\Udown/G$ is (see Corollary 4.1 in \cite{AF2005})
$$
\sysI\otimes\C = \{\ s^*( \thetau^i-\overline{\thetau^i}),\ s^* \eta^r, \ s^* \overline{ \eta^r},\ s^* \tfA^t ,\ s^*\overline{ \tfA^t }\  \}\subalg
$$
In particular note that $s^* \tfA^t \in \Lambda^2 V$ by equation \eqref{quotientV}, which shows that part \ref{EDdefV} 
of Definition \ref{EDdef} holds.   
\ifarch
A local normal form for the structure equations of  $\sysI\otimes\C $ will be derived in Section \ref{SLNF}.
\fi

The fact that $(\sysI,\DD,\JJ_0)$ satisfies the Definition \ref{EDIdef} for Darboux integrability 
now follows from part \ref{Qsing} of this theorem,  so  we only need to prove that the singular system $\calV$ is Pfaffian.
Following an argument similar to Theorem 3.1 in \cite{AFV}%
\ifjour
\else 
(see also equation \eqref{IQGV} in  Corollary \ref{IVG}), 
\fi
the singular system $\calV$ of $\calI$ is the quotient of the singular system $\calW$ for $\calE$, i.e.,  $\calV = \calW/G$. 
From equation \eqref{DefW}, the singular system for $\sysE$ is the Pfaffian system generated by $W= T^*_{1,0}N\oplus\overline{H}$. The action of $G$ is transverse to $W^{(\infty)}=T^*_{1,0}N$ and therefore to the derived system $W'$, and thus $\calV = \calW/G$ is Pfaffian. 
\end{proof}

\begin{remark}\label{HVderived}   In the 
above proof  
we used the fact that  the action of $G$ is transverse to $W^{(\infty)}$. This implies that the action of $G$ is transverse to $W^{(k)}$ for all $k$ and therefore $V^{(k)}= W^{(k)}/G$ by Theorem 5.1 in \cite{AF2005}.  This yields the following relationship between the ranks of the derived flags of the holomorphic Pfaffian system $H$ and singular system $V$,
$$
\rank H^{(k)}= \rank V^{(k)}-\rank V^{(\infty)}  .
$$
\end{remark}

\newcommand\bfp{{\mathbf p}}
\medskip

\ifarch
The transversality of the group action  in Theorem \ref{C34} also implies, by Theorem 2.1 in \cite{AFB}, that
$\sysE$ is an {\defem integrable extension} of $\sysI$.
We now recall this concept and examine its role relating the integral manifolds of $\sysE$ and $\sysE/G$.

\begin{defn}\label{DefIE}
An integrable extension of an EDS $\sysI$ on $M$ is a submersion $\bfp :N \to M$, and an EDS $\sysE$ on $N$ and a sub-bundle  $F \subset T^*N$ satisfying the following conditions:
\begin{equation}
	\text{ [i] }  \rank  F = \dim N- \dim M, 
	\ \text{ [ii] }  F^\perp \cap \operatorname{ker} {\bfp}_* =0, 
	\ \text{ [iii] }  \sysE = \langle {\mathrm S}(F) \cup  \bfp^*(\sysI) \rangle_\text{alg},
\label{IntExtDef}
\end{equation}
where, in condition [iii], the sections of $F$ are denoted by ${\mathrm S}(F)$.

Condition [ii] means that $F$ is transverse to $\bfp$, and its rank is sometimes referred to as the {\defem fiber rank} of the extension.
Further properties  of integrable extension can be found in Section  2.1 of \cite{AFB} or Chapter 7 in \cite{CfB2}.
\end{defn}

In the context of Theorem \ref{C34}, to make $\sysE$ an integrable extension of $\sysE/G$,
we can follow equation (2.22) in Theorem 2.1 of \cite{AFB} and choose the bundle $F\subset E $ to be any complement to the sub-bundle of $\piG$-semibasic forms in $E$.

We now show that  holomorphic integral manifolds of $\sysH$ and the $\JJ_0$-invariant integral manifolds  of $\sysI$ in Theorem \ref{C34} coincide.  Equation \eqref{defJJ0} in Theorem \ref{C34}  shows that if $s:\Sigma\to N$ a holomorphic integral manifold of $\sysH$
then $\piG\circ s: \Sigma\to M$ is a $\JJ_0$-invariant integral manifold of $\sysI$. The next corollary shows that the converse is true locally.


\begin{cor}  \label{HoloIM}  
Let $s:\Sigma \to M$ be a $\JJ_0$-invariant integral manifold of $\sysI=\sysE/G$ from Theorem \ref{C34}. 
Given points $\npt\in N$ and $x\in \Sigma$ such that  $\piG(\npt)= s(x)$, there exists an open neighborhood $U \subset \Sigma$ of $x$ and a holomorphic integral manifold $\hat s:U\to N$ of $\sysH$ through $\npt$ such that $s=\piG \circ \hat s$.
\end{cor}
\begin{proof} Let $s: \Sigma  \to M$ be a $\JJ_0$-invariant integral manifold of $\sysI$ 
and let $\npt\in N$, $x\in \Sigma$ with $\piG(\npt)=x$. By the integrable extension property (see the second paragraph in Section  2.1 of \cite{AFB}) there exists an open set $U\subset \Sigma$ containing $x$,
and an integral manifold  $\hat s:U \to N$  of $\sysE$  (and hence of $\sysH$) through $\npt$ such that $s= \piG  \circ \hat s$ on $U$. 

We need to show that $\hat s_* T_x U$ is $\JJ$-invariant.  Let $X\in T_xU$; 
then by $\JJ_0$ invariance and $s = \piG\circ {\hat s}$, there exists an $X' \in T_x U$ such that
\begin{equation*}
\JJ_0({\piG}_*{\hat s}_*   X)= \JJ_0(s_* X)  = {s}_* X'.
\end{equation*}
Using equation \eqref{defJJ0} this becomes
$$
{\piG}_* \JJ( {\hat s}_*X ) = {\piG}_* ({\hat s}_* X').
$$
Since ${\piG}_*: \Delta_\npt \to \DD_{\piG(\npt)}$ in equation \eqref{QDD} is an isomorphism, we have $ \JJ({\hat s}_* X) =  {\hat s}_* X'$.
Therefore ${\hat s}:U \to N$ is $\JJ$-invariant and thus defines a holomorphic integral manifold of $\sysH$ with respect to the induced complex structure from $\JJ_0$ on $\Sigma$.
\end{proof}
\fi


\ifarch
\subsection{The Local Normal Form}\label{SLNF}

In this section we give a local normal form for the structure equations of the elliptic system $\sysI= \sysE/G$ in Theorem \ref{C34}.  The proof given here is the elliptic analogue of the algorithm given in Section 7.3 of \cite{AFB} which determined a normal form for DI hyperbolic systems. The first step requires the holomorphic analogue of Theorem 2.4 in \cite{AFB}  given as follows.

\begin{thm}\label{HICF}  Let $K,N, \rH$ be as in Theorem \ref{C34} and let $Z_i$ be the infinitesimal generators of the right action of $K$ on $N$ corresponding to the choice $Z_i^L$ for the Lie algebra of $K$ satisfying \eqref{RC}.  About each point $\npt \in N$ there exists a $K$-invariant open set $\Upset \subset  N$ and a local basis $(\, \thetau ^i,\,\eta^r,\, \sigma^u \,)$
of holomorphic sections of $T^*_{1,0} N\restr_{\Upset}$
(where $1 \le i,j\le \dimG$, $1 \le r,s\le \rankH -\dimG$, and $1\le u,v \le \dimN-\rankH$)
 satisfying the following conditions:
\begin{enumerate}[label=\alfenumstyle,ref=\alfenumstyle]
\item\label{lnfH} $H\vert_\Upset=\{ \eta^r, \thetau^i \}$ 
\item\label{lnfD} The 1-forms $\thetau^i$ satisfy $\thetau^i(Z_j) = \delta^i_j$. 
\item\label{lnfB} The 1-forms $ \eta^r, \, \sigma^u $ are $K$-basic.
\item\label{lnfA} For the action $\rho:N\times K \to N$, and $\Lambda^i_j$ in equation \eqref{OM1},
\begin{equation}
\rho_{a} ^*  \thetau^ i\vert_{x\cdot a} = \Lambda^i_j(a) \thetau^j\vert_{x}, \qquad a\in K, \, x\in \Upset.
\label{eqt}
\end{equation}
\item\label{lnfS} These forms satisfy the structure equations
\begin{equation}
\begin{aligned}
d \sigma^u & = & & 0 \ ,\\
d \eta^r   &=& & \tfrac12  E^r_{uv}\, \sigma^u \wedge \sigma^v  + F^r_{u s}\,\sigma^u\wedge  \eta^s  \, ,  \\
d\thetau^i & =& & \tfrac12 Q^i_{ab} \pi^a \wedge \pi^b  - \tfrac12 C^i_{jk} \thetau^j \wedge \thetau^k,
\end{aligned}
\label{HSE}
\end{equation}
with $\pi^a=(\eta^r, \sigma^u)$ as usual.
\end{enumerate}
\end{thm}

\begin{proof} As with Theorem 2.4 in \cite{AFB},  the proof  starts with constructing a bundle chart $\Psi:\Upset \to \Wdown \times K$ for $\piK:N\to N/K$ where $\Wdown=\piK(\Upset)$  and $\DI^a$ are holomorphic coordinates on $\Wdown$ pulled back to $\Upset$.  The transversality hypothesis implies  there exist holomorphic 1-forms $\thetau_0^i$ satisfying $\thetau^i_0(Z_j) =\delta^i_j$,  and 
holomorphic 1-forms 
\begin{equation} 
\begin{aligned}
\sigma^u & = & & d w^u \\
\eta^r &=& &  d \DI^{r+n-m}+R^r_u(\DI)  d \DI^u
\end{aligned}
\label{PBT}
\end{equation}
such that $H\vert_\Upset = \spanc\{\, \thetau^i, \, \eta^r \, \}$.  The $K$-basic $\sigma^u$ are chosen complementary to $H$ in $T^*_{1,0} \Upset$, while the $\eta^r$ span the $K$-basic 1-forms in $H$ (compare equation (2.35) in \cite{AFB}). At this point \ref{lnfH}, \ref{lnfD} and \ref{lnfB} are satisfied.

Now using the chart $\Psi$ we have $\Psi_* Z_i = Z_i^L$. Therefore there exist holomorphic functions $T^i_r$, $S^i_u$ on $\Upset$ such that 
$$
\thetau^i_0 =  \Psi^* (\mu^i_L) + T^i_r \eta^r + S^i_u \sigma^u.
$$
We define the 1-forms
\begin{equation}
\thetau^i  = \thetau^i_0 - T^i_r \eta^r  = \Psi^*( \mu^i_L)+  S^i_u \sigma^u, 
\label{DTC}
\end{equation}
(corresponding to equation (2.38) in \cite{AFB}) which still satisfy \ref{lnfH} and \ref{lnfD}.  
To begin to show \ref{lnfA}, we note that, from equation \eqref{OM2} and the $K$-equivariance of the map $\Psi$, we have
\begin{equation}
\rho_a^* \Psi^*  (\mu^i_L) = \Psi^* R_a^*( \mu^i_L) = \Psi^*( \Lambda^i_k(a)  \mu^k_L ) =\Lambda^i_k(a)  \Psi^* ( \mu^k_L)
\label{PBM}
\end{equation}
for $a\in K$.  Using equations \eqref{DTC} and \eqref{PBM},
$$
\rho_a^* \thetau^i\vert_{x\cdot a} = \rho_a^* \left( \Psi^*( \mu^i_L)+ S^i_u  \sigma^u \right)\vert_{x\cdot a} 
=\Lambda^i_k(a)  \Psi^* ( \mu^k_L )\vert_{x} + S^i_u(x \cdot a ) \sigma^u\vert_x \ .
$$
But since $H$ is $K$-invariant, then $\rho_a^*\thetau^i \in H$.  It follows that $S^i_u(x\cdot a) = \Lambda^i_k(a) S^k_u(x)$ and \ref{lnfA} is established.

For part \ref{lnfS}, the derivation of the first two structure equations in \eqref{HSE} follows from taking the exterior derivative in \eqref{PBT}. The last structure equation in \eqref{HSE} follows by taking the exterior derivative of $\thetau^i$ in \eqref{DTC} and using the argument in the last paragraph of the proof of Theorem 2.4 in \cite{AFB}.
\end{proof}






Using Theorem \ref{HICF}, we can now prove our sought-after result, a local normal form for the structure equations of the 1-forms in $\calI = \calE/G$.

\begin{thm}\label{DICF}  Let $K,N, \rH,G, \sysE$, $\sysI$ and $\rV$ be as in Theorem \ref{C34}.  About each point $\mpt \in N/G$ there exists an open set $\Udown$ and a 1-adapted coframe   $(\theta^i, \eta^r,\sigma^u ,\overline{\eta^r} ,\overline{\sigma^u})$ on $\Udown$  such that:
\begin{enumerate}[label=\thmenumstyle,ref=\thmenumstyle]
\item\label{P2CB} $\Vinf = \spanc\{ \eta^r, \sigma^u\}$,
\item \label{P2CF}  $ (\rI \otimes \C) \vert_\Udown= \spanc\{ \ \theta^i, \ \eta^r,\ \overline{ \eta^r} \ \}$  and $\theta^i\vert_\mpt$ are imaginary valued on (real) vectors in $T_\mpt M$,
\item \label{P2CS} the following structure equations hold:
\begin{equation}
\begin{aligned}
d \sigma^u & = & & 0 \qquad \\
d \eta^r&  = & &  \tfrac12   E^r_{uv}\, \sigma^u \wedge \sigma^v+ F^r_{u s}\, \sigma^u \wedge \eta^s   \\
d\theta^i& =& &  \tfrac12 A^i_{ab} \,  \pi ^a  \wedge \pi^b 
+ \tfrac12 P^i_j\overline{ A^i_{a b }}\,  \pib{a}  \wedge \pib{b}
 + \tfrac12 C^i_{jk}\, \theta^j \wedge \theta^k  + M^i_{aj} \pi^a \wedge \theta^j,
\end{aligned}
\label{SE2}
\end{equation}
where $C^i_{jk}$ are structure constants for the Lie algebra of $G$ and $\ \theta^i \equiv P^i_j \overline{\theta^j} \mod \veta, \vetabar \ $ defines $P^i_j$.
\item \label{P2CT}
The coefficients $ A^i_{ab}, M^i_{aj}, E^r_{u v}, F^r_{u s} \in \Invts$.
\end{enumerate}
\end{thm}

Theorem \ref{DICF} is the elliptic analogue of the steps outlined in Section 7.3 in \cite{AFB} which produces a  4-adapted coframe for a hyperbolic Darboux integrable system.


\begin{proof} 
The proof consists of utilizing Theorem \ref{HICF} to construct $G$-basic 1-forms  on a $K$-invariant open set $\Upset \subset N$  which satisfies the conditions of the theorem.  By way of an abuse of notation we then identify the $G$-basic forms on $\Upset$  as forms on the open set $\Udown=\piG(\Upset) \subset N/G$.

\medskip
\noindent {\bf Part \ref{P2CB}:} Let $\mpt \in M$ and let $\npt \in N$ with $\piG(\npt)=\mpt$.  Choose a $K$-invariant open set $\Upset$ containing $\npt$ such that $\Psi:\Upset \to \Wdown \times K$ is a local trivialization of $\piK:N\to N/K$ as a principal $K$-bundle with  $\Psi(\npt) = (w, e)$. 
Let  $\Udown= \piG(\Upset)$.
As in the proof of Theorem \ref{HICF}  let $(\thetau^i,\eta^r,\sigma^u)$ be the local basis of sections for $T^*_{1,0}N$ defined on $\Upset$ which, along with the conjugates $\overline{\eta^r,}\ \,  \overline{\sigma^u}$, form a coframe on $\Upset$.  
The 1-forms $\sigma^u, \eta^r$ are holomorphic and $K$-basic by part \ref{lnfB} of that theorem, and by virtue of being $K$-basic, define a local basis of sections of $T^*_{1,0} \Wdown $. Viewing these as forms on $\Udown$ (with our convention of suppressing pullback under $\phi$), equation \eqref{quotientVinf} in Theorem \ref{C34} finishes the proof of part \ref{P2CB}.

\medskip
\noindent {\bf Part \ref{P2CF}:} 
 Since $\eta^u, \overline{\eta^u} \in \sysE \otimes \C\vert_\Upset$  are $K$-basic, they are also $G$-basic, and by construction  
 (and abuse of notation)
 $\eta^u, \overline{\eta^u} \in \sysI\otimes \C\vert_\Udown$.  
We now need to find $G$-basic forms $\theta^i \in \sysE\otimes\C\vert_\Upset$ satisfying equations \eqref{SE2}. We claim that the 1-forms
\begin{equation}\label{thetaown}
\theta^i =O^i_j (\thetau^j- \overline{\thetau^j}),
\end{equation}
 satisfy these conditions, where $O^i_j = \Psi^* \Omega^i_j$ and $\Omega^i_j$ is defined in \eqref{OM2}. Since $\thetau^i \in \sysE \otimes \C\vert_\Upset$, then $\theta^i \in \sysE\otimes\C\vert_\Upset$, and we now check that they are $G$-basic.

First note that (see equation \eqref{GIG}), $X_j \hook (\thetau^i - \overline{\thetau^i}) = \delta^i_j - \overline{\delta^i_j} = 0$, so the 1-forms $\theta^i$  are semibasic for the projection $\piG: \Upset \to \Udown$.
We also note that since $G$ is real then $\Lambda^i_j(g)$ and $\Omega^i_j(g)$ are real for $g\in G$.
Then using equation \eqref{eqt} and $O^i_j(x\cdot g) = O^i_\ell(x) \Omega^\ell_j(g)$ for $x \in \Upset$ we compute
$$
\begin{aligned}
\rho_{g}^* \theta^i \vert_{x\cdot g} &=& & O^i_j(x\cdot g) \,  \rho_{g}^* (\thetau^j- \overline{\thetau^j}) \vert_{x\cdot g} \\
& =& &  O^i_\ell(x)\Omega^\ell_j(g)  \Lambda^j_k(g) (\thetau^k- \overline{\thetau^k})\vert_{x} \\
& =& &  O^i_k(x)  (\thetau^k- \overline{\thetau^k})\vert_{x}  = \theta^i\vert_x
\end{aligned}
$$
Hence the $\theta^i$ are $G$-basic and the first part of \ref{P2CF} is proved. 
Now $\theta^i\vert_\mpt= O^i_j(\npt) (\thetau^j- \overline{\thetau^j})\vert_\mpt =(\thetau^i - \overline{\thetau^i})\vert_\mpt$, since $O^i_j(\npt) = \Psi^* \Omega^i_j\vert_{(w,e)} =\delta^i_j$. This shows 
that $\theta^i\vert_\mpt$ are pure imaginary-valued.

\medskip
\noindent {\bf Part \ref{P2CS}:}
By equation \eqref{HSE}  the first two equations in \eqref{SE2} are satisfied.
We now derive the final equation in \eqref{SE2}. Using condition \ref{lnfD} in Theorem \ref{SLNF} in the form  of equation \eqref{DTC}, we have $\Psi^*\mu_L^i = \thetau^\ell + L^\ell_a  \pi^a$ for some holomorphic functions $L^\ell_a$; hence by \eqref{dLam} we have
$$
d O^i_j = O^i_k C^k_{\ell j}(\thetau^\ell +L^\ell_a     \pi^a).
$$
Using this and equation \eqref{HSE}, we compute
\begin{multline*}d\theta^i = dO^i_j \wedge (\thetau^j - \thetaubj) + O^i_j (d\thetau^j - d\thetaubj) \\
				= O^i_k C^k_{\ell j} (\thetau^\ell + L^\ell_a  \pi^a) \wedge  (\thetau^j - \thetaubj) 
				-\tfrac12 O^i_j C^j_{k\ell} \left(\thetau^k \wedge \thetau^\ell - \thetaubk \wedge \thetaubl\right)  \\
				+ O^i_j (Q^j_{ab} \pi^a \wedge \pi^b -\overline{Q^j_{ab}} \pib{a} \wedge \pib{b}).
\end{multline*}
For the sake of convenience, set $\alpha^i = \thetau^i + \thetaubi$ and $\beta^i = \thetau^i - \thetaubi$.  Then modulo the $\pi^a$ and $\pib{a}$, 
\begin{align*}
d\theta^i \equiv&
 \tfrac12 O^i_k C^k_{\ell j} (\alpha^\ell + \beta^\ell) \wedge \beta^j \\
&- \tfrac18 O^i_j C^j_{k \ell} \left((\alpha^k+\beta^k) \wedge (\alpha^\ell + \beta^\ell) - (\alpha^k - \beta^k) \wedge (\alpha^\ell -\beta^\ell) \right) \\
=&\tfrac12 O^i_k C^k_{\ell j} \beta^\ell \wedge \beta^j
+\tfrac12 O^i_k C^k_{\ell j} \alpha^\ell \wedge \beta^j
-\tfrac14 O^i_j C^j_{k\ell} \left( \alpha^k \wedge \beta^\ell + \beta^k \wedge \alpha^\ell\right)\\
=&\tfrac12 C^i_{km} O^k_\ell O^m_j \beta^\ell \wedge \beta^j = \tfrac12 C^i_{km}\theta^k\wedge \theta^m.
\end{align*}
Thus,
\begin{equation}
d\theta^i =O^i_j (Q^j_{ab} \pi^a \wedge \pi^b -\overline{Q^j_{ab}} \pib{a} \wedge \pib{b})
+\tfrac12 C^i_{jk} \theta^j \wedge \theta^k +(O^i_kC^k_{\ell j} L^\ell_a )  \pi^a \wedge \theta^j,
\label{dvt}
\end{equation}
which matches the final structure equation in \eqref{SE2} with $P^i_k=-O^i_j(\overline{O}^{-1})^j_k$.

\medskip
\noindent {\bf Part \ref{P2CT}:}  Let $(\vT_i, \vN_r, \vS_u, \overline{\vN_r}, \overline{\vS_u})$ be the dual coframe. Taking the exterior derivative in last of equation \eqref{SE2} we have $\overline{\vS_u}(A^i_{ab})=0$ and $\overline{\vN_r}(A^i_{ab})=0$. Therefore by Lemma \ref{DIX}, $A^i_{ab} \in \Invts$. A similar argument shows that $ M^i_{aj}, E^r_{u v}, F^r_{u s} \in \Invts$.

\end{proof}



In the next corollary we continue the practice of identifying $G$-basic forms on $N$ with forms on $N/G$. 
\begin{cor}\label{IVG} In terms of the coframe given by Theorem \ref{DICF}, we have the following local algebraic generators for $\sysI\otimes \C$ and its singular system $\calV$:
\begin{subequations}
\begin{align}
(\sysI \otimes \C )\vert_\Udown=
((\sysE \otimes \C)/G )\vert_\Udown
 &=\{\ \theta^i,\ \eta^r,\ \overline{ \eta^r },\ d\eta^r ,\ d \overline{\eta ^r}, \ \tfA^i ,\ \overline{\tfA^i}\ \}\subalg \label{IQGI} \\
\calV\vert_\Udown  & = \{\ \theta^i,\ \eta^r\ , \sigma ^u  ,\ \overline{ \eta^r },\ d \overline{\eta ^r}, \ \overline{\tfA^i} \   \}\subalg =(\calW /G) \vert_\Udown. \label{IQGV}
\end{align}
\end{subequations}
where $\calW$ is the singular system of $\sysE$, and from equations \eqref{SE2} and \eqref{dvt} we have
\begin{equation}
\tfA^i= A^i_{ab} \pi^a\wedge \pi^b  = O^i_j Q^j_{ab} \pi^a \wedge \pi^b.
\label{TFG}
\end{equation}
Furthermore, the 2-forms $\tfA^i$ are $K$-basic.
\end{cor}

\begin{proof}  Part \ref{P2CF} of Theorem \ref{DICF} shows that the  forms $\theta^i, \eta^r, \overline{\eta^r} $ and their derivatives lie in $(\sysI\otimes \C)\vert_\Udown$. 
Due to the structure equation \eqref{dvt} we need only show that $\tfA^i \in (\sysI \otimes \C )\vert_\Udown$ and that all the 2-form generators of $\sysI \otimes \C$ are accounted for. According to the reduction algorithm in \cite{AF2005},  the 2-forms in $(\sysE\otimes \C)/G$ are determined algebraically by the $\piG$-semibasic 2-forms in $\sysE\otimes \C$. We need only identify the semibasic 2-forms in $\calE\otimes \C$ which can be expressed 
as a linear combination of $\pi^a \wedge \pi^b, \pi^a \wedge \overline{\pi^b},  \overline{\pi^a} \wedge \overline{\pi^b}$,
since any other terms involving $\theta^i,\eta^r, \overline{\eta^r}$ are already included in the right-hand side of \eqref{IQGI}.
 
The 2-forms  $Q^i_{ab} \pi^a \wedge \pi^b \in \sysE\otimes \C\vert_\Upset $ in equation \eqref{HSE} are $\piK$-semibasic (and hence $\piG$-semibasic) so determine elements of $\sysI\otimes\C\vert_\Udown$ (again see Proposition 4.2  \cite{AF2005}). In particular the 2-forms $\tfA^i$ in equation \eqref{TFG} are an invertible combination of these and are $K$-basic (using equations \eqref{DTC},\eqref{TFG}). The only other $\piG$-semibasic forms not accounted for are $\overline{Q^i_{ab}\pi^a \wedge \pi^b} \in \sysE\otimes \C\vert_{\Upset}$ but these determine forms in the span of $\overline{\tfA^i}$ which 
is contained in right-hand side of \eqref{IQGI}.

The generators of the singular system $\calV$ given in equation \eqref{IQGV} are obtained from $\calI\otimes \C$ using Lemma \ref{LocIV}. Notice that this shows by \eqref{SE2} that $\calV$ is Pfaffian. Since $\calW$ is Pfaffian and the action of $G$ is transverse to the derived system $W'$, then $\calW/G$ is  Pfaffian (by Theorem 5.1 in \cite{AF2005}). Furthermore the bundles satisfy $V = W/G$ (by part \ref{Qsing} in Theorem \ref{C34} above), which implies that $\calV = \calW/G$. 
\end{proof}
\fi



\ifarch
\bigskip
A coordinate chart description of the quotient construction of $\sysI$ is given by the diagram in equation \eqref{CD3b} of Corollary \ref{CircleD}.
\fi

\def\BenInv{\xi}
\def\Pbar{\overline{P}}
\def\BenInvbar{\overline{\xi}}
\def\Wobar{\overline{W_1}}
\def\Sup{\widehat{\Slice}}
\subsection{Examples}\label{SecondEx}

The examples in this section demonstrate Theorem \ref{C34} by starting with a holomorphic Pfaffian system $\sysH$, with \realform\ $\sysE$, and a symmetry group $K$ satisfying the hypotheses of the theorem. Our examples also satisfy the condition that the action of the real form $G$ is transverse to derived system $\sysE'$, so that by Theorem 5.1 in \cite{AF2005} the quotient $\sysI=\sysE/G$ is a Pfaffian system.  We then compute the quotient by finding $\Dplus ={\piG}_* (\Deltaplus)$ (see equation \eqref{defDD}) from which $I\otimes \C =(\Dplus\oplus \overline{\Dplus})^\perp$ and hence $\sysI$ is then determined. 

\ifarch
\begin{example}\label{EB2}
Let $N=\Jhol{2}(\C,\C)$ denote the space of 2-jets of holomorphic functions of one complex variable.
On this complex 4-manifold we take holomorphic coordinates $z,\ws0,\ws1,\ws2$ where $z$ stands for the independent variable and $\ws0,\ws1,\ws2$ for the dependent variable and its $z$-derivatives.  In terms of these, the holomorphic contact system on $N$ is
\begin{equation}\label{defj2holcontact}
\sysH = \{ \zeta_0:=d \ws0 - \ws1 \dz, \zeta_1:=d \ws1 - \ws2 \dz\}\subdiff.
\end{equation}
As in Prop. \ref{HDI}, we let $\sysE$ denote the real Pfaffian system generated by the real and imaginary
parts of the 1-forms in $\sysH$, which we call the realification of $\sysH$. We also have from equation \eqref{DefDD}, 
\begin{equation}
\Deltaplus  = T_{1,0}N  \cap \rH^\perp  
=\spanc\{ \vD_z := \dib{z} + \ws1 \dib{\ws0} + \ws2 \dib{\ws1}, \, \dib{\ws2} \}.
\label{deltae1}
\end{equation}

Consider the holomorphic vector fields on $N$ given by 
$$
\vZ_1 = \dib{\ws0}, \quad \vZ_2 = z \dib{\ws0} + \dib{\ws1},
$$
which are easily checked to be infinitesimal symmetries of $\sysH$. As these commute, they generate an action of the two-dimensional Abelian Lie group $K=\C^2$ on $N$. Taking coordinates $c_1, c_2$ on $K$, the $K$-action is computed to be
\begin{equation}\label{abelKNact}
(z,\ws0,\ws1, \ws2) \cdot (c_1, c_2) = (z,  \ws0 + c_1 +c_2 z, \ws1+ c_2, \ws2).
\end{equation}
Since
$$\begin{pmatrix}\vZ_1 \\ \vZ_2\end{pmatrix} \hook (\zeta_0, \, \zeta_1) =
\begin{pmatrix} 1 & 0 \\ z & 1 \end{pmatrix},$$
the action is transverse to $H$ and hence also to $E$ by part \ref{qcrank} of Theorem \ref{C34}. The real vector fields $\vX_i = \vZ_i + \overline{\vZ_i}$ likewise form an Abelian algebra generating the action of the real form $G \subset K$ obtained by restricting $c_1, c_2$ in equation \eqref{abelKNact} to be real.  

The derived system $\sysE'$ is generated by the real and imaginary parts of 1-form $\zeta_0$, and 
$$\begin{pmatrix}\vX_1 \\ \vX_2\end{pmatrix} \hook (\Re \zeta_0, \Im  \zeta_0) =
\begin{pmatrix} 1 & 0 \\ \Re  z & \Im z \end{pmatrix}$$
shows that the action of $G$ is transverse to $\sysE'$ on the $K$-invariant open set $\Upset$ defined by $\Im z\neq 0$.
Then $\sysI=\sysE/G$ is a Pfaffian system on the open set $\Udown = \piG(\Upset)$
in the 6-dimensional manifold $M = N/G$.  By Theorem \ref{C34} part \ref{qE}, $\sysI$ is Darboux integrable. Because $\dimG=\dim K$ and 
$\rankH = \rk H$ are both equal to 2, the local normal form given by Theorem \ref{DICF} includes no $\eta$-forms, and the integrability is {\em minimal} in the sense of Definition \ref{maxmindef}.

To find coordinates on $\Udown$ we start by noting that the function 
\begin{equation}\label{Wabelform}
U = \ws1 -\dfrac{\ws0-\wbars0}{z-\zbar}.
\end{equation}
 is $G$-invariant, and hence is well-defined on $\Udown$.  Since the holomorphic vector field
$\vD_z$ 
on $N$ is $K$-invariant (and hence $G$-invariant) then
$$U_z=\vD_z U = \ws2 - \dfrac{\ws1}{z-\zbar} + \dfrac{\ws0-\wbars0}{(z-\zbar)^2}$$
is also $G$-invariant.  In fact $z,U, U_z$ and their conjugates give a well-defined coordinate system on $\Udown$. 
Using equations \eqref{defDD} and \eqref{deltae1}, we have in terms of these coordinates
$$
\Dplus ={\piG}_* (\Deltaplus)= \{ \partial_z + U_z \partial_U + \frac{U}{\overline{z}-z}\partial_{\overline{U}} +
\frac{U+\overline{U}}{(\overline{z}-z)^2}\partial_{\overline{U_z}}  , \partial_{U_z} \, \}.
$$ 
From this we obtain $I\otimes\C =(\Dplus\oplus \overline{\Dplus})^\perp=\{\, \beta,\, \overline \beta\, \} $ where
$$
\beta = dU - U_z \, dz - \dfrac{\Ubar}{z-\zbar}d\zbar .
$$
It is clear from the form of $\beta$ that integral surfaces for 
$\sysI$ on the set $\Udown$  that satisfy the independence condition $dz \wedge d\zbar \ne 0$ are 1-graphs of solutions of the PDE
\begin{equation}
\dfrac{\di U}{\di \zbar} =\dfrac{\Ubar}{z-\zbar}.\label{Wabeleqn}
\end{equation}
(If we decompose $U$ and $z$ into their real and imaginary parts, this PDE is equivalent to a
real first-order system for two functions of two variables.) 
A contact-equivalent form of equation \eqref{Wabeleqn} was shown to be Darboux integrable in \cite{Eendebak} (see Example 8.3.3).
As the integrability is minimal, there should be $\rankVinf=\rankDplus=2$ independent holomorphic Darboux invariants, which
 are characterized by being invariant under vector fields in $\Dminus = \overline{\Dplus}$.  For systems such as this one, arising as a quotient,  by Remark \ref{KDINV} we can generate these invariants from holomorphic functions on $N$ that are $K$-invariant. Besides $z$, a second
such holomorphic function is given by $w_2$, and this descends to $\Udown$ to give the Darboux invariant
\begin{equation}
\xi =  U_z + \frac{U}{z-\overline{z}} .
\label{QDI1}
\end{equation}

As remarked after Theorem \ref{FTDI} in the Introduction, because $\sysI$ is a quotient of $\sysE$
we can generate solutions of \eqref{Wabeleqn} from those of $\sysE$.  In fact, $\sysE$ is simply the prolongation of the Cauchy-Riemann equations, and so integral surfaces of $\sysE$ satisfying the independence condition $dz \wedge d\zbar \ne 0$ are given by 
\begin{equation}\label{twojetparm}
\ws0 =f(z), \qquad \ws1 = f'(z),\qquad \ws2 = f''(z)
\end{equation}
for an arbitrary holomorphic function $f(z)$.  Substituting these into the formula \eqref{Wabelform} gives
the solution formula
$$U = f'(z) - \dfrac{\Im f(z)}{\Im z}$$
for \eqref{Wabeleqn}.

\end{example}
\fi

\bigskip

\begin{example}\label{benexample2}
In this example, we will realize equation \eqref{beneqncomplex} by a Darboux integrable Pfaffian system $\sysI$ arising
as the quotient of the \realform\ $\sysE$ of a holomorphic
Pfaffian system by the transverse action of a symmetry group $G$.  
\ifjour
\else
Again, we will use Remark \ref{KDINV} to recover the holomorphic Darboux invariants given in Example \ref{benexample1}, 
and we will also demonstrate the calculation of the local normal form for $\sysI$ provided by Theorem \ref{DICF}.
\fi

\ifjour
\else
As in Example \ref{EB2}, let $N=\Jhol{2}(\C,\C)$, on which we have holomorphic coordinates $z,\ws0,\ws1,\ws2$, the holomorphic contact system $\sysH$ defined by \eqref{defj2holcontact}, and its \realform\ $\sysE$.
\fi

Let $K$ be the complex Lie group of matrices
\begin{equation}\label{solvableKdefinition}
K = \left\{ \begin{pmatrix} c_1 & 0 \\ c_2 & 1 \end{pmatrix} \bigg\vert c_1, c_2 \in \C, c_1\ne 0\right\},
\end{equation}
and let $K$ act on the right on $N$ by contact transformations
\begin{equation}\label{bexKNact}
(z,\ws0,\ws1, \ws2) \cdot (c_1, c_2) = (z, c_1 \ws0 + c_2, c_1 \ws1, c_1 \ws2)
\end{equation}
which are symmetries of $\sysH$.  (Below we will let $\rho: N\times K \to N$ denote this action.)

The infinitesimal generators of the $K$-action are holomorphic vector fields on $N$ given by 
\begin{align*}
\vZ_1 &= \ws0\dib{\ws0} + \ws1 \dib{\ws1} + \ws2 \dib{\ws2}\\
\vZ_2 &= \dib{\ws0}.
\end{align*}
Let $G \subset K$ be the subgroup where $c_1,c_2\in \R$.  
The vector fields $\vX_i = \vZ_i + \overline{\vZ_i}$ are the 
infinitesimal generators of an action of $G$ on $N$ defined by restricting $c_1,c_2$ to be real in \eqref{bexKNact} and these form an algebra of real symmetry vector fields of $\sysE$. Since 
$$\begin{pmatrix}\vZ_1 \\ \vZ_2\end{pmatrix} \hook (\zeta_0, \, \zeta_1) =
\begin{pmatrix} \ws0 & \ws1 \\ 1 & 0 \end{pmatrix},$$
then $K$ acts transversely to $\sysH$  provided that $\ws1 \ne 0$.  Furthermore, the action of $G$ is
transverse to the derived system $\sysE'$ when $\Im \ws0 \neq 0 $.
We compute $\sysE/G$ by restricting to the $G$-invariant open subset $\Upset \subset N$ where these two transversality conditions hold.

Let $\piG:N \to M = N/G$ be the quotient map, and let $\Udown = \piG(\Upset)$.
It is easy to check that 
\begin{equation}
\label{beninvariantW}
W = \dfrac{2\ws1}{\wbars0 - \ws0}
\end{equation}
is $G$-invariant.    
\ifarch
As in Example \ref{EB2}, we can apply the $K$-invariant vector field
\else 
We can apply the $K$-invariant vector field
\fi
$
\vD_z = \dib{z} + \ws1 \dib{\ws0} + \ws2 \dib{\ws1}
$
to obtain an additional $G$-invariant function
\begin{equation*}
W_1 =\vD_z W=  \dfrac{2\ws2}{\wbars0 - \ws0} + \dfrac{2\ws1^2}{(\wbars0 - \ws0)^2}.
\end{equation*}
Along with their conjugates, $z, W, W_1$ provide a well-defined coordinate system on $\Udown$.
\def\Wobar{\overline{W_1}}

Let $\Deltaplus$ be the holomorphic distribution annihilated by $H$, defined by \eqref{deltae1}.
Then in our coordinates  
$$
\Dplus ={\piG}_* (\Deltaplus) = 
\spanc \{ \di_z + W_1 \di_W + \tfrac12 |W|^2 \di_{\Wbar} 
+ W\left ( \tfrac14 \Wbar^2 +\tfrac12 \overline{W_1}\right) \di_{\Wobar}, \di_{W_1} \}
$$ 
It follows that $I\otimes\C =(\Dplus\oplus \Dminus)^\perp=\{\, \beta,\, \overline \beta\, \} $ where
\begin{equation} \label{defbeta}
\beta := dW - W_1 \dz - \tfrac12 |W|^2 d\zbar. 
\end{equation}
It is clear from the form of $\beta$ that $\sysI$ encodes the first-order complex PDE \eqref{beneqncomplex}.  Because it arises as a quotient satisfying Theorem \ref{C34}, we recover the fact that $\sysI$ is Darboux integrable.  The second independent holomorphic Darboux invariant on $\Udown$, besides $z$, is obtained from the $K$-invariant holomorphic function $\ws2/\ws1$.  In terms of our $G$-invariants $W, W_1$ this is given by 
\begin{equation}\label{QDI2}
\xi=\dfrac{W_1}{W} - \dfrac{W}{2}
\end{equation} 
which agrees with the Darboux invariant given in equation \eqref{BenRinv}.

Again, since $\sysI$ is a quotient of $\sysE$,
\ifjour
\else
we can generate solutions of \eqref{beneqncomplex} from integral surfaces of $\sysE$ (parametrized as in \eqref{twojetparm}) 
provided that they also satisfy our regularity assumptions.  
Specifically, if $f(z)$ is a holomorphic function satisfying $\Im f(z) \ne 0$ and $f'(z) \ne 0$ on some open domain $D \subset \C$, 
then by substituting $w_0=f(z)$ into equation \eqref{beninvariantW} gives the general
\fi
solution to equation \eqref{beneqncomplex} as 
$$W = \ri f'(z)/\Im f (z).$$


\ifarch
We now demonstrate Theorem \ref{DICF} and construct the normal form for the structure equations of $\sysI$ by first computing 1-forms in $\sysH$ that are dual to vector fields $\vZ_1, \vZ_2$ as in the proof of Theorem \ref{HICF}:
$$\vartheta^1=w_1^{-1}(dw_1 - w_2\dz), \quad \vartheta^2 = dw_0 - w_1 \dz - (w_0/w_1)(dw_1 - w_2 \dz).$$
We obtain the chart $\Psi: \Upset \to \Wdown \times K$ used in the proofs of Theorem \ref{HICF} and \ref{DICF} by identifying the 
orbit space $\Wdown = \Upset/K$ with the slice $\Sup \subset N$ where $w_0=\ri$, $w_1=1$ and $w_2=\xi$, and then letting $\Psi$
be the inverse of the restriction of the right $K$-action $\rho$ to $\Sup \times K$; in coordinates this is given by
$$
\quad z=z, \quad \xi = w_2/w_1 , \qquad c_1 = w_1, \quad c_2 = w_0 -\ri w_1 .
$$
It is easy to check that 
$$\Psi_* \vZ_1 = c_1\dib{c_1}+c_2\dib{c_2}, \quad \Psi_* \vZ_2 = \dib{c_2},$$
the latter being left-invariant holomorphic vector fields on $K$.  Their duals are left-invariant (1,0)-forms 
$$\muL^1 = c_1^{-1}dc_1, \qquad \muL^2 = dc_2 - (c_2/c_1) dc_1,$$
while the right-invariant (1,0)-forms that coincide with these at the identity in $K$ are
$$\muR^1 = c_1^{-1} dc_1, \qquad \muR^2 = c_1^{-1} dc_2.$$
As in equation \eqref{thetaown} in the proof of Theorem \ref{DICF}, we let $\theta^i = (\Psi^* \Omega^i_j) (\vartheta^j - \overline{\vartheta^j})$, where $\Omega$ is the matrix as in \eqref{OM2} such that $\muR^i = \Omega^i_j \muL^j$.  This yields the $G$-basic 1-forms
\begin{equation}\label{ben4adaptup}
\begin{aligned}
\theta^1 &= \dfrac{1}{\ws1}(d\ws1-\ws2 \dz) - \dfrac{1}{\wbars1}(d\wbars1-\wbars2 \dz),\\
\theta^2 &=  \dfrac{1}{\ws1}(d\ws0 - \ws1\dz - d\wbars0 + \wbars1 d\zbar) 
+ \left(\dfrac{\wbars0 - \ws0}{\ws1 \wbars1}\right) (d\wbars1 - \wbars2\,d\zbar) -\ri \theta^1.
\end{aligned}
\end{equation}
The coframe given by Theorem \ref{DICF} is completed by closed 1-forms 
$\sigma^1=dz$, $\sigma^2 = d\BenInv$ and their conjugates.  (Since $\rank H$ = $\dim K$ there are no $\eta$ 1-forms in this coframe.)  
\fi

We will later revisit this example in \S\ref{MoreEx}, reconstructing the integrable extension $\sysE$ from the
system $\sysI$.
\end{example}


\bigskip
The next example shows that for a given $N$ and $\sysH$, and complex Lie group $K$ acting on $N$ preserving $\sysH$, inequivalent real forms of $K$ lead to 
inequivalent Darboux integrable quotient systems.

\begin{example}\label{Liou12}
Now we take $N = \Jhol{3}(\C,\C)$ with complex coordinates $z,w_0,w_1,w_2,w_3$, and define the 
holomorphic contact system 
$$\sysH = \{ dw_0 - w_1 \dz, dw_1 - w_2\dz, dw_2 - w_3\dz \}\subdiff$$
with \realform\ $\sysE$.
On $N$, define the holomorphic vector fields
\begin{equation}\label{ZonN}
\begin{aligned}
\vZ_1 &= \dib{w_0}\\
\vZ_2 &= w_0\dib{w_0} + w_1 \dib{w_1} + w_2 \dib{w_2} + w_3 \dib{w_3}\\
\vZ_3 &= \tfrac12 w^2 \dib{w_0} + w w_1 \dib{w_1} + (w_0 w_2 + w_1^2) \dib{w_2} + (w_0 w_3 +3w_1 w_2)\dib{w_3},
\end{aligned}
\end{equation}
which are infinitesimal symmetries of the contact system, and satisfy
$$[\vZ_1,\vZ_2] = \vZ_1, \quad [\vZ_1,\vZ_3] = \vZ_2, \quad [\vZ_2, \vZ_3]=\vZ_3,$$
matching the structure constants of $SL(2,\R)$.
Thus, the $\vZ_i$ form a Lie algebra which generate a local holomorphic action of $K=SL(2,\C)$ on $N$, while the real vector fields
$\vX_i = \vZ_i + \overline{\vZ_i}$ form a real subalgebra arising from the action of the non-compact real form $G = SL(2,\R)\subset SL(2,\C)$ on $N$. The $K$-action is transverse to $\calH$ when $w_1\neq 0$ while
the $G$-action is transverse to $\sysE'$ at points where $\Im w_0 \ne 0, w_1\neq 0$. Accordingly, we let
$\Upset \subset N$ denote the open subset where $\Im w_0 \ne 0, w_1\neq 0$, and  let $\Udown = \piG(\Upset)$ .


%

The $K$-action preserves holomorphic functions as well as the contact
system $\sysH$, and it also commutes with differentiation by the total derivative vector field
$$\vD_z = \dib{z} + w_1 \dib{w_0} + w_2 \dib{w_1} + w_3 \dib{w_2}  $$
as well as with its complex conjugate $\vD_{\zbar}$, 
\ifjour
\else
as in the previous examples.
\fi
The lowest-order $G$-invariant functions on $\Upset$ are $z$ and the real-valued function
$w_1 \overline{w_1}/(w_0-\overline{w_0})^2$, and we can generate higher-order invariants by applying $\vD_z$ and $\vD_{\zbar}$.
If we let 
\begin{equation}\label{Uplusformula}
U = \ln\left( - 4w_1 \overline{w_1}/(w_0-\overline{w_0})^2\right) = \ln\left( |w_1|^2/(\Im w_0)^2\right),
\end{equation}
one can check that 
\begin{equation}\label{UplusLiouvillez}
\vD_{\zbar} (\vD_z U) =\tfrac12 e^U.
\end{equation}
For a given function $f$ on $\Jhol{3}(\C,\C)$, $\vD_z f$ gives a function that
restricts to any integral surface of the contact system $\calE$ to give the $z$-derivative of the restriction of $f$ (and similarly for the $\zbar$-derivative).
Thus, it follows from \eqref{UplusLiouvillez}
that the image of an integral surface of $\calE$ under $\piG$ gives a solution to the minus form of the
elliptic Liouville equation, 
\begin{equation}\label{ellipticLiouvillePlus}
\Delta U -2 e^U=0.
\end{equation}  (Note that if $z=x+\ri y$ then $\Delta=\dib{x}^2 + \dib{y}^2 = 4 \dib{z} \dib{\zbar}$.)  It follows from Theorem \ref{C34} that the system $\calI = \calE/G$ encoding this elliptic Liouville equation is Darboux integrable.   Moreover, setting $w_0=f(z),w_1=f'(z)$ for holomorphic $f$ in \eqref{Uplusformula} yields the
second solution formula in \eqref{ELpm}.

Finally we use Remark \ref{KDINV} to compute the Darboux invariants. The lowest-order  holomorphic $K$-invariants  on $N$ are $z$ and the Schwarzian derivative
$$
\xi = \dfrac{w_3}{w_1} - \dfrac32 \dfrac{w_2^2}{w_1^2}.
$$
When $\xi$ is expressed in terms of the derivatives of $U$ we get,
\begin{equation}
\xi = U_{zz} - \tfrac12 (U_z)^2,
\label{QDI3}
\end{equation}
which is the second independent holomorphic Darboux invariant for the elliptic Liouville equation \eqref{ellipticLiouvillePlus}.  

\medskip

Consider now a {\em different} basis $\vW_i$ for the same complex Lie algebra $\mathfrak{sl}(2,\C)$, given in terms of the vector fields $\vZ_i$ in equation \eqref{ZonN} on $N$ by
$$\vW_1 = \tfrac12 \vZ_1 + \vZ_3, \quad \vW_2 = \ri \vZ_2, \quad \vW_3 = \ri (\tfrac12 \vZ_1 - \vZ_3).$$
This basis satisfies bracket relations with structure constants matching those of $\mathfrak{su}(2)$:
$$[\vW_1,\vW_2] = \vW_3, \quad [\vW_1, \vW_3] = -\vW_2, \quad [\vW_2, \vW_3] = \vW_1.$$
Again, the real vector fields $\vW_j + \overline{\vW_j}$ form a (real) subalgebra with the same structure constants, and these
arise from the action of the compact real form  $\widetilde G= SU(2) \subset SL(2,\C)$ on $N$.
This action of $SU(2)$ is transverse to $\calE'$ on the open subset where $w_1 \ne0$, so we now let $\Upset \subset N$ denote this set, and let $\Udown = \pi_{\widetilde G}(\Upset)$.

The lowest order $\widetilde G$-invariants are $z$ and $|w_1|/(1+|w|^2)$.  If we let 
\begin{equation}\label{UminusLiouvillez}
U = \ln \left( 4|w_1|^2/(1+|w|^2)^2\right)
\end{equation}
then proceeding as before we find that
$$\vD_{\zbar} (\vD_z U) = - \tfrac12 e^U$$
while $z$ and $\xi = U_{zz}-\frac{1}{2}U_z^2$ generate the Darboux invariants.
So, the image of an integral surface of $\calE$ under $\pi_{\widetilde G}$ gives a solution 
of $\Delta U + 2e^U=0$, the plus form of the elliptic Liouville equation.  In particular, setting $w_0=f(z),w_1=f'(z)$ for holomorphic $f$ in \eqref{UminusLiouvillez} yields the first solution formula in \eqref{ELpm}.

Moreover, the corresponding Pfaffian
system $\widetilde \sysI=\sysE/\widetilde G$ is not contact-equivalent to the system $\sysI$ encoding \eqref{ellipticLiouvillePlus}.  This is due to the fact that, as we will see in 
Theorem \ref{VessiotAlg} and Corollary \ref{NonCont}
below, the Vessiot algebras $\mathfrak g$ and $\widetilde{\mathfrak g}$ are contact invariants (up to Lie algebra isomorphisms) of the respective Darboux integrable systems $\sysI$ and $\widetilde \sysI$.  

\end{example}

\def\ra{k}
\def\rb{\ell}
\def\raz{k_z}
\def\razz{k_{zz}}
\def\razzz{k_{zzz}}
\def\rafourz{k_{zzzz}}

In the previous examples of this section, Darboux integrable elliptic systems were obtained as quotients
of holomorphic contact systems, i.e., systems whose realification is a prolongation of the Cauchy-Riemann equations.
The next example shows that this is not the only possibility.
\begin{example} 
Consider the Cartan-Hilbert equation
\begin{equation}
\rb_z = (\ra_{zz})^2
\label{HCHeq}
\end{equation}
for holomorphic functions $\ra(z), \rb(z)$.  Equation \eqref{HCHeq} can be represented by a rank 3 holomorphic Pfaffian system $\sysH$ on a complex 5-dimensional manifold with coordinates $z, \ra, \rb, \raz, \razz$:
$$\sysH = \{ d\ra-\raz\,dz, d\raz - \razz \,dz, d\rb - (\razz)^2\,dz\}\subdiff.$$
This system admits the complex split form of $\fg_2$ as a symmetry algebra; the specific
symmetry vector fields are the holomorphic versions of those given in Appendix C in \cite{AA}.  
(Note that in \cite{AA} the Cartan-Hilbert system is described using real variables $x,u,z$ corresponding to our $z,\ra,\rb$ respectively.)

Consider, as in Section 4.1 of \cite{AA}, the 5-dimensional nilpotent subalgebra $\fk$ of symmetries of $\sysH$ generated by the holomorphic vector fields
$$
\begin{aligned}
\vZ_{3} & = & & 3 z \partial_\ra + 3 \partial_{\ra_z}\\
\vZ_{7} & = & & -7 \partial_\ra \\
\vZ_{10} & = & & \partial_\rb \\
\vZ_{11} & = & &- \frac{z^2}{4}\partial_\ra -\frac{z}{2} \partial_{\ra_z} - \frac{1}{2} \partial_{\ra_{zz}} -\ra_z \partial _\rb\\
\vZ_{14} &=& & -\frac{z^3}{144} \partial_\ra   -\frac{z^2}{48} \partial_{\ra_z} - \frac{z}{24} \partial_{\ra_{zz}}+\frac{1}{12}(\ra-z \ra_z)\partial_\rb\ .
\end{aligned}
$$
(We retain the numbering for the vector fields from \cite{AA}.)
The action generated by these vector fields does not satisfy the transversality conditions in Theorem \ref{C34}.
In fact, we must pass to the second prolongation of system $\sysH$, 
$$\widehat\sysH =\{ d\ra-\raz\,dz, d\raz - \razz \,dz, d\razz - \razzz\,dz, d\razzz - \rafourz \, dz, d\rb - (\razz)^2\,dz\}\subdiff$$
defined on a complex 7-manifold with coordinates $(z, \ra, \rb, \ra_z, \ra_{zz},\ra_{zzz},\ra_{zzzz})$, and prolong the above vector fields
to obtain vector fields $\widehat\vZ_\alpha$ on the 7-manifold which are transverse symmetries of $\widehat\sysH$.
We now summarize the construction of the quotient of the realification $\sysE$ of $\widehat\sysH$ by the action generated by the real form $\fg=\spanr\{ \widehat\vX_\alpha=\widehat\vZ_\alpha + \overline{\widehat\vZ_\alpha} \}$ of $\fk$.

The lowest order real differential invariant of $\fg$ is 
\begin{equation}
\Ug = {\bf i}\left(  
\frac{\rb-\bar \rb}{4} -\frac{(\ra-\bar \ra)^2}{(z-\bar z)^3} -\frac{ \ra^2+\ra \bar \ra +\bar \ra^2}{z-\bar z} 
+3\frac{ (\ra-\bar \ra)(\ra_z+\bar \ra_{\bar z})  }{(z-\bar z)^2}  \right).
\label{BAEQ}
\end{equation}
By applying the $\fg$-invariant vector field
$$D_z:= \partial_z + \ra_z \partial_k+ \ra_{zz} \partial_{\ra_z}+\ra_{zzz} \partial_{\ra_{zz}}+\ra_{zzzz} \partial_{\ra_{zzz}}+(\ra_{zz})^2\partial_\rb$$
and its complex conjugate, we obtain further invariants
\begin{equation}\label{CHinvts}
\Ug_z = D_z \Ug=  \frac{\bf i}{4} \alpha^2, \qquad \Ug_{\bar z} = -\frac{\bf i}{4} \bar \alpha^2,\qquad \ \Ug_{z\bar z} = {\bf i} \frac{\alpha \bar \alpha}{\bar z-z},
\end{equation}
which we express in terms of the complex $\fg$-invariant quantity 
$$
\alpha:= \ra_{zz} - \frac{4 \ra_z+ 2 \bar \ra_{\bar z}} {z-\bar z}+6\frac{\ra-\bar \ra}{(z-\bar z)^2}.
$$
It is evident from the relations among the invariants in \eqref{CHinvts} that the quotient system $\sysE/G$ represents the second-order elliptic PDE
in the plane given by 
\begin{equation}
\Ug_{z\overline{z}}=-4 \frac{\sqrt{\Ug_z\Ug_{\bar z}}}{{\bf i}(z-\bar z)} \quad {\rm or} \quad
\Ug_{xx}+\Ug_{yy} = 4\frac{ \sqrt{\Ug_x^2+\Ug_y^2} }{y}
\label{GPQ}
\end{equation}
We should note that \eqref{GPQ} is  Darboux integrable only after one prolongation, and the system $\sysE/G$ defined on a 9-dimensional
manifold encodes this prolongation.  


A second independent Darboux invariant, in addition to $z$, is found from the lowest-order holomorphic invariant of $\fk=\spanc\{\widehat\vZ_\alpha\}$ on the complex 7-manifold (see Remark \ref{KDINV}).
In terms of the $\fg$-invariants $\Ug,\Ug_z, \Ug_{zz}, \Ug_{zzz}$, this Darboux invariant has the rather formidable expression
\begin{equation}
\xi=-(1+{\bf i})\sqrt{2}\, \ra_{zzzz} =  2\frac{\Ug_{zzz}}{\sqrt{\Ug_z}}+8\frac{\sqrt{\Ug_z}}{(z-\bar z)^2} +8 \frac{\Ug_{zz}}{\sqrt{\Ug_z}(z-\bar z)}-\frac{\Ug_{zz}^2}{\Ug_z^{3/2}}.
\label{G3xi}
\end{equation}
The general solution to equation \eqref{GPQ} is given by \eqref{BAEQ} where $\ra$ and $\rb$ are holomorphic functions satisfying equation \eqref{HCHeq},
and on such solutions $\xi$ is a holomorphic function of $z$.

Similarly, it is straightforward to find elliptic analogues of the other Darboux integrable hyperbolic PDE found in  \cite{AA}.
\end{example}
\newcommand\zb{\overline z}

\newcommand\wb{\overline w}
\newcommand\re{\rm e}

\begin{example}
In this example we construct a Darboux integrable system whose solution depends on a holomorphic function of two variables. 
Let  $N=J^2_{\rm hol}(\C^2,\C)$ be the 8 (complex) dimensional manifold of 2 jets of holomorphic functions of one variable.
Using coordinates $(w,z,U,{U}_w,{U}_z,{U}_{ww},{U}_{wz}, {U}_{zz})$ the complex rank 3 holomorphic contact system is
$$
\sysH = \{ \theta:= dU -U_w dw - U_z dz, \ \theta_w:= dU_w -U_{ww} dw - U_{wz} dz, \ \theta_z:= dU_z -U_{wz} dw - U_{zz} dz\ \}
$$
let $\sysE$ be the realification of $\sysH$. The system $\sysE$ is a maximally Darboux integrable and $\Delta_+ = T_{1,0} N\cap H^\perp$ (see equation \eqref{DefDD1}) is 
$$
{\Delta_+} = {\rm span}_{\C} \{\, D_w:=\partial_w+U_w\partial_U+U_{ww}\partial_{U_w}+U_{wz} \partial_{U_z},
\, D_z:=\partial_z+U_z\partial_U+U_{wz}\partial_{U_w}+U_{zz} \partial_{U_z},\, \partial_{U_{ww}},\, \partial_{{U}_{wz}},\, \partial_{ {U}_{zz}}\, \}
$$
where $D_w$ and $D_z$ are the total derivative vector fields on $N$.

Let the infinitesimal action of the 2 dimensional complex Lie group $K$ in equation \eqref{solvableKdefinition} on $N$ be given by,
$$
Z_1= \partial_U,  \quad Z_2 = U\partial_U + {U_z} \partial_{{U}z} + {U}_w \partial_{{U}_w} + {U}_{ww} \partial_{{U}_{ww}} +{U}_{wz} \partial_{{U}_{wz}}+{U}_{zz} \partial_{{U}_{zz}}
$$
which is a generalization of the action in equation \eqref{bexKNact} in Example \ref{benexample2}. The corresponding real Lie algebra of vector fields 
$$
X_1 = Z_1 + \overline{Z_1}, \quad X_2 = Z_2 + \overline{Z_2}, 
$$
are the infinitesimal action for the real form $G$ of $K$ given in Example \ref{benexample2}. We  now find the Darboux integrable system $\sysI = \sysE/G$ where $\sysE$ is the realification of $\sysH$ which is of real rank 6.  The action of $G$ is tranverse to $\sysE'$ and so the quotient $\sysI$ is a rank 4 Pfaffian system.

The action of $G$ admits the following two (local) complex independent invariants of first order
$$
A = \frac{2{U}_w}{\overline{U}-U }\, , \quad B =  \frac{2{U}_z}{\overline{U}-U}.
$$
The quotient manifold $M=N/G$ has 7 complex dimensions and we chose coordinates
 $(\,w,z,A,B, A_w=D_w(A), B_z=D_z(B),B_w= D_w(B) \, )$ using the fact that $D_w$, $D_z$  commute with $X_1$ and $X_2$.
 The distribution  ${\mathcal D}_+ ={\pi_G}_* \Delta_+$ (see equation \eqref{defDD}) whose complex complement determines the singular system for $\sysI$ is computed  to be,
\begin{equation*}
\begin{aligned}
{\mathcal D}_+ &=&&{\pi_G}_* \Delta_+ =  \{\  \partial_{A_w}, \ \partial_{B_w} , \ \partial_{B_z}\ , \  \\
 &&&  \partial_w  + A_w \partial_A + B_w \partial_B + \frac{|A|^2}{2} \partial_{\overline A} +\frac{\overline{B} A}{2}\partial_{\overline B} 
+\frac{\overline{A}(A^2+2A_w)}{4}\partial_{\overline{A_w}} +\frac{\overline{A}(AB+2B_w)}{4}\partial_{\overline{B_w}} 
+\frac{\overline{B}(AB+2B_w)}{4}\partial_{\overline{B_z}} 
\,
 \\ 
&&&  \partial_z  + B_w \partial_A + B_z \partial_B + \frac{\overline{A}B}{2} \partial_{\overline A} +\frac{|B |^2}{2}\partial_{\overline B} 
+\frac{\overline{A}(AB+2B_w)}{4}\partial_{\overline{A_w}} +\frac{\overline{A}(B^2+2B_z)}{4}\partial_{\overline{B_w}} 
+\frac{\overline{B}(B^2+2B_z)}{4}\partial_{\overline{B_z}}
 \  \}
\end{aligned}
\end{equation*}
Then  $ \sysI\otimes \C =( {\mathcal D}_+\oplus {\mathcal D}_-)^\perp$ is 
$$
\begin{aligned}
\sysI\otimes \C = \{ \quad  &  \theta_A:= dA - A_w dw - \frac{1}{2} |A|^2  d\wb - B_w dz -  \frac{1}{2} \overline{B} A  d \zb , \ \overline{\theta_A}, \\
& \theta_B := dB - B_w dw -  \frac{1}{2} \overline{A} B d \wb- B_z dz - \frac{1}{2} |B|^2 d \zb, \  \overline{\theta_B} \ \}.
\end{aligned}
$$
The corresponding differential equations for the integral manifolds of $\sysI$ are then
\begin{equation}
\partial_{\bar w} A = \frac{1}{2} |A|^2 \ ,\quad  \partial_{\bar z} A = \frac{1}{2} \overline{B} A \, \ , \qquad
\partial_{\bar w} B = \frac{1}{2} \overline{A} B\ , \quad \partial_{\bar z} B = \frac{1}{2} |B|^2 \ ,  \quad \partial_z A = \partial_w B .
\label{hwzpde}
\end{equation}

The lowest order holomorphic $K$ invariant on $N$ besides $w$ and $z$ leading to a homomorphic Darboux invariant is
$$
\xi =   \frac{{U}_w}{{U}_z} =\frac{A}{B}.
$$
The Darboux invariant $\xi$ satisfies on account of equation \eqref{hwzpde}
$$
\partial_{\zb} \xi =\frac{ A_{\bar z}}{B} -\frac{A B_{\bar z}}{B^2} = \frac{1}{2} \frac{\overline{B} A}{B} - \frac{1}{2}\frac{A |B|^2 }{B^2}=0 
$$
and similarly $\partial_{\bar w} \xi =0$ on solutions.
\end{example}


\begin{example} 
In this example we find a Darboux integrable system which is not Pfaffian and whose integral manifolds depend on two function which  are each a holomorphic a function of different complex variables. Let  $ N=J^1_{\rm hol}(\C,\C) \times J^1_{\rm hol}(\C,\C)$ be the product of the jet space of 1-jets of holomorphic functions on one variable with itself which has 6 complex dimensions.  Let  $(w,U,U_w; z , V, {V}_z)$ be complex coordinates on $N$ and the homomorphic complex rank 2 contact system is
$$
\sysH =  \{\ \theta_U=dU-U_w dw, \ \theta_V = dV-V_z dz\  \} _{\rm diff}
$$
and let $\sysE$ be the realification of $\sysH$. The system $\sysE$ is a maximally Darboux integrable and $\Delta_+ = T_{1,0} N\cap H^\perp$ (see equation \eqref{DefDD1}) is 
$$
\Delta_+ =  {\rm span}_{\C}   \{\ D_w:=\partial_w + U_w \partial_U,\  D_z := \partial_z + V_z \partial_V, \partial_{U_w}, \ \partial_{V_z} \ \} \ .
$$

In this example we consider the infinitesimal holomorphic action of the two dimensional complex Lie group $K$ in equation \eqref{solvableKdefinition} given by 
$$
Z_1= \partial_U+\partial_V,  \quad Z_2 = U\partial_U + {U}_w \partial_{{U}_w}  +V\partial_V + V_z \partial_{V_z}.
$$
which is again a generalization of the action in equation \eqref{bexKNact} in Example \ref{benexample2}. 
The corresponding real vector fields for the real form $G$  in Example \ref{benexample2}  are
$$
X_1 = Z_1 + \overline{Z_1}, \quad X_2 = Z_2 + \overline{Z_2}.
$$
We now compute the quotient $\sysI = \sysE/G$ which is Darboux integrable but {\bf not}  a Pfaffian system since $G$ is not transverse to $\sysE'$ (which has rank 0). There is one independent complex $X_1, X_2$   invariant on $N$ of order zero given locally by
\begin{equation}
A=\log  \frac{U - V}{U - {\overline U} + V - {\overline V} } . 
\label{nonPA}
\end{equation}
The vector fields $D_w$ and $ D_z $ commute with $X_1$ and $X_2$ and so 
 $(\,w,z,A,A_w=D_w(A), A_z=D_z(A) )$ form local coordinates on the quotient manifold $N/G$ which has 5 complex dimensions. A set of (local) generating holomorphic  $K$ invariants on $N$  are
\begin{equation}
w, \ z , \  \xi_1 =\frac{ U_w}{U-V} \qquad \xi_2 = \frac{V_z}{V-U}.
\label{nonPDI}
\end{equation}
A set of generators for $\sysH$ which are useful for computing $\sysI=\sysE/G$ are given by
\begin{equation}
\sysH = \{ \ \theta_U,\ \theta_V,\ \Omega^1:= \ dw \wedge \phi, \ \Omega^2:=  dz \wedge \psi\ \}_{\rm alg}
\label{npH}
\end{equation}
where from equation \eqref{nonPDI},
$$
\phi = d \xi_1+ \xi_1 \xi_2 dz \qquad \psi = d\xi_2-  \xi_1\xi_2 dw.
$$
The 2-forms $\Omega^1, \Omega^2$ are $K$ basic (and hence $G$ basic) since $dw,dz,d\xi_1, d\xi_2$ are. The generators of $\sysE\otimes\C = \sysH \oplus \overline{\sysH}$ can be written using equations \eqref{nonPA} and \eqref{npH} as
\begin{equation}
\sysE\otimes \C =\{\ \theta_U ,\ \eta= 
  (1-{\rm e}^A)\theta_U
-(1+{\rm e}^A)\theta_V + {\rm e}^A(\theta_{\overline{V}}+\theta_{\overline{U}}),\ \Omega^1,\ \Omega^2,\ \, \overline{ \theta_{U} }, \overline{\eta},\ \overline{\Omega^1} ,\ \overline{\Omega^2}  \}_{\rm alg}
\label{ppQ}
\end{equation}
where the form $(U-V)^{-1}\eta$ is $G$-basic. The complexification of the  quotient $\sysI = \sysE/G$ is determined by the G-basic forms 
from equation \eqref{ppQ} to be
$$
\sysI \otimes \C = \{\ \Theta,\ \overline{\Theta} , \ \Omega^1,\ \Omega^2,\ \overline{\Omega^1} ,\ \overline{\Omega^2}  \   \}_{{\rm alg}}
$$
where  $\pi_G^* \Theta= (U-V)^{-1}\eta$,  $\ \Omega^1= dw \wedge \phi $, $ \ \Omega^2=  dz \wedge \psi $  and 
$$
\begin{aligned}
\Theta &=& &dA -A_w dw-A_z dz - \frac{\overline{A_w} }{1-{\rm e}^{-\overline{A}}} d \wb - \frac{\overline{A_z}}{1+{\rm e}^{-\overline{A}}}d \zb \\
\phi& =& & ({\rm e}^A -1)^{-1} \left( d A_w -\frac{A_w {\re}^A}{\re^A-1} dA + \frac{A_w A_z}{\re^A+1} dz \right) \\
\psi & = & &   ({\rm e}^A +1)^{-1} \left( d A_z -\frac{A_z {\re^A}}{\re^A+1} dA - \frac{A_w A_z}{\re^A-1} dz \right) .
\end{aligned}
$$

In terms of the coordinates on $N/G$ we have the holomorphic Darboux invariants, 
$$
w,\quad z,\quad \xi_1 =\frac{ A_w}{1-e^A}, \quad \xi_2 = \frac{A_z}{1+e^A},
$$
while the singular system is $V={\rm span}_{\C} \{ \Theta, dw,dz, \phi, \psi \}$ and satisfies $V^\infty={\rm span}_{\C}\{ dw,dz,\pi, \psi\}={\rm span}_{\C}\{ dw,dz,d\xi_1, d\xi_2\}$.

The forms $\phi$ and $\psi$ on $N/G$ can be written $\mod \ \Theta$ as,
\begin{equation}
\begin{aligned}
(\re^A-1) \phi &\equiv& & \mathit{dA_w}  -\frac{  {A_w}^2}{1- \re^{-A}}\mathit{dw} -\Big \vert \frac{  {A_w}  }
{ 1-\re^{-A}}\Big \vert^2  \mathit{d\wb}
-{A_w}{A_z}\coth A\  \mathit{dz}
-\frac{  {A_w}\overline{A_z} }{  (1-\re^{-A})(1+\re^{-\overline{A}})}  \mathit{d\zb}
\quad \mod \Theta\\
(\re^A+1) \psi &\equiv& & \mathit{dA_z}  -\frac{  {A_z}^2}{1+ \re^{-A}}\mathit{dz} -\Big \vert \frac{  {A_z}}
{ 1+\re^{-A}}\Big \vert^2 \mathit{d\zb}
- {A_w}{A_z} \coth A\  \mathit{dw}
-\frac{  \overline{A_w} {A_z}}{  (1+\re^{-A})(1-\re^{-\overline{A}})}  \mathit{d\wb}
\quad \mod \Theta \ .
\end{aligned}
\label{modTC}
\end{equation}
An integral manifold of $\sysI$ of the form $A=A(w,\wb,z,\zb)$ satisfies the differential equations from $\Theta$,
$$
A_{\wb}=\frac{\overline{A_w} }{1-e^{-\overline{A} }}\ ,  \quad
A_{\zb} = \frac{\overline{A_z}}{1+e^{-\overline{A} }} 
$$
together with the following from $dw \wedge \phi$ and $ dz \wedge \psi$  using \eqref{modTC},
$$
A_{w\wb} = \Big\vert \frac{A_w }{1-e^{-A}}\Big\vert ^2 , \quad 
 A_{wz} =  A_w A_z \coth A, \quad
 A_{w \zb} = \frac{  {A_w} \overline{A_z} }{  (1-\re^{-A})(1+\re^{-\overline{A}})} , \quad
 A_{z\overline{z}}=\Big \vert \frac{A_z  }{1 + e^{-A}} \Big \vert ^2.
$$
These second order equations are also the compatibility conditions for the first order system.

Using these differential equations, we have
$$
\begin{aligned}
\partial_{\wb} \xi_1 &=&& \frac{ A_{w \wb} }{1-e^A} + 
\frac{ A_{w} A_{ \wb} {\rm e}^A }{(1-{\rm e}^A)^2}=\Big \vert \frac{A_w }{1-e^{-A}}\Big \vert ^2 \frac{1}{1-e^A}+
\frac{ A_{w} }{(1-{\rm e}^A)(\re^{-A}-1)}\frac{\overline{A_w}}{1-e^{-\overline{A} }}=0\\
\partial_{\zb} \xi_1 & = & &  \frac{ A_{w \zb} }{1-e^A} + 
\frac{ A_{w} A_{ \zb} {\rm e}^A }{(1-{\rm e}^A)^2} 
=\frac{  {A_w}{\overline{A_z}}}{ (1-e^A) (1-\re^{-A})(1+\re^{-\overline{A}})}+\frac{ A_{w}  }{(1-{\rm e}^A)(\re^{-A}-1)}  \frac{\overline{A_z}}{1+e^{-\overline{A} }}= 0
\end{aligned}
$$
which shows $\xi_1$ is holomorphic with respect to $w$ and $z$ on integral manifolds of $\sysI$, while a similar computation holds for $\xi_2$.
\end{example}

\section{The Converse of the Quotient Theorem and the Vessiot Algebra}\label{coframes}

In this section we prove two main results. 
We first prove  that for any Darboux integrable elliptic decomposable system $\sysI$ with Pfaffian singular system
there exists a complex manifold $N,$ a holomorphic Pfaffian system $H,$ and a complex Lie group $K$ satisfying the conditions of Theorem \ref{C34}, 
so that locally $\sysI = \sysE/G$ where $G$ is a real form of $K$ (see Theorem \ref{loopreduction} and Corollary \ref{CircleD} below in \S\ref{ConvTheorem}).  
This result is a local converse to Theorem \ref{C34} which is the necessary part of Theorem \ref{MTE} in the introduction.

The construction of $N, K$ and $H$ requires only algebraic operations and (in some cases) the complex Frobenius theorem, and this allows us to find closed form general solutions to Darboux integrable systems 
\ifjour
\else
in \S\ref{MoreEx}.
\fi
In order to prove Theorem \ref{loopreduction} 
we identify for a given Darboux integrable elliptic distribution (see Definition \ref{EDIdef}) a Lie algebra, which turns out to be the Lie algebra of  the group $G$ used 
to construct the extension in Theorem \ref{extendexist}, by means of a locally-defined coframe we call the Vessiot coframe based at a chosen point $\mpt\in M$.  
The second main result, which is proved in \S\ref{TVA}, is that the Lie algebra computed using the Vessiot coframe is independent of the choice of $\mpt\in M$ and is an invariant of the Darboux integrable elliptic system. 
Examples which demonstrate Theorem \ref{loopreduction} are given in \S \ref{MoreEx}.

Again we remind the reader that we are assuming that $M$ is a real-analytic manifold of dimension $\dimM$, and that all elliptic distributions and elliptic systems on $M$ are real-analytic.  Besides being required for the complex Frobenius theorem, the assumption of real analyticity is essential for the final step in the construction of the Vessiot coframe.


\subsection{The Vessiot Coframe}

The principal result of this subsection is a refinement of the 1-adapted coframe for Darboux integrable elliptic distributions $(\DD, \JJ)$ given in
 Theorem \ref{oneadaptedexist}.  
This is necessary for identifying the Lie algebra of group $G$
in our construction of an extension. 
Note that the coframe given by the following theorem is the elliptic analogue of  the `4-adapted' coframe in Theorem 4.5 in \cite{AFV}.

\begin{thm}  \label{fouradaptedexist}  Let $(\DD, \JJ)$ be a Darboux integrable elliptic distribution of rank $\rankD$ on $M$, with $I =\DD^\perp$ and singular bundle $\rV$, and let $\rankVinf = \rank \Vinf$.  Then given any point $\mpt \in M$ there exists an open set $\Udown$ about $\mpt$ and a 1-adapted coframe  
$(\theta^i, \eta^r, \sigma^u, \etabars{r}, \sigmabars{u})$ on $\Udown$ 
 such that 
\begin{enumerate}[label=\alfenumstyle,ref=\alfenumstyle]
\item \label{Fa1}
$\Vinf \vert_\Udown= \spanc\{ \eta^r, \sigma^u\}$, where $1 \le r \le \rankVinf - \rankDplus, \ 1\le u \le  \rankDplus;$

\item\label{Fa2}
$(\IC \otimes \C)\vert_\Udown=  \spanc\{\theta^i,\eta^r,  \overline{\eta^r}\}$, where $\ 1 \leq i \leq \dimG=\dimM-2\rankVinf$;

\item \label{Fa3}
the coframe satisfies structure equations of the form
\begin{subequations}\label{threeadapt}
\begin{align}
d\sigma^u&=0, \label{threedsigma} \\
d\eta^r &=\tfrac12 E^r_{u v} \sigma^u \wedge \sigma^v +
F^r_{u s} \sigma^u \wedge \eta^s,\label{threedeta} \\
d\theta^i &= \tfrac12 A^i_{ab} \pi^a \wedge \pi^b + \tfrac12 P^i_j \overline{A^i_{ab}} \pibars{a}\wedge \pibars{b}
+ \tfrac12 C^i_{jk} \theta^j \wedge \theta^k + M^i_{aj} \pi^a \wedge \theta^j,\label{threedtheta}
\end{align}
\end{subequations}
where $(\eta^r, \sigma^u)$ are as usual combined into a single vector $\vpi= (\pi^a)$ with $1\le a \le \rankVinf $,
 and $\theta^i\equiv P^i_j \overline{\theta^j} \ \mod \eta^r,\overline{\eta^r}$ defines $P^i_j$;

\item \label{FaH}
the coefficients $ A^i_{ab}, M^i_{aj}, E^r_{u v}, F^r_{u s} \in \Invts$;

\item \label{Fa4} the 1-forms $\theta^i$ are imaginary-valued on (real) vectors in $T_\mpt M$;

\item \label{Fa5} the coefficients $C^i_{jk}$ are real constants, and are the structure constants of a Lie algebra $\fg$ of dimension $\dimG$. 


\end{enumerate}
\end{thm}

We call a  1-adapted coframe $(\theta^i, \eta^r, \sigma^u, \etabars{r}, \sigmabars{u})$ which satisfies the conditions in Theorem \ref{fouradaptedexist} a {\it Vessiot coframe} for $(\DD, \JJ)$ {\em based at $\mpt \in M$}. 

\begin{remark}\label{LNF}
\ifarch
It is worth noting here that, for Darboux integrable systems of the form $\sysE/G$ as in Theorem \ref{C34}, 
Theorem \ref{DICF} below shows that 
\fi
about each point $\mpt \in N/G$ there exists a Vessiot coframe for the corresponding elliptic decomposable Darboux integrable system $\sysI=\sysE/G$ 
such that the coefficients $C^i_{jk}$ are the structure constants of the Lie algebra of $G$. 
\ifarch
Since Theorem  \ref{fouradaptedexist} gives a Vessiot coframe for all Darboux integrable elliptic distributions, this suggests the converse to Theorem \ref{DICF} might hold locally for all Darboux integrable elliptic decomposable systems; see Corollary \ref{CircleD}.
\fi
\end{remark}


\medskip

\noindent
{\it Proof:}  We prove Theorem \ref{fouradaptedexist} by following (with modifications as necessary) the steps leading to a `4-adapted' coframe as in \cite{AFV}; see also \S10.3 in \cite{CfB2}.

\ifjour
\else
\subsubsection*{Step 1}  
Let $(\theta^i, \eta^r, \sigma^u, \etabars{r}, \sigmabars{u})$ be a 1-adapted coframe satisfying the conditions of Theorem \ref{oneadaptedexist}
on an open neighborhood about $\mpt$, and let $(\vT_i, \vN_r, \vS_u, \overline{\vN_r}, \overline{\vS_u})$ be the dual frame.  (It will sometimes be convenient to refer to parts of the coframe as vectors of 1-forms, denoted as $\vtheta$, $\veta$ and $\vsigma$.)
Note that $\rV\vert_\Udown = \{ \theta^i, \eta^r, \sigma^u, \etabars{r}\}$, and this
coframe already satisfies parts \ref{Fa1}, \ref{Fa2} of the theorem as well as structure equations \eqref{threedsigma} and \eqref{threedeta}.  In these first two steps, we will modify the $\theta^i$ so that \eqref{threedtheta} is also satisfied.  

It follows from the fact that $d\vsigma$,  $d\vtheta$ and $ d\veta$ contain no $\sigma \wedge \sigmabar$ terms (see equations \eqref{onedsigma}, \eqref{onedeta} and \eqref{onedtheta}) 
that the vector fields $\vS_u$, which form a basis of sections for $\Dplus$,  satisfy the essential condition $[\vS_u, \overline{\vS_v}]=0$.  
Following the argument on page 1925 of \cite{AFV} let $\vS_{uv} = [\vS_u, \vS_v]$, and more generally for a multi-index $W$ whose entries range
between 1 and $\rankDplus$ let 
$$\vS_{W\!v} = [\vS_W, \vS_v].$$
Because $d\vsigma=0$ and $d\vpibar $ contains no $\sigma \wedge \sigma$ terms, $\vS_{uv}$ is in the span of the vectors $\vT_i$ and $\vN_r$.  Moreover, since $d\vetabar$ contains no $\sigma \wedge \theta$ or $\sigma \wedge \eta$ terms, then $[\vS_u, \vT_i]$ and $[\vS_u, \vN_r]$ are
also in the span of $\vT_i$ and $\vN_r$.  It follows by induction (and the Jacobi identity) that 
$S_W \in \spanc \{ \vT_i, \vN_r\}$ whenever $|W|>1$.  (We let $|W|$ denote the length of $W$.)

There will be a finite subset of these iterated brackets that span $\Dplusinf = \spanc\{ \vT_i, \vN_r, \vS_u\}$. In particular, there will be a
finite set $\{ S_{W_r}\}$, with $|W_r|>1$, such that $(\vT_i, \vS_{W_r},\vS_u )$ is a local
basis of sections of $\Dplusinf$.  Thus, $(\vT_i, \vS_{W_r}, \vS_u, \overline{\vS_{W_r}}, \overline{\vS_u})$ is a local frame on $M$, and we let $(\widetilde\theta^i, \widetilde\eta^r, \sigma^u,\overline{\widetilde\eta^r},\overline{ \sigma^u})$ denote
its dual coframe.  Note that the 1-forms $\sigma^u$ are unchanged, and the $\widetilde\eta^r$ have the same span as the $\eta^r$, but may not satisfy
the desired structure equations \eqref{threedeta}; in fact, we will discard $\widetilde\eta^r$ and use $(\widetilde\theta^i, \eta^r, \sigma^u,\overline{\eta^r},\overline{ \sigma^u} )$ as our new coframe.

Since $[\vS_u, \overline{\vS_v}] = 0$ it follows by the Jacobi identity that $[\vS_W, \overline{S_v}] = 0$ for all $W$ and hence $[S_{W_r}, \overline{S_{W_s}}] = 0$ for all $r,s$.  Thus, the expressions for $d\widetilde\theta^i$
in terms of our new coframe contain no terms of the form $\veta \wedge \vsigmabar$, $\vsigma \wedge \vetabar$, or $\veta \wedge \vetabar$.  Thus (dropping the tildes) our new structure equations have the form
\def\tC2{{C}}
\begin{equation}
d\theta^i = \tfrac12 A^i_{ab} \pi^a \wedge \pi^b + \tfrac12 B^i_{ab} \pibars{a} \wedge \pibars{b}
+ \tfrac12 \tC2^i_{jk} \theta^j \wedge \theta^k + M^i_{ak} \pi^a \wedge \theta^k + N^i_{ak} \pibars{a} \wedge \theta^k,
\label{COF2}
\end{equation}
where 
the complex coefficients $A^i_{ab}, B^i_{ab}, \tC2^i_{jk}$ are skew-symmetric in the lower indices.  (The coframe at this point is analogous to the `2-adapted' coframe in Theorem 4.3 of \cite{AFV}; compare \eqref{COF2} to equation (4.13) in \cite{AFV}.) Our coframe is still 1-adapted in the sense of Theorem \ref{oneadaptedexist} and in particular this implies that $B^i_{ab}=P^i_j \overline{A^j_{ab}}$ where  $\theta^i\equiv P^i_j \overline{\theta^j} \ \mod \eta^r,\overline{\eta^r}$ defines $P^i_j$.

\subsubsection*{Step 2}
Let $(\theta^i, \eta^r, \sigma^u, \etabars{r}, \sigmabars{u})$ be the coframe from Step 1.
Let $(t^i, z^a, \zbars{b})$ be adapted local coordinates (in the sense of Lemma \ref{Vinfdz}) defined on a possibly smaller neighborhood $\Udown$ of $\mpt$.
The 1-forms $(\theta^i, dz^a, d\zbars{b})$ also form a local coframe on $\Udown$, and let $(\vT_i, \vU_a, \vUbars{b})$ be its dual frame. 
(Note that since the span of the $dz^a$ is the same as that of the $\pi^a$, then the $\vT_i$ are the same vectors
as in the dual frame to $(\theta^i, \eta^r, \sigma^u, \etabars{r}, \sigmabars{u})$ used above.)
We now follow the argument on page 1926 in \cite{AFV} for our next coframe adaptation. 

Since by \eqref{COF2}, $d\vtheta$ contains no $\vpi \wedge \vpibar$ terms, then $[\vU_a, \vUbars{b}]=0$. Similarly to Step 1, for any multi-index $A$ whose entries range between $1$ and $\rankVinf$ we define iterated brackets
$$\vU_{Ab} = [\vU_A, \vU_b]$$
and note that, by the Jacobi identity,  $[\vU_A, \vUbars{b}] = 0$ for all $A,b$.  
Moreover, because of the structure equations satisfied by our temporary coframe $(\theta^i, dz^a, d\zbars{a})$, 
$[\vU_a, \vU_b]$ is in the span of the $\vT_i$, and the same is true for $\vU_A$ whenever $|A|>1$.  
On the other hand, since $(\Dplus)^\perp = \spanc \{ \theta^i, \pibars{a}, \eta^r\}$, then 
$$\Dplus \subset  (\spanc \{\theta^i, \pibars{a}\})^\perp= \spanc \{ \vU_a\},$$ 
and it follows that $\Dplusinf \subset \spanc\{ \vU_A \mid |A|\ge 1\}$.
Since 
$$\spanc\{ \vU_A \mid |A|> 1\} \subset \spanc\{ \vT_i \} \subset \Dplusinf,$$
then there are finitely many $\vU_{A_i}$ with $|A_i|>1$ that 
form a local basis for $\spanc \{ \vT_i\}$.

Thus, $(\vU_{A_i}, \vU_a, \vUbars{b})$ is a local frame.  If we let $(\widetilde\theta^i, \widetilde\pi^a, \overline{\widetilde\pi^b})$
be the dual coframe, then at each point the span of our new 1-forms $\widetilde\theta^i$ is the
same as that of the $\theta^i$.  However, we retain the previous 1-forms $\vpi = (\eta^r, \sigma^u)$, taking
 $(\widetilde \theta^i, \pi^a, \pibars{b})$ as our new coframe.
Because $[\vU_A, \vUbars{b}]=0$, and the span of the $\widetilde \pi^a$ is the same as that of the $\pi^a$, the structure equations satisfied by $\widetilde \theta^i$ in our new coframe contain no $\widetilde\theta \wedge \pibar$ terms, and so after dropping the tildes we have
\fi
\def\tC2{{C}}
\begin{equation}
d\theta^i = \tfrac12 A^i_{ab} \pi^a \wedge \pi^b + B^i_{ab} \pibars{a}\wedge \pibars{b}
+ \tfrac12 \tC2^i_{jk} \theta^j \wedge \theta^k + M^i_{aj} \pi^a \wedge \theta^j. \label{threedthetap}
\end{equation}
(At this point our coframe $(\theta^i, \eta^r, \overline{\eta^r}, \sigma^a, \overline{\sigma^a})$ is analogous to the `3-adapted' coframe in Theorem 4.4 in \cite{AFV} where $\theta^i$ plays the role of $\theta^i_X$ and $\overline{\theta^i}$ plays the role of $\theta^i_Y$; compare equation \eqref{threedthetap} to (4.22) in \cite{AFV}.) 
Again our coframe is still 1-adapted and so $B^i_{ab} = P^i_j \overline{A^j_{ab}}$ and part \ref{Fa3} is proved; note, however, that the $C^i_{jk}$ are not necessarily constants.

Let $(\vT_i, \vU_a, \vUbars{b})$ be the dual frame to our new coframe.  
If we take the exterior derivative of equation \eqref{threedthetap}, the terms containing $\pibars{b}$ 
imply that
\begin{equation}
\vUbars{b}( A^i_{ab}) = 0, \qquad \vUbars{b}(\tC2^i_{jk}) = 0,\qquad    \vUbars{b}(M^i_{aj})=0, 
\label{ACMinvt}
\end{equation}
and so $A^i_{ab},\tC2^i_{jk},M^i_{aj}\in \Invts$ by Lemma \ref{DIX}, establishing part \ref{FaH} of the theorem.  For later use, we also note that
the terms containing $\theta^i\wedge \theta^j\wedge \theta^k$ lead to the identity
\begin{equation}\label{Cjacobi}
C^i_{j l}C^l_{km} +C^i_{kl}C^l_{mj}  +C^i_{ml}C^l_{jk} =0.
\end{equation}

\bigskip

Before giving the final coframe adaptation which will finish the proof of
Theorem \ref{fouradaptedexist}, we need two lemmas.

\begin{lemma}\label{Pprop} 
Let $(\theta^i, \pi^a, \pibars{b} )$, with $\pi^a=(\eta^r, \sigma^u)$, be a 1-adapted coframe on $\Udown$ satisfying equations \eqref{threedsigma},  \eqref{threedeta} and \eqref{threedthetap}.  Let $P$ be the non-singular $n \times n$ matrix $P$ of functions such that  $\theta^i \equiv P^i_j \overline{\theta^j} \mod \veta, \vetabar$. Then 
\begin{equation}
\begin{aligned}
P^i_j \overline{\tC2^j_{\ell m}} &= -\tC2^i_{jk} P^j_\ell P^k_m.
\end{aligned}
\label{algebraicCPidentity} 
\end{equation}
\end{lemma}

\begin{proof} Equation \eqref{algebraicCPidentity} follows by substituting  
$\theta^i \equiv P^i_j \overline{\theta^j}  \mod \veta,\overline{\veta}$  into equation \eqref{threedthetap} and its conjugate.
\end{proof}

\begin{lemma}\label{realatm} Let $(\theta^i, \pi^a, \pibars{b})$ be a coframe defined on a neighborhood $\Udown$ of $\mpt \in M$ again satisfying \eqref{threedsigma},  \eqref{threedeta} and \eqref{threedthetap}.  There exists a constant-coefficient change of coframe $\widetilde\vtheta = K \vtheta$ such that  $\widetilde \vtheta$ is imaginary-valued on (real) vectors in $T_{\mpt} M$.
\end{lemma}
Note that after this change of coframe, part \ref{Fa4} holds and the new matrix $\widetilde P^i_j$ defined as in Lemma \eqref{Pprop} satisfies $\widetilde P^i_j(\mpt)=-\delta^i_j$.

\begin{proof}
Let $(\vT_j, \vU_a, \vUbars{a} )$ be the dual frame to the given coframe.
In terms of an adapted coordinate system  $(t^i, z^a, \zbars{a})$,   
there is a nonsingular $n\times n$ matrix $B$ whose entries are complex-valued functions satisfying
$$
\vT_j = B^i_j \dib{t_i}.
$$
Then there is a nonsingular constant matrix $K$ such that $\Re( B(\mpt) K^{-1}) = 0$.  It follows that
if we make a change $\widetilde\theta^i=K^i_j \theta^j$ of the adapted coframe from Step 2, then the 1-forms 
$\widetilde\theta^i$ are imaginary-valued at $\mpt$.  
\end{proof}

\subsubsection*{Step 3}
\def\tP{\mathring{P}}
\def\bt{{\bf{t}}}
\def\bz{{\bf{z}}}
\def\bw{{\bf{w}}}
\def\proj{\pi_{\bz}}

Let $(\theta^i, \pi^a, \pibars{b})$ a local adapted coframe from Step 2, modified by a constant-coefficient change of coframe as in Lemma  \ref{realatm} so that it satisfies part \ref{Fa4} of the theorem. Let $(t^i, z^a, \zbars{b})$ be adapted coordinates near $\mpt$.
(In what follows will be convenient to suppress indices and write these coordinates as vectors, e.g., $\bt = (t^i)$, $\bz = (z^a)$.) 
Denote the local coordinates of $\mpt $ by $(\bt_0,\bz_0)\in \R^n\times \C^q$.
Let $\tP^i_j = P^i_j \restr_{\bt=\bt_0}$;
these are real-analytic (but not holomorphic) functions on an open set $\Udown_0\subset \C^q$ about $\bz_0$, 
and by Lemma  \ref{realatm} satisfy $\tP^i_j\restr_{\bz_0}=-\delta^i_j $.
Then substituting $\bt=\bt_0$ into equation \eqref{algebraicCPidentity} gives
\begin{equation}
\tP^i_j(\bz, \overline{ \bz}) \overline{\tC2^j_{\ell m}(\bz)} = -\tC2^i_{jk}(\bz) \tP^j_\ell ( \bz, \overline{\bz})  \tP^k_m( \bz, \overline{ \bz})  , \label{COP}
\end{equation}
where we have emphasized the dependence on local coordinates.

We now use $\bz$ and $\bw=(w^a)$ as coordinates on $\C^{2q}=\C^q \times \C^q$, and let $D=\{ (\bz, \bw) \mid \bw=\overline{\bz}\}$ be the `diagonal' in $\C^{2q}$.  Let $\proj:D \to \C^q$ be the restriction to $D$ of the projection $(\bz,\bw)\mapsto \bz$ and let $D_0 = D\cap \proj^{-1}(\Udown_0)$.
Using a construction called {\defem polarization} (see Lemma 3.1.4 and Proposition 3.1.5  in \cite{JL}), there exist unique
holomorphic functions $F^i_j(\bz,\bw)$ defined on a open set $\Upset_0\subset \C^{2q}$ containing $D_0$ such that
$$F^i_j\restr_{D_0} = \proj^* \tP^i_j.$$
In the other words, the functions $F^i_j$ are the unique holomorphic extensions of $\proj^*\tP^i_j$ to the vicinity of $D_0$.  
(We shrink $\Upset_0$ as necessary to ensure that the $F^i_j$ are components of a nonsingular matrix.)
Similarly, we let $S^i_{jk}$ be the unique holomorphic functions on $\Upset_0$ such that
$$S^i_{jk}\restr_{D_0} = \proj^* \overline{C^i_{jk}}.$$
Note that because $\overline{C^i_{jk}}$ depends only on $\overline\bz$, $S^i_{jk}$ depends only on $\bw$.

Now by a slight extension of polarization applied to the components of equation \eqref{COP}, we have 
\begin{equation}
F^i_j(\bz,  \bw) S^j_{\ell m}( \bw) = -\tC2^i_{jk}(\bz) F^j_\ell ( \bz,  \bw)  F^k_m(\bz, \bw) 
\label{FCzw}
\end{equation}
for $(\bz,\bw) \in \Upset_0$.
If we set $\bw=\overline{\bz_0}$ in equation \eqref{FCzw}, and let $K^i_j(\bz)=-F^i_j\restr_{\bw=\bz_0}$ we have 
\begin{equation}
\overline{\tC2^i_{jk}}\restr_{\bz=\bz_0}= (K^{-1})^i_\ell \tC2^\ell_{km}  K^k_j  K^m_k ,
\label{constC}
\end{equation}
for which the left-hand side is constant.
({In practice, as shown in the examples below, when 
one has explicit formulas for the entries of the matrix $P$ one can compute the matrix $K(\bz)$ by treating $\bz$ and $\overline{\bz}$ as
independent variables and setting $K^i_j(\bz)  =-P^i_j\restr_{\bt=\bt_0, \overline\bz=\overline {\bz_0}}$})

We now make the final change of coframe, setting
\begin{equation}
\widetilde \theta^i = (K^{-1})^i_j \theta^j .
\label{Kinvtheta}
\end{equation}
Because $K^i_j \in \Invts$, structure equations of the form \eqref{threedthetap}  are still satisfied by the $\widetilde\theta^i$.
The corresponding coefficients $\widetilde A^i_{ab}, \widetilde M^i_{ak}$ again satisfy equation \eqref{ACMinvt}  so that  
$\widetilde A^i_{ab}, \widetilde M^i_{ak} \in  \Invts$. Moreover, by equation \eqref{constC},
$$
\widetilde C ^i_{kj} = (K^{-1})^i_\ell \tC2^\ell_{km} K^k_j   K^m_k
= \overline{\tC2^i_{jk}}\restr_{ \bz=\bz_0  }
$$
are constants. The new matrix $\widetilde P$ defined as in Lemma \ref{Pprop} satisfies
 $\widetilde  P^i_j = (K^{-1})^i_kP^k_\ell K^\ell_ j$.  Since $K^i_ j(\mpt)=-P^i_j(\mpt) = \delta^i_j$ we have 
 $\widetilde  P^i_j(\mpt) = -\delta^i_j$, so that the 1-forms $\widetilde \theta^i$  are again pure imaginary on $T_\mpt M$. 
At this point, parts \ref{Fa3} through \ref{Fa4} are satisfied by $\widetilde{\theta}^i$, and $\widetilde C^i_{jk}$ are constants.

Now dropping the tildes, we have for the coframe at hand,
$$
d\theta^i = \tfrac12 A^i_{ab} \pi^a \wedge \pi^b + \tfrac12 P^i_j \overline{A^i_{ab}} \pibars{a} \wedge \pibars{b}
+ \tfrac12  C^i_{jk} \theta^j \wedge \theta^k + M^i_{ak} \pi^a \wedge \theta^k,
$$
where 
we use equation \eqref{PED2} to replace the $B^i_{ab}$ coefficients.
The identity in equation 
\eqref{algebraicCPidentity} holds here, which when evaluated at $\mpt$ and using 
$P^i_j(\mpt) = -\delta^i_j$ gives
$$
\overline{C^j_{\ell m}}=  C^j_{\ell m}, 
$$
so that the $C^i_{jk}$ are real. Finally, equation \eqref{Cjacobi} shows that the $C^i_{jk}$ are the structure constants of a real Lie algebra of dimension $n=m-2q$.  This finishes the proof of part \ref{Fa5} and completes the proof of Theorem \ref{fouradaptedexist}.\qed










\def\hatsysE{\widehat \sysE}
\def\hatH{\widehat H}
\def\hatE{\widehat E}
\subsection{Constructing an Integrable Extension}\label{CONSIE}

The next step we need in proving the converse of Theorem \ref{C34} is to construct 
a further coframe adaptation which will allow us to define an action of the complex group $K$ on $M$ along with a slice for that action (see \S{}II.4 in \cite{Bredon}).

\begin{prop}\label{getomegas} 
Let $(\sysI, \dD,\JJ)$ be a Darboux integrable elliptic decomposable system on manifold $M$.
Let $(\vtheta,\vpi,\vpibar)$ be a Vessiot coframe defined on an open set $\Udown \subset M$ and based at $\mpt \in \Udown$. 
Let $C^i_{jk}$ be the constant coefficients in the structure equations \eqref{threeadapt} satisfied by this coframe,
which are structure constants for the corresponding real Lie algebra $\fg$, 
and let $P$ be the matrix of functions defined by $\theta^i=P^i_j \overline{\theta^j} \ \mod \eta^r,\overline{\eta^r}$.  Then, defined on a possibly a smaller neighborhood of $\mpt$, there is a nonsingular $\dimG \times \dimG$ matrix $R$ and an $\dimG \times \rankVinf$ matrix $S$, both of whose
entries are functions in $\Invts$, such that the 1-forms
\begin{align}
\omega^i &= R^i_j \theta^j + \psi^i + (R P \overline{R^{-1}})^i_j \overline{\psi^j}
\qquad \text{where }\psi^j := S^j_a \pi^a,
\label{omegaform} \\
\intertext{satisfy}
\label{omegastruc}
d\omega^i &= \tfrac12 C^i_{jk} \omega^j \wedge \omega^k.
\end{align}
Moreover, the matrix $R$ at each point  defines an automorphism of the complexification of the Lie algebra $\fg$, i.e.,
$$R^i_j C^j_{k \ell} = C^i_{m n} R^m_k R^n_\ell .$$
In addition $R$ may be chosen such that $R^i_j(\mpt)\in \R$, and so the forms $\omega^i$ are imaginary-valued at $\mpt$.
\end{prop}


Part of the proof follows the arguments in section 4.4 of \cite{AFV}, which utilizes the Frobenius Theorem. In the case at hand the complex version is needed instead, necessitating the assumption of real-analyticity stated at the beginning of this section of the paper.

\begin{proof} Following the proof of Theorem 4.1 in \cite{AFV}, we obtain 1-forms
\begin{equation}
\label{AFV5A}
\widehat\theta^i =  R^i_j \theta^j + S^j_a \pi^a
\end{equation}
which satisfy structure equations
\begin{equation*} 
d\widehat\theta^i = \tfrac12 G^i_{ab} \pibars{a} \wedge \pibars{b} + \tfrac12 C^i_{jk} \widehat\theta^j \wedge \widehat\theta^k.
\end{equation*}
(For the dependence of the entries of $R,S$ see the first paragraph of section 4.4 in \cite{AFV}, noting that $S^i_a$ corresponds to $\phi^i_a$ in \cite{AFV}.)  The proof  breaks into cases depending on the type of the Lie algebra $\fg$, and one can check that in each case $R$ can be chosen to be real-valued at a given point.  

The proof that the 1-forms $\omega^i$ satisfy \eqref{omegastruc} follows the proof of Lemma 5.8 in \cite{AFV}.
(Again, note that the 4-adapted coframe components $\theta^i_X, \theta^i_Y$ in \cite{AFV} correspond respectively to the Vessiot coframe components $\theta^i$ and $\overline{\theta^i}$ in Theorem \ref{fouradaptedexist}, and hence $\hat{R}, Q, \check{R}$ and $\psi^j_\alpha$ in Lemma 5.8 of \cite{AFV} correspond respectively to $R, P, \overline{R}$ and $\overline{S^j_a}$ in this paper.)
\end{proof}

See  equations 
\ifarch \eqref{Benomega},  \eqref{G35A} and \eqref{omegaBH}
\else \eqref{Benomega} and \eqref{G35A}
\fi for examples of the construction of the forms $\omega^i$.
(Note that these 1-forms do not, in general, belong to $\sysI$.)
\bigskip



Let $\omega^i$ be the 1-forms associated to our Vessiot coframe as in Prop. \ref{getomegas}, based at $\mpt$ and defined on $\Udown \subset M$.  
Then $(\omega^i,\pi^a,\pibars{a})$ is also a coframe on $\Udown$.  If $(\vZ_i, \vU_a,\vUbars{a})$ is its dual coframe,
then from \eqref{omegastruc} the vector fields $\vZ_i$ satisfy 
\begin{equation}\label{MZbrack}
[\vZ_j, \vZ_k] = -C^i_{jk} \vZ_i.
\end{equation}
Let $K$ be a complex Lie group with structure constants $C^i_{jk}$, with a basis $\vZL_i$ of left-invariant holomorphic vector fields satisfying \eqref{RC}.  Since the corresponding right-invariant vector fields $\vZR_i$ satisfy the same bracket relations as \eqref{MZbrack}, $\Udown$ is foliated by orbits of a (local) left $K$-action.  We will assume the real and imaginary parts of the vector fields $\vZ_i$ are complete, so that this action is global.  We let  $\Gam: K \times \Udown \to \Udown$ denote the action of $K$ such that 
$\Gam_* \vZR_i = \vZ_i$.

Now it follows from 
\eqref{omegaform} that 
\begin{equation}\label{omegabardentity}
\omega^i = (R P \overline{R^{-1}})^i_j \overline{\omega^j},
\end{equation}
and so $\{\omega^i\}$ annihilates a real codimension-$\dimG$ distribution $\EDelta$ spanned by the real and imaginary parts of the $\vU_a$.
Thus, $\EDelta$ has rank  $2\rankVinf$ and it follows from \eqref{omegastruc} that it satisfies the Frobenius condition.  Let 
$\Slice$ be the maximal integral manifold of $\EDelta$ through $\mpt$. 
Since $\omega^i(\vZ_j) = \delta^i_j$ we have that $\Slice$ is transverse to the orbits of $K$.

By \eqref{omegabardentity}, $P(\mpt)=-I$, and because $R(\mpt)$ is real-valued we have  $\omega^i\vert_\mpt=-\overline{\omega^i}\vert_\mpt $, so that
$$
\omega^i(\vX_j) = \omega^i(\vZ_j+\overline{ \vZ_j}) \vert_\mpt= \omega^i(\vZ_j)\vert_\mpt+ \omega^i(\overline \vZ_j)\vert_\mpt =0.
$$ 
If we let $G=K_\mpt\subset K$ be the isotropy subgroup  at $\mpt$, then the Lie algebra of $G$ has the basis $\{ \vX^R_i=\vZ^R_i+{\overline \vZ^R_i}\}$ and $G$ is a real form of $K$.

Suppose now the $\Slice$, $K$ , $\Gam$ are as above. Since $[\vU_a, \vX_i]=0$
 we have
$$
\vX_i \vert_x = 0, \quad x \in \Slice,
$$
since $\Slice$ is connected.  Hence the Lie algebra of the isotropy subgroup $K_x \subset K$ is the same at all points $x\in \Slice$, and the submanifold $\Slice$ forms an infinitesimal analogue of a slice.  Sufficient conditions for $\Slice$ to be an actual slice are given in the next lemma and will be assumed to simplify the exposition.

\begin{lemma}\label{slicelemma}   Suppose for each $x\in \Slice$ that $K_x$ is connected, and that $\Slice$ is a cross-section to the left action of $K$ on $\Udown$. Let $\gam: \Slice \times K \to \Udown $ be the restriction of $\,\Gam$ to $\Slice \times K$, i.e.,
\begin{equation}\label{gammadef}
\gam(x,k) := \Gam(k,x) = k \cdot x, \qquad x \in \Slice \ .
\end{equation}
\begin{enumerate}[label=\thmenumstyle,ref=\thmenumstyle]
\item
The map $\gam: \Slice \times K \to \Udown$ 
is a smooth submersion which is equivariant with respect to left-multiplication of $K$ on itself in the domain, and $\gam_* \vZR_i = \vZ_i$.

\item
$ \gam$ is invariant under right-multiplication by elements of $G=K_\mpt$, and induces a left $K$-equivariant injective local diffeomorphism $\Phi:\Slice \times K/G \to \Udown$ such that this diagram commutes:
\begin{equation} 
\begin{gathered}
\tikzstyle{line} = [draw, -latex', thick]
\begin{tikzpicture}
\node(B) {$\Slice\times K$};
\node[below of=B, node distance=20mm](I2) {$\Slice\times K/G$};
\node[right of=I2, node distance=25mm](I1) {$\Udown$};
\path[line](B) -- node[right, label={$\gam$}] {} (I1);
\path[line](B) -- node[left] {$\bqg$} (I2);
\path[line](I2) -- node[below] {$\Phi$} (I1);
\end{tikzpicture}
\end{gathered}
\label{sliceCD}
\end{equation}
where the usual quotient map $\bqg: K \to K/G$ is extended trivially to the product with $\Slice$.

\end{enumerate}

\end{lemma}
\begin{proof}  

{\bf Part (1):} The function $\gam$ defined by \eqref{gammadef}
clearly has the equivariance property stated and $\gam_*\vZ^R_i =\vZ_i$. Since $\omega^i(\vZ_j) = \delta^i_j$, and $\omega^i\vert_\Slice=0$ ,  the orbits of $K$ are transverse  to $\Slice$ which implies that $\gam$ is a submersion. 

\smallskip
\noindent
{\bf Part (2):}
Since the isotropy subgroup $K_x$ at $x$ and $G$ have the same Lie algebra and are connected, $K_x=G$ for all $x \in \Slice$. Then  for $(x,k) \in \Slice\times K$ and $g \in G= K_\mpt$ we have $\gam(x,kg) = k g\cdot x = k\cdot x= \gam(x,k)$, so $\gam$ is invariant under the right action of $G$ and thus factors through the quotient map $\bqg$. Therefore there exists $\Phi: \Slice \times K/G \to \Udown$ so that the commutative diagram in  \eqref{sliceCD} holds.

Equivariance of $\Phi$ with respect to the left $K$-action is clear, so we must show $\Phi$ is a diffeomorphism onto its image. Because $\gam$ is a submersion, $\Phi$ is also a submersion.  Since its domain and range have the same dimension, it is a local diffeomorphism, so we must show it is injective. Suppose now that, for $x,y\in \Slice$ and cosets $kG$ and $\ell G$, $\Phi(x,kG) = \Phi(y, \ell G)$.  Then $\ell^{-1}k\cdot x= y$. Since $\Slice$ is assumed to be a cross-section to the action of $K$, we have $x=y$ and $\ell^{-1} k \in G$, that so $k G =\ell G$.  Thus, $\Phi$ is injective and part (2) is proved.
\end{proof}

\begin{remark}\label{tubesliceremark} The map $\Phi$ in \eqref{sliceCD} is known as a {\defem tube} about the orbit  $K\cdot \mpt$ 
and $ \Slice$ is a slice through $\mpt$ (again, see \cite{Bredon} for definitions). If the connectivity of $G_x$ fails or $\Slice$ fails to be a global cross-section, then an infinitesimal version of the diagram \eqref{sliceCD} exists. Also note that if $K$ is solvable then $\Slice$ can be found by quadratures; see Corollary 3.1 in  \cite{FelsSolve}.
\end{remark}


The 1-forms $( \omega^i, \pi^a ,  \pibars{b} )$ from Proposition \ref{getomegas} define a complex coframe on $\Udown$ such that
the $\omega^i$ vanish when pulled back to $\Slice$.  Thus, $( \pi^a\vert_\Slice, \pibars{b}\vert_\Slice )$ define a complex coframe on $\Slice$,  where we use restriction as an abbreviation for pullback under the inclusion $\Slice \hookrightarrow M$.
Let $J_{\Slice}$ be the almost complex structure on $\Slice$ such that 
\begin{equation*}
T_{1,0}^* \Slice = \spanc \{ \pi^a\vert_\Slice \}=\spanc \{\ \sigma^u\vert_\Slice,\ \eta^r\vert_\Slice\ \}.
\end{equation*}
Since the pullback of $\spanc \{\pi^a\} = \Vinf$ is a Frobenius system, then $J_{\Slice}$ is integrable.
Note that the set of independent holomorphic Darboux invariants $z^a \in \Invts$ in Lemma \ref{Vinfdz} restricted to $\Slice$ define  complex coordinates on $\Slice$. The structure equations \eqref{threedsigma} and \eqref{threedeta}  along with 
$E^r_{u v}, F^r_{u s} \in \Invts$ shows that $\pi^a\vert_\Slice$ are holomorphic 1-forms.

Now let $J$ be the induced complex structure on the product complex manifold $N=\Slice \times K$.  Then $T_{1,0}^*(\Slice \times K)= \spanc \{ \pi^a\vert_\Slice , \muR^j \}$ where $\muR^j$ are the right-invariant $(1,0)$-forms on $K$ (see \S\ref{conv}), and 
we suppress the pullback under the projections to each factor.  Since the projections are holomorphic mappings, these 1-forms define a holomorphic coframe on $\Slice \times K$.


\def\bfp{{\bf p}}

By re-defining $\Udown$ in Lemma \ref{slicelemma} to be the image of $\Phi$ we assume from here on that $\Phi$ is a diffeomorphism.
\ifarch
Our next result uses the product decomposition of $\Udown$ into $K$-orbits and the slice $\Slice$ defined by $\EDelta$ in Lemma \ref{slicelemma} to  construct an integrable extension of $\sysI$ (see Definition \ref{DefIE}).
\fi


\begin{thm}\label{extendexist} Let $(\sysI,\dD,\JJ)$ be a Darboux integrable elliptic decomposable system whose  singular system $\calV$ is Pfaffian, 
let $(\theta^i,\eta^r,\sigma^u,\overline{\eta^r},\overline{\sigma^u})$ be a Vessiot coframe based at $\mpt \in \Udown$ (as in Theorem \ref{fouradaptedexist}), and let $\omega^i$ and $\psi^j$ be defined as in equation \eqref{omegaform} of Proposition \ref{getomegas}. 
Let $\Slice \subset \Udown$ and the action $\Gam: K \times \Udown \to \Udown$ satisfy the conditions of Lemma \ref{slicelemma}. 

If we define 
\begin{equation}
\hatH = \spanc \left\{ \muR^j - \gam^* \psi^j, \gam^* \eta^r \right\} \label{Hext},
\end{equation}
then this rank $n+q-d$ sub-bundle of $ T^*(\Slice \times K)\otimes \C$ defines a holomorphic Pfaffian system on $\Slice\times K$.
Moreover, its \realform\ $\hatE=\hatH_\R$ generates an integrable extension $\hatsysE$ of $\,\sysI$ relative to $\gam:\Slice\times K \to \Udown$ which has fiber rank $\dimG$. 
\end{thm}

\begin{proof}  As noted above, the holomorphic Darboux invariants and the 1-forms $\pi^a$ restrict to be holomorphic on $\Slice$. Since $\psi^i= S^i_a \pi^a$ where $S^i_a\in \Invts$, the $\psi^i$ also restrict to be holomorphic 1-forms on $\Slice$. 
Furthermore since the 1-forms $\pi^a=(\eta^r, \sigma^u)$ are $K$-invariant, $\gam^* \pi^a$ is the same as the pullback of $\pi^a\vert_{\Slice}$
to the product, and the same is true for $\gam^* \psi^i$.
The $\muR^i$ are holomorphic on $K$, and thus all the generators for $\hatH$ in equation \eqref{Hext} are holomorphic on $\Slice \times K$.
 
In order to  show that $\hatsysE$ is an integrable extension of $\sysI$, we begin by showing that
$\gam^* \theta^i \in \hatsysE \otimes \C$, which is the crux of the argument.  
We will use the local representation of $\sysI\otimes\C$ given in equation \eqref{LocalI} coupled with the Vessiot coframe provided by Theorem \ref{fouradaptedexist}.

We may replace the 1-forms $\theta^i$ by $\tilde\theta^j= R^i_j \theta^j$ where $R^i_k$ is as in equation \eqref{omegaform}, because $R^i_j\in \Invts$ and $R^i_j$ defines an automorphism of the algebra $\fk=\fg\otimes \C$
 at each point of $\Udown$, so that $\tilde\theta^j$ satisfies the structure equations in Theorem \ref{fouradaptedexist} with the same real constants $C^i_{jk}$. Furthermore since $R(\mpt)$ is real-valued,  $\tilde \theta^i$ is imaginary-valued at $\mpt$, and
so $(\tilde \theta^i, \pi^a, \pibars{b})$ is a Vessiot coframe based at $\mpt$. Now equation \eqref{omegaform} gives
\begin{equation}\label{omegaform2}
\tilde\theta^i = \omega^i - \psi^i -  \tilde P^i_j \overline{\psi^j},
\end{equation}
where $\tilde\theta^i = \tilde P^i_j \overline{\tilde\theta^j} \mod \eta^r , \overline{\eta^r}$.  We now drop the tildes and assume $R$ is the identity matrix for the rest of the proof. From \eqref{omegaform2} we have
\begin{equation}\label{pullbacktheta}
\gam^* \theta^i = \gam^*( \omega^i - \psi^i - P^i_j \overline{\psi^j}).
\end{equation}

To compute the right-hand side of \eqref{pullbacktheta} we begin by finding $\gam^* \omega^j$. Since the $\omega^j$ restrict to be zero on $\Slice$, $T^*K\otimes \C$ is spanned by the $\muR^i$ and their conjugates.  Since $\vZ^R_j \hook \gam^* \omega^i = \gam^*(\vZ_j \hook \omega^i) = \delta^i_j$, then 
\begin{equation}\label{pimomega}
\gam^* \omega^j = \muR^j - B^j_k \overline{\muR^k}
\end{equation}
for some coefficients $B^j_k$.  Since   $\muR^k = \Omega^k_\ell \muL^\ell$ from \eqref{OM2}, and the vector fields $\vX^L_i$ are tangent to the fibers of $\bqg$, then 
$$0 = \vX^L_i \hook  (\muR^j - B^j_k \overline{\muR^k}) = \vX^L_i \hook
(\Omega^j_k \muL^k - B^j_k \overline{\Omega^k_\ell} \,\overline{\muL^\ell})  = \Omega^j_i - B^j_k \overline{\Omega^k_i}.$$
It follows that $B^j_k = \Omega^j_i \overline{\Lambda^i_k}$, so we have
\begin{equation}\label{pullbackomega}
\gam^*\omega^j = \muR^j - \Omega^j_i\overline{\Lambda^i_k}\, \overline{\muR^k}= \Omega^j_i (\muL^i - \overline{\muL^i}).
\end{equation}

Now from \eqref{pullbackomega} it follows that 
$$\gam^* \omega^j = -B^j_k \gam^*\overline{\omega^k},$$
and comparing with \eqref{omegabardentity} shows that $\gam^* P^i_j = -B^i_j = -\Omega^i_k \overline{\Lambda^k_j}$.  Substituting this and \eqref{pullbackomega} into
\eqref{pullbacktheta} yields
\begin{equation}\label{gampulltheta}
\gam^*\theta^i = \muR^i -\gam^* \psi^i - B^i_j \left(\overline{\muR^j} - \gam^* \overline{\psi^j}\right) = \tau^i- B^i_j \overline{\tau^j} 
\end{equation}
where for the sake of convenience in equation \eqref{gampulltheta} we have defined 
\begin{equation}
\tau^j := \muR^j -\gam^* \psi^j \in \hatH.
\label{taudef}
\end{equation}
Equations \eqref{Hext} and \eqref{gampulltheta} now show that $\gam^* \theta^i \in  \hatH \oplus \overline{\hatH} = \hatE\otimes \C$.

We next show that the 2-form generators of $\hatsysE\otimes \C$ are pullbacks of those of $\sysI\otimes \C$. Since $\hatsysE \otimes \C = \{ \tau^j, \overline{\tau^j}, \gam^* \eta^r, \gam^* \overline{\eta^r} \}\subdiff$, 
and the 2-forms $d (\gam ^*\eta^r )$ are simply the pullbacks by $\gam$ of the 2-forms $ d \eta^r \in \sysI\otimes \C$, we only need to consider 
the $d \tau^j$ and their conjugates. 

When we differentiate equation \eqref{taudef}, the Maurer-Cartan equations determine the term $d\muR^j$, and so we only need to compute $d\psi^j$.  If we substitute for $\omega^i$ from equation \eqref{omegaform2} into \eqref{omegastruc} we obtain, using equation \eqref{threedtheta},
\begin{multline*}\tfA^i + P^i_j \overline{\tfA^j} +\tfrac12 C^i_{jk} \theta^j \wedge \theta^k + M^i_{aj} \pi^a \wedge \theta^j + d\psi^i + dP^i_j\wedge \overline{\psi^j} + P^i_j d\overline{\psi^j} \\
= \tfrac12 C^i_{jk} \left( \theta^j + \psi^j + \tilde P^j_\ell \overline{\psi^\ell}\right) \wedge (\theta^k + \psi^k + \tilde P^k_m \overline{\psi^j}),
\end{multline*}
where $\tfA^i= \tfrac12 A^i_{ab}\pi^a \wedge \pi^b$. Isolating the terms involving $\vpi \wedge \vpi$ gives
$$d\psi^i = \tfrac12 C^i_{jk} \psi^j \wedge \psi^k - \tfA^i.$$
Consequently,
$$
d\tau^i = \tfrac12 C^i_{jk} \mu^j_R \wedge \mu^k_R + \gam^* \left(\tfA^i -  \tfrac12 C^i_{jk} \psi^j \wedge \psi^k\right) \equiv \gam^* \tfA^i  \ \mod \tau^j,
$$
and similarly $d\overline{\tau^i} \equiv \gam^* \overline{\tfA^i}$ mod $\overline{\tau^j}$.
Hence we have
\begin{align} \hatsysE \otimes \C &= \{\ \tau^j,\ \overline{\tau^j},\ \gam^*\eta^r, \ \gam^* \overline{\eta^r}\ \}\subdiff \notag \\
			&= \{ \tau^j,\ \overline{\tau^j},\ \gam^* \tfA^j,\ \gam^*\overline{\tfA^j},\ \gam^*\eta^r,\ \gam^* \overline{\eta^r}, 
			\ \gam^*d\eta^r, \
			\gam^* d\overline{\eta^r}\ \}\subalg \label{bigElist} \\
			& = \{ \tau^j , \gam^*( \sysI \otimes \C)\}\subalg , 
\label{QuotientGenerators1}
\end{align}
where in the last step we use \eqref{gampulltheta} to solve for $\overline{\tau^i}$ in terms of $\gam^* \theta^i$ and $\tau^i$, and use the local representation of $\sysI\otimes\C$ given by equation \eqref{LocalI}. (This is where our hypothesis of $\calV$ being Pfaffian is used.)
To obtain the real version of equation \eqref{QuotientGenerators1}, note that \eqref{gampulltheta} can be written 
\begin{equation}
\gam^* \theta^i = \tau^i - B^i_j \overline{\tau^j}= \tau^i-\Omega^i_k \overline{\Lambda^k_j \tau^j}=\Omega^i_k( \Lambda^k_j \tau^j -\overline{\Lambda^k_j \tau^j})
\label{pullgamthe}
\end{equation}
and that the forms $\Lambda^k_j \tau^j -\overline{\Lambda^k_j \tau^j}$ are pure imaginary.  Equation \eqref{pullgamthe} 
allows
us to rewrite \eqref{QuotientGenerators1} in the equivalent form
\begin{equation}
\hatsysE \otimes \C = \{ \Lambda^k_j \tau^j, \gam^*( \sysI \otimes \C)\}\subalg = \{ \Lambda^j_k \tau^k + \overline{\Lambda^j_k \tau^k}, \gam^*( \sysI \otimes \C)\}\subalg. \label{HETC}
\end{equation}
The pointwise span of the real 1-forms $\alpha^j = \Lambda^j_k \tau^k + \overline{\Lambda^j_k \tau^k}$ 
define a sub-bundle $J \subset T^*(\Slice \times K)$ of rank $\dim G$. 
The forms $\alpha^i$
 satisfy $\alpha^i(\vX_j^L)=\delta^i_j$ and hence $J$ is transverse to the fibers of $\gam$ in equation \eqref{sliceCD}.  This shows $J$ satisfies conditions $\text{\rm [i]}$ and $\text{\rm [ii]}$ in \eqref{IntExtDef}. From equation \eqref{HETC} we then have 
\begin{equation*}
\hatsysE  = \{  \alpha^j, \gam^* \sysI \}\subalg.
\end{equation*}
This verifies condition [iii] in \eqref{IntExtDef}, and therefore the real Pfaffian system $\hatsysE$ with $J\subset T^*(\Slice \times K)$ defines an integrable extension of $\sysI$ relative to $\gam:\Slice\times K \to \Udown$, with fiber rank $\dim G=\dimG$.
\end{proof}

\subsection{Proof of the Converse to the Quotient Theorem} \label{ConvTheorem}

The next theorem and its corollary show that the extension $\hatsysE$ in Theorem \ref{extendexist} satisfies the hypotheses of Theorem \ref{C34} and that $\Phi^*\sysI$ on $\Slice \times K/G$ is the quotient of $\hatsysE$ by $G\subset K$.   In particular this shows that every Darboux integrable elliptic system whose singular system is Pfaffian can be obtained locally from the quotient of the \realform\ of a holomorphic Pfaffian system, as in Theorem \ref{C34}.

\begin{thm}\label{loopreduction} Let $(\sysI, \dD,\JJ)$ be a Darboux integrable elliptic decomposable system with $\calV$ Pfaffian, and let $\Slice, K, \hatH, \hatE$ be as in  Theorem \ref{extendexist}  and $\Phi:\Slice\times K/G \to \Udown$ as in diagram \eqref{sliceCD}. Then

\begin{enumerate}[label=\thmenumstyle,ref=\thmenumstyle]
\item 
$\hatH= \spanc \left\{ \muR^j - \gam^* \psi^j, \gam^* \eta^r \right\} $ is a rank $\dimG+\rankVinf-\rankDplus$ holomorphic Pfaffian
system on $\Slice \times K$ invariant under right-multiplication by $K$.

\item  The right action of $K$ on $\Slice\times K$ is transverse  to $\hatH$.

\item  The real form $G$ of $K$ satisfies $\sysI = \Phi^*(\hatsysE/G)$ on $\Udown$, with $G$ acting on the right on $\Slice \times K$.

\item The system $\hatH$ satisfies $\hatH^{(\infty)}=0$.

\end{enumerate}
\end{thm}

\begin{proof}  {\bf Part (1):}  By Theorem \ref{extendexist} we only need to check that $\hatH$ in equation  \eqref{Hext} is  right $K$-invariant. Using the Vessiot coframe based at $\mpt$ from the proof of Theorem \ref{extendexist}, the forms $\pi^a=(\sigma^u, \eta^r)$
on $\Udown$ are invariant with respect to the $\Gam$ action of $K$, and so  by the equivariance of $\gam$ and $\pi^a(\vZ_i)=0$, the forms $(\gam^*\sigma^u,\gam^*  \eta^r)$ are $K$-basic for the left action of $K$ on $\Slice \times K$. A form which is left $K$-basic on $\Slice \times K$ is also 
$K$-basic for the right $K$-action, and so this shows the  $\gam^* \eta^r \in \hatH$ satisfy the conditions in (1).

In equation \eqref{omegaform} the coefficients $S^i_a$ are holomorphic Darboux invariants and thus $\psi^i=S^i_a \pi^a$ are $K$-invariant on $\Udown$. Therefore $\gam^* \psi^i=\gam^*(S^i_a \pi^a)$  are left (and hence right) $K$-basic. 
The forms $\muR^j$ on $\Slice \times K$ are right $K$-invariant, so we conclude that $\hatH$ is right $K$-invariant.  

\medskip
\noindent
{\bf Part (2):}  The holomorphic infinitesimal generators $\vZ^L_i$ defined in \S\ref{conv} for the right action of $K$ on $\Slice \times K$ satisfy $ \vZ^L_i \hook( \muR^j - \gam^* \psi^j)=\Omega^i_j$  by equation \eqref{OM2}.  The matrix $\Omega^i_j$ is pointwise invertible and so the right action of $K$ is transverse to $\hatH$.

\medskip
\noindent
{\bf Part (3):}  In order to compute $\hatsysE/G$ on $S\times K/G$ we need to determine generators for the $G$-basic 1-forms and 2-forms
of $\hatsysE$ on $S\times K$; see \S2.2 in \cite{AFB}.
The commutative diagram \eqref{sliceCD} gives  $\gam^* =\bqg^* \circ  \Phi^* $ so that the forms
\[
\gam^* \tfA^i,\ \gam^*\overline{\tfA^i},\ \gam^*\eta^r,\ \gam^* \overline{\eta^r},\ \gam^*d\eta^r,\ 
			\gam^* d\overline{\eta^r}, \ \gam^* \theta^i
\]
are automatically right $G$-basic on $\Slice\times K$ and by 
\eqref{bigElist}
are among the algebraic generators of $\hatsysE \otimes \C$. 
Since the expression of any $G$-basic form of $\hatsysE \otimes \C$ in terms of the generators in \eqref{bigElist}
cannot include any combination of the $\tau^j$ or $\overline{\tau^j}$ (as they are not even $\piG$-semibasic), we have
\begin{equation*}
(\hatsysE\otimes \C)/G= \{ \ \Phi^* \tfA^i,\ \Phi^*\overline{\tfA^i}, \  \Phi^*\eta^r,\ \Phi^*\overline{\eta^r},\ \Phi^*d\eta^r, \
		 \Phi^*d\overline{\eta^r}, \ \Phi^* \theta^i \ \}\subalg= \Phi^* (\sysI\otimes \C)
\end{equation*}
on $\Slice \times K/G$, and hence $\hatsysE/G=\Phi^* \sysI$.

\medskip
\noindent
{\bf Part  (4):} On $\Slice \times K$ the singular bundle for $\hatsysE$ is  $W= T^*_{1,0}(\Slice \times K)\oplus\overline{\hatH}$ (see equation \eqref{DefW}) which has rank 
$ 2\dimG+2\rankVinf-\rankDplus$, 
while by equation \eqref{quotientV} the singular bundle $V$ for $\sysI$ satisfies $W/G =\Phi^*V$ and has rank 
$ \dimG+2\rankVinf-\rankDplus$. 
We have 
$W^{(\infty)}= T^*_{1,0}(\Slice \times K)\oplus \overline{\hatH}^{(\infty)}$ which has rank $
\dimG+q +\rank \overline{\hatH}^{(\infty)} $.
Since $G$ acts transversely to $W^{(\infty)}$, 
then $\rank(W^{(\infty)}/G)= \rankVinf+ \rank \overline{\hatH}^{(\infty)}$. Since $G$ also acts transversely to $W$, by Theorem 2.2  in \cite{AFB} we have $\Phi^*(V^{(\infty)})= W^{(\infty)}/G$.  Since since $\rank V^{(\infty)}=\rankVinf$, this implies 
that $\rank \overline {\hatH}^{(\infty)}=0$. 


\end{proof}

We now identify a manifold $N$ with a (right) $K$ action and Pfaffian system $\sysE$ such that $\sysI=\sysE/G$. This 
canonical construction follows from the fact that $\sysI = \Phi^*(\hatsysE/G)$. Let $ N =(\Phi^{-1})^* (\Slice\times K)$ be the pullback principal $G$ bundle of $\bqg:\Slice\times K \to \Slice\times K/G$ from diagram \eqref{sliceCD} 
and let $\Psi:N \to \Slice \times K$ be the induced diffeomorphism. The manifold $N$ has a complex structure
induced from $\Psi$, with free action  $\rho: N\times K \to N$ 
$$ 
 (x, (y,k))\cdot h = ( x', (y,k h)) \quad \text{where } x=\gam(y, k), x'= \gam(y, k h),
 $$
for $y\in \Slice$ and $h,k \in K$.  The quotient map $\piG : N \to N/G$  satisfies $\piG= \gam\circ \Psi$ and gives rise to the commutative diagram
\begin{equation}
\begin{gathered}
\tikzstyle{line} = [draw, -latex', thick]
\begin{tikzpicture}
\node(B) {$N =(\Phi^{-1})^* (\Slice\times K)$ };
\node[below of=B, node distance=20mm](I1) { $\Udown$ };
\node[right of=B, node distance=60mm](I3) { ${\Slice\times K}$ };
\node[below of=I3, node distance=20mm](I2)  { $\Slice \times K/G  $};
\path[line] (B) --node[left] {$\piG$} (I1);
\path[line] (I2) -- node[above] {$\Phi$} (I1);
\path[line] (B) -- node[above] {$\Psi$} (I3);
\path[line] (I3) -- node[left] {$\qG$} (I2);
\path[line] (I3) --node[left] {${\gam\quad \ } $} (I1);
\end{tikzpicture}
\end{gathered}
\label{CDM}
\end{equation}

\begin{cor}\label{CircleD}  The Pfaffian system  $H=\Psi^* \hatH$ is holomorphic on $N$ and $K$ is a symmetry group of $H$ which acts transversely to $H$. The pullback Pfaffian system $\sysE =\Psi^* \hatsysE$ 
is the \realform\ of  $H$ and satisfies $\sysI = \sysE/G$. 
\end{cor} 
\begin{proof}   The proof  follows by standard commutative diagram arguments.
\end{proof}

Corollary \ref{CircleD} completes the final proof that every elliptic Darboux integrable system with Pfaffian singular system arises as a quotient as in Theorem \ref{C34}. Finally we note that Corollary \ref{CircleD} and Theorem \ref{loopreduction} can be summarized by the commutative diagram of 1-forms for the differential systems in Corollary \ref{CircleD}.
\begin{equation}
\begin{gathered}
\tikzstyle{line} = [draw, -latex', thick]
\begin{tikzpicture}
\node(B) { $E\otimes \C=\{\ \piG^*\eta^r, \ \piG^*{\overline \eta^r},\ \piG^*{ \theta^i} ,\ \piG^*{\overline \theta^i} , \ \Psi^* \tau^i   \ \} _N $ };
\node[below of=B, node distance=20mm](I1) { $I\otimes\C=\{\ \eta^r,\  \overline{\eta^r},\ \theta^i, \ \overline{\theta^i} \ \}_\Udown$};
\node[right of=B, node distance=100mm](I3) { $ \{\ \gam^*\eta^r,\ \gam^*\overline{\eta^r}, \ \gam^*{ \theta^i},\  \gam^*\overline {\theta^i},  \ \tau^i \ \} _{\Slice\times K}$};
\node[below of=I3, node distance=20mm](I2)  { $\{\ \Phi^* \eta^r,\ \Phi^* \overline{\eta^r},\Phi^* \theta^i, \Phi^*\overline{ \theta^i} \ \}_{\Slice \times K/G} $};
\path[line] (B) --node[left] {$\piG$} (I1);
\path[line] (I2) -- node[above] {$\Phi$} (I1);
\path[line] (B) -- node[above] {$\Psi$} (I3);
\path[line] (I3) -- node[left] {$\qG$} (I2);
\path[line] (I3) --node[left] {${\gam\quad \ } $} (I1);
\end{tikzpicture}
\end{gathered}
\label{CD3b}
\end{equation}
See Example \ref{benequation3} where Corollary \ref{CircleD} and the commutative diagram \eqref{CD3b} are demonstrated.

\ifarch
\begin{remark}\label{Qloop} The commutative diagrams \eqref{CDM} and \eqref{CD3b} can also be constructed starting with $N$, $H$ and $K$ as in Theorem \ref{C34} as follows. Take the  chart $\Psi$ on $\Upset$ used in the proof of Theorem \ref{DICF} along with the Vessiot coframe on $\Udown$ constructed in Theorem \ref{DICF}. Using the adapted chart $\Phi^{-1}: \Udown \to \calW \times K/G$, the systems on the right side of diagram \eqref{CD3b} are just the pullbacks (or the local chart representation) of $E\otimes \C$ on $\Upset$ and $I\otimes\C$ on $\Udown$ from Theorem \ref{C34} and Theorem \ref{DICF}.  
In this case the forms $\tau^i$ in the diagram \eqref{CD3b} 
are defined using equations \eqref{OM2} and \eqref{DTC} by $\tau^i = \Omega^i_j (\Psi^{-1})^* \thetau^j$.
\end{remark}
\fi

It is worth noting that if $\sysI$ is a maximally Darboux integrable Pfaffian system then $K$ is the trivial group and Theorem \ref{loopreduction} and Corollary \ref{CircleD} degenerate to Proposition \ref{MaxPfaff}.

\def\MB{\Lambda}
\def\ga{{{a}}}
\def\npt2{\mathrm n}

\subsection{The Vessiot Algebra}\label{TVA}  
Given an elliptic decomposable Darboux integrable system $(\sysI,\DD,\JJ)$  on a manifold $M$,  about each point $\mpt\in M$ we have associated a coframe in Theorem \ref{fouradaptedexist} where  the $C^i_{jk}$ in equation \eqref{threedtheta} are the structure constants of a real Lie algebra $\fg$. In this section we prove that if $M$ is connected, then the algebra $\fg$ does not depend on the point $\mpt$. We then show that  this is an invariant of the Darboux integrable system and  as in the case of Darboux integrable hyperbolic systems we call  $\fg$ the Vessiot algebra of $(\sysI,\DD,\JJ)$.

We begin by showing that all Vessiot coframes based at a given point $\mpt$ determine the same Lie algebra up to isomorphism.  

\begin{thm} \label{Vessmn} Let $(\theta^j,\pi^a,\overline{\pi^a})$ and $(\widetilde \theta^j,\alpha^a,\overline{\alpha^a})$ be two Vessiot adapted coframes based at $\mpt$ and defined on an open set $\Udown$ (as in Theorem \ref{fouradaptedexist}), and satisfying
structure equations \eqref{threeadapt} with real coefficients $C^j_{k\ell}$ and $\widetilde C^j_{k\ell}$ respectively.
Then the two corresponding algebras $\fg$ and $\widetilde{\fg}$ are isomorphic as real Lie algebras.
\end{thm}
\begin{proof} Condition (2) of Theorem \ref{fouradaptedexist} implies there exists functions $Q^j_k,S^j_r,T^j_r$ such that
\begin{equation}
\widetilde \theta^j = Q^j_k \theta^k + S^j_r \eta^r + T^j_r {\overline \eta^r},
\label{FourChange}
\end{equation}
while condition (1) of Theorem \ref{fouradaptedexist} implies $Q^i_j$ is invertible at each point of $\Udown$.  If $(\vZ_j, \vU_a, \overline{\vU_a})$ denotes the dual frame to $(\theta^j,\pi^a,\overline{\pi^a})$, then taking the exterior derivative of equation \eqref{FourChange} and using the structure equations in Theorem \ref{fouradaptedexist} for both coframes implies $\overline{\vU_a}(Q^i_j)=0$, hence by Lemma \ref{DIX} we have $Q^i_j\in \Invts$. Then matching the coefficients of $\theta^j \wedge \theta^k$ in
the exterior derivative of equation \eqref{FourChange} gives
\begin{equation}
Q^i_m C^m_{jk} = \widetilde C ^i_{mn} Q^m_j Q^n_k.
\label{Raut}
\end{equation}
Equation \eqref{Raut} implies that the {\em complex} Lie algebras defined by $C^i_{jk}$ and $\widetilde{C}^i_{jk}$ are isomorphic. 
However, since $\overline{\theta^k}\vert_\mpt = -\theta^k\vert_\mpt$ and $\overline{\widetilde \theta^j}\vert_\mpt = - \widetilde\theta^j\vert_\mpt$, taking conjugates in \eqref{FourChange} gives $Q^j_k(\mpt) \in \R$.  Hence when equation \eqref{Raut} is evaluated at $\mpt$, $Q^i_j(\mpt)$ defines an isomorphism of the real Lie algebras defined by $C^i_{jk}$ and $\widetilde{C}^i_{jk}$

\end{proof}

The next theorem uses a Vessiot coframe based at $\mpt$ with corresponding algebra $\fg$ to construct a Vessiot coframe based at 
an arbitrary point  $\npt2$ in a neighbourhood of $\mpt$ such that the associated algebra for the Vessiot coframe at $\npt2$ is again $\fg$.

\def\ga{h}
\begin{thm}\label{CCUdown} Let $(\theta^j,\pi^a,\overline{\pi^a})$ be a Vessiot coframe based at $\mpt\in \Udown$ with $C^i_{jk}$ defining the algebra $\fg$ and let $\Phi:\Slice \times K/G  \to\Udown$  be the diffeomorphism from Lemma \ref{slicelemma}.  
Let $\npt2 \in \Udown$, let $(s,\ga G )=\Phi^{-1}(\npt2) \in \Slice \times K/G$, and let $\widetilde \theta^j = \MB^j_k(\ga) \theta^k$
where $\Lambda^i_j(\ga)=[\Ad^*(\ga)]^i_j$ is defined in equation \eqref{OM1}. Then
\begin{equation*}
(\widetilde \theta^j, \pi^a, \overline{\pi^a})
\end{equation*}
is a Vessiot coframe based at $\npt2$, and the $\widetilde{\theta}^j$ satisfy structure
equations \eqref{threedtheta} with the same coefficients $C^i_{jk}$ as those satisfied by $\theta^j$.
\end{thm}

\begin{proof}  
The structure equations satisfied by $\widetilde \theta^j = \MB^j_k(\ga) \theta^k$ are of the same form as those in \eqref{threedtheta} with
the same coefficients $C^i_{jk}$ because $\Lambda(\ga)=[\Ad^*(\ga)]$ defines an automorphism of $\fk^*=\fg^*\otimes \C$. We now need to show that $\widetilde \theta^j $ is imaginary-valued on $T_\npt2 \Udown$ which will finish the proof of the theorem.

Without loss of generality we may assume as in the proof of Theorem \ref{extendexist} that (after a change in frame, as in equation \eqref{omegaform2}) we have
$$
\theta^i = \omega^i -\psi^i -P^i_j \overline{\psi^j}.
$$
Let $(\vZ_k, \vU_a , \overline{\vU_a})$ denote the dual frame and note that
\begin{equation}
\widetilde \theta^j (\vU_a) = \MB^j_k(\ga) \theta^k(\vU_a)=0, \quad \widetilde \theta^j (\overline{\vU_a})=0
\label{FSIm1}
\end{equation}
so that we need only show
\begin{equation}
\widetilde \theta^j(\Re(\vZ_k))\vert_{\npt2}= \MB^j_\ell(\ga) \omega^\ell(\Re(\vZ_k))\vert_{\npt2} , \qquad 
\widetilde \theta^j(\Im(\vZ_k))\vert_{\npt2} = \MB^j_\ell(\ga) \omega^\ell(\Im(\vZ_k)) \vert_{\npt2}
\label{FSIm2}
\end{equation}
are pure imaginary.

We show this by using the map $\gam$ in equation \eqref{sliceCD} where $\gam(s,a)=\Phi(s,aG)=\npt2$ and equation \eqref{pullbackomega} to find,
$$
\gam^ *( \MB^i_j(\ga) \omega^j\vert_{\npt2})  =\Lambda^i_j(\ga) \Omega^j_k(\ga)(\mu^k_L -\overline{\mu^k_L})\vert_{(s,\ga)}=(\mu^i_L -\overline{\mu^i_L})\vert_{(s,\ga)},
$$
since $\Omega = \Lambda^{-1}$. The right side of this equation is pure imaginary valued on real vectors which along with  $\gam$ being  a submersion shows the terms in equation \eqref{FSIm2} are pure imaginary. Therefore by equations \eqref{FSIm1} and \eqref{FSIm2} we have $\tilde \theta^j$ 
is imaginary valued on $ T_{\npt2 }\Udown$.
\end{proof}

Theorem \ref{CCUdown} shows that the  mapping which assigns to each point  $\mpt\in M$ the corresponding algebra $\fg$ of a Vessiot coframe based at $\mpt$ is locally constant. 
From this we have the following result.

\begin{cor}\label{CVA1} Let $(\sysI,\DD,\JJ)$ be a Darboux integrable elliptic decomposable system on a connected manifold $M$.  Then the algebra $\fg$ associated to a Vessiot coframe at a point $\mpt$ is independent of $\mpt$. We call this the {\it Vessiot algebra}
of $(\sysI,\DD,\JJ)$.
\end{cor}

\ifjour
\else
The Vessiot coframes constructed  in Theorem \ref{DICF} give the second corollary.
\fi

\begin{cor}
Let $(\sysI, \DD,\JJ)$ be a Darboux integrable elliptic decomposable system on a connected manifold such that 
$\sysI$ arises as a quotient $\sysE/G$ of the \realform\ of a holomorphic system $\sysH$ satisfying the conditions of Theorem \ref{C34}. Then the Lie algebra $\fg$ of $G$ is the Vessiot algebra of  $(\sysI, \DD,\JJ)$.
\end{cor}

In the proof of Lemma \ref{slicelemma}, where the vector fields $\vZ_i$ in the dual frame are infinitesimal generators for the action of $K$,
we have assumed there exists a global action of $K$ on $M$ such that $\Slice$ defines a global slice through $\mpt$. 
In general a global action may not exist and the mapping $\Phi$ in equation \eqref{sliceCD} may only define a local diffeomorphism on a neighbourhood $\Udown$ of $(s_0, [e]) \in \Slice \times K/G$. However  the same proof of Theorem \ref{CCUdown} holds in  this case for points $\npt2 \in \Udown$ and Corollary \ref{CVA1} holds similarly.

Motivated by definitions in complex geometry and the case of second order PDE in the plane, discussed in Example \ref{introE1} and below, we make the following definition.

\begin{defn}\label{VessAlgDef}
Two  elliptic distributions $(\DD_\ell, \JJ_\ell)$, $\ell=1,2$  are {\defem equivalent} if there exists
a diffeomorphism $\Phi: M_1 \to M_2$ such that $\Phi_* \DD_1 =\DD_2$ and either
\begin{align*}
\Phi_* \JJ_1(X) &= \JJ_2 (\Phi_* X)\qquad \forall \ X \in \DD_1
\intertext{or}
\Phi_* \JJ_1(X) &= -\JJ_2 (\Phi_* X)\qquad \forall \ X \in \DD_1.
\end{align*}
\end{defn}
Due to the fact that $\dD_{\ell,+}$ and $\dD_{\ell,-}$ are the $\pm \ri$ eigenspaces of $\JJ_\ell$, these conditions 
hold if and only if 
\begin{equation}
\Phi_* \dD_{1,\pm} = \dD_{2,\pm} , \quad {\rm or} \quad \Phi_* \dD_{1,\pm} = \dD_{2,\mp}
\label{EquivDs}
\end{equation}


As expected, the Vessiot algebra is well-behaved under this notion of equivalence.

\begin{thm} \label{VessiotAlg} 
Let $(\DD_\ell,  \JJ_\ell)$, $\ell=1,2$ be equivalent Darboux integrable elliptic distributions with Vessiot algebras $\fg_1,\fg_2$ and with equivalence $\Phi:M_1\to M_2$ such that $M_1$ and $M_2$ are connected. Then the Vessiot algebras are isomorphic.
 \end{thm}

\begin{proof}  We first assume $\Phi_* \dD_{1,\pm}= \dD_{2,\pm}$ in equation \eqref{EquivDs}. Since the singular bundles
are given by $V_\ell =( \DD_{\ell,-} )^\perp$, we have
\begin{equation}
\Phi^* V_2 = V_1.
\label{PBV21}
\end{equation}
Now let $(\theta^i_2, \pi^a_2,\overline{ \pi^a_2}) $ be a Vessiot coframe for $(\DD_2, \JJ_2)$ based at $\mpt_2$ and defined on open set $\Udown_2$.  We claim that
$$
(\theta^i_1, \pi^a_1,\overline{ \pi^a_1}):=\Phi^*(\theta^i_2, \pi^a_2,\overline{ \pi^a_2}) 
$$ 
defines a Vessiot coframe for $(\DD_1, \JJ_1)$ based at $\mpt_1$ defined on $\Udown_1$, where $\mpt_1=\Phi^{-1}(\mpt_2)$ and $\Udown_1=\Phi^{-1}(\Udown_2)$. 
Equation \eqref{PBV21} shows that the algebraic conditions  in parts \ref{Fa1} and \ref{Fa2} of  Theorem \ref{fouradaptedexist} are satisfied by $(\theta^i_1, \pi^a_1,\overline{ \pi^a_1})$. While the structure equations in part \ref{Fa3} are satisfied by commuting the exterior derivative and pullback, with coefficients $C^i_{jk}$ the same for both coframes.  Moreover, since the $\theta^i_2$ are imaginary-valued at $\mpt_2$, the
$\theta^i_1$ are imaginary-valued at $\mpt_1$.
Therefore by Theorem \ref{Vessmn} and Corollary \ref{CVA1} the Lie algebras $\fg_1$ and $\fg_2$ are isomorphic. 

Lastly, if we assume instead that $\Phi_* \dD_{1,-}= \dD_{2,+}$ then we can make a similar argument, but with 
$$
\theta^i_1 =\Phi^*\overline{\theta^i_2}, \quad \pi^a_1 = \Phi^* \overline{\pi^a_2}, \quad \overline{\pi^a_1} = \Phi^*\pi^a_2 
$$ 
providing a Vessiot coframe for $(\DD_1, \JJ_1)$.\end{proof}

We now examine the invariance property of the Vessiot algebra for Darboux integrable second order elliptic PDE in the plane.
Let $ F_\ell(x^i,u,u_i,u_{ij})=0, \ell=1,2$ be two elliptic PDE in the plane as in Example \ref{introE1} and let 
$(\calI_\ell,\DD_\ell, \JJ_\ell)$, be the corresponding decomposable elliptic system defined on the 7-manifolds $M_{{\ell}}$ which we assume to be connected. Suppose now that $\Phi:M_{1} \to M_{2}$ defines an equivalence of the differential systems so that $\Phi^* \calI_2 = \calI_1$.  The conformal bilinear form defined in equation \eqref{GKform} satisfies 
$$\Phi^* \langle\ , \ \rangle_2=\lambda \langle\ , \ \rangle_1$$ 
where $\lambda$ is a nowhere vanishing function on $M_{1}$. Using this fact, the argument leading to equations \eqref{PDECont}  gives $\Phi_* \JJ_1=\pm \JJ_2 \Phi_*$. Therefore
 $\Phi$ satisfies Definition \ref{VessAlgDef} and defines an equivalence of the corresponding elliptic distributions  $(\DD_\ell, \JJ_\ell)$. We then have the following.

\begin{cor}\label{NonCont}  If $F_\ell=0$ for $\ell=1,2$ are two Darboux integrable second order  PDE in the plane whose Vessiot algebras are not isomorphic, then $F_\ell$ are not locally equivalent.
\end{cor}


\subsection{Examples}\label{MoreEx}
In this section we demonstrate the construction of the Vessiot coframe in Theorem \ref{fouradaptedexist} along with the
construction of the extension $\sysE$ and the quotient map from Theorem \ref{loopreduction} and Corollary \ref{CircleD} for a number of examples. 

In Example \ref{benequation3} we re-visit the equation $W_{\bar z}=\tfrac12 |W|^2$ appearing in Examples \ref{benexample1}, \ref{benexample2} 
and give all the details and formulas leading to diagrams \eqref{CDM}, \eqref{CD3b}.
\ifarch
In Example \ref{sixesandsevens} we include a short discussion on the possible geometry of the holomorphic system $H$ from Theorem \ref{loopreduction} when
the original Darboux integrable system exists on a 6- or 7-dimensional manifold. We relate this to the fact that locally there are
only three Darboux integrable systems on a 6-manifold \cite{Eendebak}
\fi
In Example \ref{GoursatThree} we demonstrate the construction of the Vessiot coframe for a system on a 7-dimensional manifold arising from a second-order PDE in the plane; this example is sufficiently general that every step of the proof of Theorem \ref{fouradaptedexist} is required. 
\ifarch
In Example \ref{BHE} we again demonstrate the construction of the systems $H$ and $E$ in a higher dimensional case. In particular we show 
that the bi-harmonic equation can be written as a Darboux integrable system on a 12-dimensional manifold with a 4-dimensional Vessiot algebra. 
The constructions for the Vessiot coframe and the extension $E$ are given, along with the solution formula provided by the quotient map.
\fi

\bigskip
\begin{example} \label{benequation3}
\def\BenInv{\xi}
\def\Pbar{\overline{P}}
\def\BenInvbar{\overline{\xi}}
\def\Wobar{\overline{W_1}}
We return to the first-order complex PDE \eqref{beneqncomplex} considered in Examples \ref{benexample1} and \ref{benexample2}.  For the sake of convenience, we will express the Pfaffian system $\sysI$ encoding this equation in terms of complex coordinates $z,W,W_1=D_z W$ and their conjugates on $M$, so that (see equation \eqref{defbeta}),
$$\sysI\otimes \C = \spanc\{ dW-W_1\dz -\tfrac12|W|^2 d\zbar, d\Wbar-\Wobar d\zbar - \tfrac12|W|^2 dz\}.$$
(In terms of the coordinates used in Example \ref{benexample1}, $z=x+\ri y$, $W=u+\ri v$ and
from \eqref{deeWmaudi} we have $W_1= u_1 + \ri v_1 -\tfrac12(u^2+v^2)$.) In Example \ref{benexample1} we showed that this is an elliptic decomposable system which is Darboux-integrable, with holomorphic Darboux invariants given by $z$ and 
$$\BenInv = \dfrac{W_1}{W} - \dfrac{W}2.$$
In Example \ref{benexample2} we showed how this system arises as a quotient of the canonical contact system on the space $N$ of 2-jets of holomorphic functions, by the action of the real part of a complex solvable Lie group $K$.  Here, we will show how to reconstruct $K$ and the system $\sysE$ on $N$  as an integrable extension of $\sysI$, as outlined in Theorem \ref{loopreduction} and Corollary \ref{CircleD}.

We begin by following the algorithm in Theorem \ref{fouradaptedexist} and compute a Vessiot coframes for this system, 
in terms of coordinates $z,W,W_1$, and defined on the open subset $\Udown \subset M$ where $W\ne 0$. 
We choose the point where $z=0$, $\BenInv=0$ and $W=\ri$ as the `origin' $\mpt\in\Udown\subset M$.
Following Steps 1--3 in the proof of Theorem \ref{fouradaptedexist}, we obtain a coframe 
$(\theta^i, \pi^a,\pibars{a})$ on $M$ with $\pi^1=dz,\pi^2=d\BenInv$ and
\begin{equation}
\begin{aligned} 
\theta^1 &= \dfrac{dW}{W} -\dfrac{d\Wbar}{\Wbar} - \BenInv\dz +\BenInvbar\,d\zbar\\
\theta^2 &= 2 \dfrac{d\Wbar}{|W|^2} - dz - \dfrac{\Wbar + 2\BenInvbar}{W}d\zbar - \ri \theta^1.
\end{aligned}
\label{ben4adapt}
\end{equation}
(Note that $\theta^1, \theta^2$ are imaginary-valued at $\mpt$.)
These satisfy structure equations
\begin{align*}
d\theta^1 &= \pi^1 \wedge \pi^2 - \pibars1 \wedge \pibars2, \\
d\theta^2 &= -\ri \pi^1 \wedge \pi^2 + (\ri + \frac{2}{W}) \pibars1 \wedge \pibars2
- \theta^1 \wedge \theta^2 + ( (\ri \BenInv-1) \theta^1+ \BenInv \theta^2)  \wedge \pi^1,
\end{align*}
The further change of coframe in Step 4 is not needed, as the coefficients of $\theta^1 \wedge \theta^2$ in the structure equations are already real constants, and we have a Vessiot coframe based at $\mpt$. We also note that the matrix $P^i_j$ in Lemma \ref{realatm}
is 
\begin{equation}
P = \begin{pmatrix} -1 & 0 \\ \ri + \dfrac{2-\ri \Wbar}{W} & \dfrac{\Wbar}W\end{pmatrix}
\label{Patm}
\end{equation}
and satisfies $P^i_j(\mpt) = - \delta^i_j$.

Next, we apply the procedure from Proposition \ref{getomegas} 
to obtain the 1-forms $\omega^1, \omega^2$ in equation \eqref{omegaform}.  
In particular, we compute that 
\begin{equation}
R = \begin{pmatrix}1 & 0 \\ 0 & 1 \end{pmatrix}, \quad 
\begin{bmatrix}\psi^1 \\ \psi^2 \end{bmatrix} = \begin{bmatrix} \BenInv\dz \\ (1-\ri \BenInv)\dz \end{bmatrix},
\label{Bpsi}
\end{equation}
which along the matrix $P$ in equation \eqref{Patm} gives
\begin{equation}
\omega^1 = \dfrac{dW}{W} - \dfrac{d\Wbar}{\Wbar}, \qquad \omega^2 = 2\dfrac{d\Wbar}{W\Wbar}-\ri \omega^1.
\label{Benomega}
\end{equation}

We now find the action of $K$ and the slice as in Lemma \ref{slicelemma}. The vector fields dual to $\omega^1, \omega^2$ (and annihilated by the $\pi^a$ and $\pibars{a}$) are
$$\vZ_1 = \tfrac12 \ri W\left( W\dib{W} + \Wbar\dib{\Wbar}\right) + W \dib{W},\qquad
\vZ_2 = \tfrac12 W\left( W\dib{W} + \Wbar\dib{\Wbar}\right),
$$
satisfying $[\vZ_1, \vZ_2] = \vZ_2$. This coincides with the bracket relations of the right-invariant holomorphic vector fields on the solvable Lie group $K$ defined in Example \ref{benexample2}.
Thus, there is a left-action of $\Gam:K \times M \to M$ such that $\Gam_* c_1 \dib{c_1} = \vZ_1$ and $\Gam_* c_1 \dib{c_2} = \vZ_2$.  (Here, we use the same holomorphic coordinates $c_1, c_2$ on $K$ as in \eqref{solvableKdefinition}, and let $G\subset K$ be the subgroup where $c_1,c_2\in \R$.)

Let $\Slice \subset M$ be the maximal integral manifold of the Frobenius system spanned by 
$\{\omega^1,\omega^2\}$ through $\mpt$; in other words, let $\Slice$ be the submanifold where $W=\ri$.  As in Lemma \ref{slicelemma}, we let $\gam: \Slice \times K \to M$  be the restriction of $\Gamma$ to $\Slice \subset M$.  Since $\vZ_1, \vZ_2$ are imaginary-valued at points on $\Slice$, at these points the isotropy of the $K$-action is the real subgroup $G \subset K$.  Thus, $\gam$ induces a local diffeomorphism
$\Phi: \Slice \times K/G \to \Udown$.  By integrating the vector fields $\vZ_1, \vZ_2$ given above, we find that $\Phi$ is determined by $z=z, \BenInv=\BenInv$ and
$$W = \dfrac{-2c_1}{c_2+\ri c_1 - (\overline{c_2}-\ri \overline{c_1})}$$
where  the coordinates $z,\BenInv$ on $\Udown$ restrict to provide coordinates on $\Slice$.
Note that the right-hand side of this equation can be expressed in terms of local coordinates on $K/G$, i.e., functions on $K$ which are invariant under right-multiplication by $G$; for example, if 
\begin{equation}
r=\Im c_1 / \Re c_1, \qquad s=\Im c_2/ \Re c_1
\label{rscoords}
\end{equation}
then in these coordinates $W = (\ri - r)/(1+s)$.

Next, we compute the integrable extension $\sysE$ in Theorem \ref{extendexist}.  The right-invariant (1,0)-forms on $K$ are given by
$\muR^1 = (1/c_1)dc_1, \muR^2 = (1/c_1) dc_2$; then using equations \eqref{Hext}, \eqref{Bpsi}, and \eqref{rscoords} the extension on $\Slice \times K$ is generated by the holomorphic Pfaffian system
$$\hatH = \spanc \{ dc_1 - c_1 \BenInv\dz, dc_2 - c_1 (1-\ri \BenInv)\dz\}.$$

If we now let $ \Wdown= \Udown/K$ with $q_K: \Udown \to \Wdown$ the quotient map, then $q_K:\Slice\to \Wdown$ is a diffeomorphism and using this diffeomorphism we may replace $\Slice$ in diagram \eqref{CDM} with $\Wdown$ and obtain the commutative diagram,

\begin{equation}
\begin{gathered}
\tikzstyle{line} = [draw, -latex', thick]
\begin{tikzpicture}
\node(B) {  $\Upset $ };
\node[below of=B, node distance=20mm](I1) { $\Udown$};
\node[right of=B, node distance=30mm](I3) { $ \Wdown \times K $};
\node[below of=I3, node distance=20mm](I2)  { $\Wdown \times K/G $};
\path[line] (B) --node[left] {$\piG$} (I1);
\path[line] (I2) -- node[above] {$\Phi$} (I1);
\path[line] (B) -- node[above] {$\Psi$} (I3);
\path[line] (I3) -- node[left] {$\qG$} (I2);
\end{tikzpicture}
\end{gathered}
\label{CDPB}
\end{equation}
where again $N=\Upset$ is the pullback $G$-bundle $\Upset= (\Phi^{-1})*(\Wdown\times K)$. At this point diagram \eqref{CDPB} is nothing more that an adapted bundle chart on $\Udown$ induced by the principal $G$-bundle $\piG: \Upset \to \Udown$. Finally we let $\sysE$ the \realform\ of $H=(\Phi^{-1})^* \hatH$ which completes diagram \eqref{CD3b}. We now give a concrete description of $\sysE$ which utilizes Example \ref{benexample2}.

In Example \ref{benexample2}  we have already represented $\sysI$ as a quotient $\sysE/G$. In particular, using coordinates $(z,w_0,w_1,w_2)$ on $\Upset$ as in Example \ref{benexample2}, 
the map $\piG:\Upset \to \Udown$ is 
$$
\piG:(z,w_0,w_1,w_2) \mapsto ( z=z,\ 
W=\dfrac{2\ws1}{\wbars0 - \ws0},\ W_1=\dfrac{2\ws2}{\wbars0 - \ws0} + \dfrac{2\ws1^2}{(\wbars0 - \ws0)^2} ).
$$
The bundle chart $\Phi^{-1}$ in \eqref{CDPB} induced by the choice of $ \Slice \subset \Udown$ is given in coordinates 
by
$$
\Phi^{-1}:( z,W,W_1) \mapsto (z, \xi=\dfrac{W_1}{W} - \dfrac{W}2,\ r=-\dfrac{\Re W}{\Im W} ,\ s= \dfrac{1}{\Im W} -1\ ),
$$
and from \eqref{rscoords} we have
$$
\qG:(z,\xi, c_1, c_2) \mapsto (\ z,\ \xi,\ r=\Im c_1 / \Re c_1, \ s=\Im c_2/ \Re c_1 \ ).
$$
When we work in coordinates to write down a mapping $\Psi$ that makes the diagram \eqref{CDPB} commute, as a well as being an equivariant holomorphic $K$-bundle chart, we obtain the unique solution
$$
\Psi:(z, w_0,w_1,w_2) \mapsto ( z,\ \xi=w_2 w_1^{-1},  c_1=w_1,  c_2=w_0-\ri w_1)
$$
in coordinates.  It is then easy to check $\Psi^* \widetilde{H}$ coincides with the contact system of Example \ref{benexample2} defined by equation \eqref{defj2holcontact} and $\sysE$ is  the \realform\ of the contact system.
Thus we have constructed all the mappings and systems in the commutative diagrams \eqref{CDM}, \eqref{CD3b}.

\ifarch
Finally it should be noted  that the two Vessiot coframes given in equations \eqref{ben4adapt} and \eqref{ben4adaptup} are in fact identical and that one could have constructed the diagrams \eqref{CDM}, and \eqref{CD3b}  in accordance with Remark \ref{Qloop}.
\fi
\end{example}


\ifarch
\begin{example}\label{sixesandsevens} In Examples \ref{EB2} and \ref{benexample2} we have identified two inequivalent rank 2  Darboux integrable Pfaffian systems on a 6-dimensional manifold, while in Example \ref{benequation3} we constructed the extension for Example \ref{benexample2} using Theorem \ref{loopreduction} and Corollary \ref{CircleD}. In general if $M$ is a $6$ dimensional manifold and $\sysI$ is a rank $2$ Darboux integrable Pfaffian system, then $N$ constructed in Corollary \ref{CircleD} is a 4 dimensional complex manifold while
  $H$ is a rank 2 holomorphic Pfaffian system with derived flag $[2,1,0]$ (since $H^{(\infty)}=0$). The local structure equations for $H=\spanc\{\theta^0, \theta^1\}$  can then be written in the coframe $\{ \theta^0, \theta^1,\pi^1, \pi^2\}$ as
  $$
d \theta^0 = \rho_1 \wedge \theta^1 + \rho_2 \wedge \theta^0 , \qquad d \theta^1 \equiv \pi^1 \wedge \pi^2 \quad {\rm mod } \ H
$$
where  $\rho_i = a_i \pi^1 + b_i \pi^2$ are  holomorphic and $\rho_1$ is nowhere zero.  The holomorphic vector field
$$
X = b \partial_{\pi^1} - a \partial_{\pi_2} 
$$
is a  Cauchy characteristic for $\theta^0$, and therefore  we may (locally) consider  $\theta^0$ as a holomorphic one form on the complex 3 manifold $N/X$ where $d \theta^0 \wedge \theta^0 \neq 0$. In standard coordinates (see Theorem 3.1 \cite{DBlair}), 
$$
\theta^0 = dw - w_z dz. 
$$
There are two locally inequivalent two-dimensional infinitesimal contact actions in the tables 1 and 3  in Olver \cite{PurplePete} given by,
$$
\fk_1= {\rm span} \{ \partial_w, \ w \partial_w\}, \quad  \fk_2= {\rm span} \{ \partial_w, \ z \partial_w\}. 
$$
The quotients in this case lead to the two Darboux integrable given in  equation \eqref{Wabeleqn} from example \ref{EB2} and  equation \eqref{beneqncomplex} from example \ref{benexample1}. These two equations together with the Cauchy-Riemann equations are the only 3 (locally up to contact transformations) Darboux integrable equations on a 6 manifold \cite{Eendebak}.

In the case where $M$ is a $7$ manifold corresponding to a second order PDE in the plane with rank 3 Darboux integrable Pfaffian system $\sysI$ with a 3 dimensional Vessiot group $G$,  then in Corollary \ref{CircleD} $N$ has 5 complex dimensions and
the holomorphic system $H$ has rank 3, with either $[3,2,1,0]$ or $[3,2,0]$ as the only possibilities for the ranks of its derived system.  
\end{example}
\fi

The following example is intended to illustrate the construction of the Vessiot coframe in a case
where Steps 1--3 lead to a coframe where the coefficients $C^i_{jk}$ in \eqref{threedtheta} are non-constant functions of the holomorphic Darboux invariants, and a non-trivial change of coframe is needed in Step 4.

\begin{example}\label{GoursatThree}
We will show below that the second-order elliptic PDE
\begin{equation}\label{Gthreeq}
u_{z\zbar} = \dfrac{\sqrt{1+u_z^2}\sqrt{1+u_{\zbar}^2}}{\cos u}
\end{equation}
for real $u$ is Darboux-integrable.  This equation was obtained from a classification of Darboux-integrable second-order hyperbolic Monge-Amp\`ere equations by Goursat and Vessiot (see \S5 in \cite{ClellandIvey} for details), by replacing independent variables $x$ and $y$ by $z$ and $\zbar$.  (Other DI elliptic equations can be obtained from the Goursat-Vessiot list by similar modifications.)

\def\pbar{\overline{p}}
\def\rbar{\overline{r}}
\def\xibar{\overline{\xi}}

In order to rationally parametrize the jets of solutions of \eqref{Gthreeq} we introduce a complex parameter $p \ne \pm1$ such that 
$$u_z = \dfrac{2p}{1-p^2},\qquad \sqrt{1+u_z^2} = \dfrac{1+p^2}{1-p^2}.$$
Then the PDE \eqref{Gthreeq} is equivalent to 
$$\dfrac{\di p}{\di \zbar} = \dfrac{(1-p^2)(1+\pbar^2)}{2(1-\pbar^2)\cos u}.$$

On $\R \times \C \times \C \times \C$ use real coordinate $u$ on the first factor and complex coordinates $z,p,r$ on the other factors, and let $M$ be the open subset be the where $p^2 \ne 1$ and $\sin u \ne 0$. On $M$ define a (real) Pfaffian system $\sysI$ 
such that $I \otimes \C = \{ \beta_0, \beta_1, \overline{\beta_1}\}$, where
$$\beta_0 := du - \dfrac{2p}{1-p^2}dz - \dfrac{2\pbar}{1-\pbar^2}d\zbar, \quad
\beta_1 = dp - r\dz - \dfrac{(1-p^2)(1+\pbar^2)}{2(1-\pbar^2)\cos u}d\zbar.$$
Computing the derivatives of these 1-form generators, and finding decomposable 2-forms, leads to singular systems $\rV$ and $\rVb$ where 
$$\rV = (I \otimes \C) \oplus \{ dz, (1 -\pbar^2) \cos^2 u\,dr - \left( p(1+\pbar^2)(\sin u - r \cos u) + \pbar (1+p^2)\right)d\zbar\} $$
We designate $\rV$ as {\em the} singular system, since this makes coordinate $z$ a holomorphic Darboux invariant; by computing $\Vinf$ we uncover a second Darboux invariant
$$
\xi = \dfrac{r-\tfrac12 (1+p^2)\tan u}{1-p^2}= \frac{1}{2}\left( \frac{u_{zz}}{\sqrt{1+u_z^2}} -\sqrt{1+u_z^2} \tan(u) \right)
$$
As $\Vinf$ has rank two, $(\sysI,\rV) $ is minimally Darboux integrable.

We continue our computations using coordinates $u$ as well as $z, p, \xi$ and their conjugates on $M$.   
Choosing the point $\mpt$ defined by $u=0, z=0, p=0$ and $\xi=1$ as origin, we follow Steps 1--3 in the proof of Theorem \ref{fouradaptedexist} to obtain an adapted coframe with $\pi^1 = dz, \pi^2=d\xi$ and 
$$
\begin{bmatrix} \theta^1 \\ \theta^2 \\ \theta^3 \end{bmatrix} = A \begin{bmatrix} 
\beta_0 \\ \beta_1 \\ \overline{\beta_1} \end{bmatrix}, \qquad 
A := \dfrac{1}{4(1-p^2)} \begin{pmatrix} 
-\dfrac{p}{\xi} & -1 & \dfrac{(1-p^2) \sin u  + (1+p^2) \xi^{-1}\cos u}{1-\pbar^2} \\
\dfrac{\ri p}{\xi} & -\ri & \dfrac{ \ri (1-p^2) \sin u  - \ri (1+p^2) \xi^{-1}\cos u}{1-\pbar^2} \\
2\ri (1+p^2) & 0 & \dfrac{-8\ri p \cos u}{1-\pbar^2}
\end{pmatrix}.
$$
Since the 1-forms $(\beta_0, \beta_1, \overline{\beta_1})$ are permuted in a simple way by conjugation, the matrix $P$ such that $\theta^i = P^i_j \thetabars{j}$ is given by
$$P = A P_0 \overline{A^{-1}}, \qquad P_0 := \begin{pmatrix} 1 & 0 & 0 \\ 0 & 0 & 1 \\ 0 & 1 & 0 \end{pmatrix}.$$
Note that $P^i_j(\mpt)=-\delta^i_j$, so this coframe also satisfies the conditions of Lemma \ref{realatm}.

To obtain the change of coframe in Step 4, we set all variables in $P$ -- including $\xibar$ and $\zbar$ but not including   $\xi$ and $z$  -- equal to their values at $\mpt$, which determines the matrix $K$ in equation \eqref{constC}.  The inverse of the resulting matrix is
$$K^{-1} = \begin{pmatrix} \tfrac12 (\xi+1) & \tfrac12 \ri (\xi-1) & 0 \\[6pt] \tfrac12 \ri (1-\xi) & \tfrac12 (\xi+1) & 0 \\ 0 & 0 & 1\end{pmatrix}.
$$
Making the change from equation \eqref{Kinvtheta}, $\widetilde\theta^i = (K^{-1})^i_j \theta^j$ results in a coframe satisfying structure equations
$$\left.
\begin{aligned}
d\theta^1 &\equiv -2 \theta^2 \wedge \theta^3 - \tfrac14 \ri( 1+ 2\xi) \theta^3 \wedge \pi^1 - \tfrac14 \pi^1 \wedge \pi^2 \\
d\theta^2 &\equiv - 2 \theta^1 \wedge \theta^3 -\tfrac14 (2\xi-1) \theta^3 \wedge \pi^1 + \tfrac14 \ri \pi^1 \wedge \pi^2 
\\
d\theta^3 &\equiv -16 \theta^1 \wedge \theta^2 +2\ri (1+2\xi) \theta^1 \wedge \pi^1  +2(1-2\xi) \theta^2 \wedge \pi^1
\end{aligned}\right\} \mod \pibars1 \wedge \pibars2,
$$
where we have dropped the tildes. It is easy to check that the Vessiot algebra is isomorphic to $\mathfrak{sl}(2,\R)$.  

Since the Vessiot algebra is semi-simple, the forms $\widehat \theta^i$ in equation \eqref{AFV5A} can be found algebraically using the fact shown in \cite{AFV} that $M^i_{aj}$ define an automorphism of the complexification of the Vessiot algebra. We solve (see equation (4.57)  in \cite{AFV})
$$
M^i_{aj} = S^\ell_a C^i_{\ell j}, \quad R^i_j = \delta^i_j
$$
which in this example leads to
$$
S^\ell_a =\frac{1}{8} \left[\begin{array}{cc}
-2\xi +1 & 0 
\\
 -\ri (2 \xi+1) & 0 
\\
 0 & 0 
\end{array}\right]
$$
and the forms $\omega^i$ in equation \eqref{omegaform} (with $R^i_j=\delta^i_j$) are then
\begin{equation}\label{G35A}
\left[\begin{array}{c} \omega^1 \\ \omega^2  \\ \omega^3 \end{array}\right]=
\left[\begin{array}{c} 
\frac{p {du}}{4 (p^{2}-1)}+\frac{{dp}}{4 (p^{2}-1)}+\frac{\left(\left(\cos u -\sin  u  \right) p^{2}+\sin u +\cos  u \right) {d\overline{p}}}{4 \left(p^{2}-1\right) \left({\overline{p}}^{2}-1\right)}
\\
-\frac{\ri p {du}}{4(p^{2}-1)}+\frac{\ri {dp}}{4(p^{2}-1)}-\frac{\ri \left( p^{2}(\sin u  +\cos u )-\sin u +\cos u \right) {d\overline{p}}}{4\left(p^{2}-1\right) \left({\overline{p}}^{2}-1\right)}
\\
-\frac{\ri  \left(p^{2}+1\right) {du}}{2(p^{2}-1)}-\frac{2 \,\ri p \cos  u\  {d\overline{p}}}{\left(p^{2}-1\right) \left({\overline{p}}^{2}-1\right)}
\end{array} \right]
\end{equation}

In principle, given the Vessiot coframe and the 1-forms $\omega^i$ constructed above, one can obtain the integrable
extension $\hatsysE$ of Theorem \ref{extendexist}.  We leave further details to the interested reader.

\end{example}

\ifarch
\begin{example}\label{BHE}
The biharmonic equation $\Delta (\Delta u)=0$ is equivalent to a system of two second-order elliptic PDEs,
\begin{equation}\label{biharms}
\Delta u = 2v, \qquad \Delta v = 0
\end{equation}
(the factor of 2 being for the sake of convenience).  As we will see, it is this version of the biharmonic equation that is Darboux-integrable in the sense of our definition.

To encode \eqref{biharms} by an exterior differential system, we restrict the standard contact system on $J^2(\R^2, \R^2)$ to the 12-dimensional submanifold $M$ cut out by the equations.
In detail, we take coordinates $x$, $y$, $u$, $v$, $u_i$, $v_i$, $u_{ij}=u_{ji}$ and $v_{ij}=v_{ji}$ on the jet space (where $1\le i,j\le 2$).   
On $M$, we retain most of these but replace $u_{ij}$ and $v_{ij}$ by $u_{12}$, $v_{12}$, $r=\tfrac12(u_{11}-u_{22})=u_{11}-v$ and $s=\tfrac12(v_{11}-v_{22})=v_{11}$.  In these coordinates, the contact system pulls back to $M$ to be the rank 6 Pfaffian system
\begin{multline*}\sysI = \{ du - u_1\dx - u_2\dy, 
du_1 - (v+r) \dx -  u_{12}\dy, 
du_2 - u_{12}\dx - (v-r)\dy,  \\
dv-v_1 \dx - v_2 \dy, 
dv_1 - s \dx - v_{12} \dy, 
dv_2 - v_{12}\dx + s \dy\}.
\end{multline*}
The rank 6 distribution $\dD$ on $M$ annihilated by these 1-forms is spanned by the coordinate vector fields $\di_{r}$, $\di_s$, $\di_{u_{12}}$, $\di_{v_{12}}$ and
\begin{align*}
D_x &= \di_x + u_1 \di_u + (v+r) \di_{u_1} + u_{12} \di_{u_2}
+ v_1 \di_v + s \di_{v_1} + v_{12} \di_{v_2},\\
D_y &= \di_y + u_2 \di_u + u_{12} \di_{u_1} + (v-r) \di_{u_2} + v_2 \di_v + v_{12} \di_{v_1} - s\di_{v_2}.
\end{align*}
As in  Example \ref{benexample1}, examining the complex singular integral elements of $\sysI$ lets us determine a sub-complex structure on $\dD$ which makes it an elliptic distribution.  Specifically, we split $\dD \otimes \C$ as $\Dplus \oplus \Dminus$, where
$$\Dplus = \{ D_x - \ri D_y, \di_{r} + \ri \di_{u_{12}}, \di_{s} + \ri \di_{v_{12}}\}, \qquad \Dminus = \overline{\Dplus}.$$
As usual we let $V \subset T^*M \otimes \C$ be the sub-bundle annihilated by $\Dminus$.  Then $V^{(\infty)}$ is spanned by the differentials of
$$x+\ri y, \quad v_1- \ri v_2, \quad s-\ri v_{12}, \quad  r - \ri u_{12} - \tfrac12 (x-\ri y) (v_1-\ri v_2).$$
(In other words, these functions are a basis for the holomorphic Darboux invariants of $\sysI$.)
Thus, $V^{(\infty)}+ \overline{V}$ spans the whole of $T^*M \otimes \C$, and
so by Definition  \ref{EDIdef}, $\sysI$ is Darboux integrable.
Because  $V^{(\infty)}$ has rank 4, the integrability is neither minimal or maximal in the sense of Definition \ref{maxmindef}.

In what follows we will construct a Vessiot coframe for $\sysI$, as well as compute the group 
action and integrable extension given by Theorem \ref{extendexist}.  To facilitate
calculations, we will change coordinates on $M$ so that our coordinate 
system incorporates as many Darboux invariants as possible.  In detail, we will use coordinates 
$$z=x+\ri y, \  u_z= \tfrac12(u_1 -\ri u_2), \  v_z =\tfrac12(v_1-\ri v_2), 
\  p = r - \ri u_{12}- \overline{z} v_z, \  q = s - \ri v_{12}$$
along with their complex conjugates, as well as $u$ and $v$.  In terms of these coordinates,
$\sysI \otimes \C$ is generated by the following 1-forms
\begin{multline*} 
Du:=du-u_z dz - \overline{u_z}d\zbar, \quad
Du_z:=du_z - \tfrac12 (p+\zbar v_z) dz - \tfrac12 v d\zbar, \\
Dv:=dv - v_z dz - \overline{v_z}d\zbar, \quad
Dv_z:=dv_z - \tfrac12 q\,dz,
\end{multline*}
together with the complex conjugates $\overline{Du_z},\overline{Dv_z}$.  

Let $\mpt \in M$ be the origin in our coordinate system.
Using the process described the proof of Theorem \ref{fouradaptedexist} (in particular Steps 1-3 and Lemma \ref{realatm}), we find an adapted coframe $(\vtheta, \veta, \overline{\veta}, \vsigma,\overline{\vsigma})$ on $M$ given by
\begin{align*}
\theta^1 &= \ri (Du - \zbar \overline{Du_z}), \\
\theta^2 &= Du_z - \overline{Du_z} - \tfrac12 \zbar Dv, \\
\theta^3 &= \ri( Du_z + \overline{Du_z}) - \tfrac12 \ri \zbar Dv,\\
\theta^4 &= \ri Dv, \\
\eta &= 2 Dv_z, \quad \sigma^1 = dz, \quad \sigma^2 =dp, \quad \sigma^3 = dq.
\end{align*}
(The factors of $\ri$ ensure that the 1-forms $\theta^i$ are imaginary-valued at the origin $\mpt$.)
These satisfy the structure equations
\begin{align*}
d\theta^1 &=  -\ri \zbar \sigmabars1 \wedge (\tfrac12 \sigmabars2+\tfrac14 z \etabar) 
&+&\tfrac12 \sigma^1 \wedge (\ri \theta^2 +\theta^3) \\
d\theta^2 &= \tfrac12 (\sigma^1 \wedge \sigma^2 - \sigmabars1 \wedge \sigmabars2) 
- \tfrac14 (z+\zbar)\sigmabars1 \wedge \etabar   
&+&\tfrac12 \ri \sigma^1 \wedge \theta^4\\
d\theta^3 &= \tfrac12 \ri(\sigma^1 \wedge \sigma^2 + \sigmabars1 \wedge \sigmabars2)  +\tfrac14 \ri (z-\zbar)\sigmabars1 \wedge \etabar
&+&\tfrac12 \sigma^1 \wedge \theta^4\\
d\theta^4 &= \tfrac12 \ri (\sigma^1 \wedge \eta + \sigmabars1 \wedge \etabar)\\
d\eta &= \sigma^1 \wedge \sigma^3,
\end{align*}
which agree with the structure equations \eqref{threedthetap}.  Moreover, equation \eqref{threedtheta} is also satisfied 
since the coefficients $C^i_{jk}$ are zero, and so we have a Vessiot coframe. Thus, the Vessiot algebra is Abelian, and $K$ is the additive group $\C^4$.


The 1-forms $\omega^i$ provided by Proposition \ref{getomegas} (see \eqref{omegaform}) are 
\begin{equation}
\left[
\begin{matrix}
\omega^1 \\
\omega^2 \\
\omega^3 \\
\omega^4 \\
\end{matrix}
\right]
=
\left[ \begin {array}{cccc} 1&-{\frac {1}{2}}\ri\, z&-{\frac {1}{2}}z&0
\\ \noalign{\medskip}0&1&0&-{\frac {1}{2}}\ri\, z\\ \noalign{\medskip}0&0&1&
-{\frac {1}{2}} z \\ \noalign{\medskip}0&0&0&1\end {array} \right] 
\left[
\begin{matrix}
\theta^1 \\
\theta^2 \\
\theta^3 \\
\theta^4 
\end{matrix}
\right]
+ 
\left[ \begin {array}{c} -{\frac {1}{2}}\ri\, p\, z \,   \sigma^1  
\\ \noalign{\medskip} \frac{1}{2}(p+  z \, v_z)\, \sigma^1
\\ \noalign{\medskip} \frac{1}{2} \ri(p -z\, v_z )\, \sigma^1
\\ \noalign{\medskip} \ri\, v_z \,\sigma^1 \end {array} \right]
+
\left[ \begin {array}{c} -{\frac {1}{2}}\ri \, \overline{p}\, \overline{z}\,  \sigmabars1  
\\ \noalign{\medskip} \frac{1}{2}(\overline{p}+  \overline{z}\, \overline{v_z}) \, \sigmabars1
\\ \noalign{\medskip} \frac{1}{2}\ri (\overline{p} -\overline{z}\, \overline{v_z} )\,  \sigmabars1
\\ \noalign{\medskip} \ri\, \overline{v_z} \, \sigmabars1 \end {array} \right]
\label{omegaBH}
\end{equation}

In this case, equation \eqref{omegastruc} implies that the 1-forms given by \eqref{omegaBH} are closed.
(In fact, the $\omega^i$ can be expressed as constant-coefficient linear combinations of the differentials of the functions $u - z u_z - \zbar \overline{u_z} +\tfrac12 |z|^2 v$, $u_z - \tfrac12 v \zbar$, 
$\overline{u_z} - \tfrac12 v z$ and $v$.)
The remaining coordinates $z,p,q,v_z$ and their conjugates restrict to give 
coordinates along the leaves of distribution $\EDelta$ annihilated by the $\omega^i$.
Transverse to this distribution, the generators of the $K$-action on $M$ are the dual vector fields to the $\omega^i$ annihilated by the $\eta, \etabar, \vsigma, \vsigmabar$, i.e.,
\begin{align*}
\vZ_1 &= -\ri \di_u, \\
\vZ_2 &= \tfrac12 (\di_{u_z} - \di_{\overline{u_z}}+ (z-\zbar) \di_u) \\
\vZ_3 &= -\tfrac12 \ri (\di_{u_z} + \di_{\overline{u_z}}+ (z+\zbar) \di_u)\\
\vZ_4 &= -\tfrac12\ri (2\di_v + z \di_{\overline{u_z}}+ \zbar \di_{u_z} + |z|^2 \di_u)
\end{align*}
Note that these vector fields are pure imaginary.  (In general, their real parts only vanish at points along the maximal
integral manifold of $\EDelta$ through the origin; however, since $K$ is Abelian in this case the isotropy subalgebra 
is the same at every point.)  

If we use coordinates $c_j = a_j + \ri b_j$ on $K=\C^4$, then the action of $K$ on $M$ generated by the $\vZ_j$ is
\begin{align*}
u &\mapsto u + 2b_1 + \ri b_2 (z-\zbar) + b_3 (z+\zbar) + b_4 |z|^2,\\
u_z &\mapsto u_z + b_3 +\ri b_2  + b_4 \zbar, \\
\overline{u_z} &\mapsto \overline{u_z} + b_3 - \ri b_2+ b_4 z ,\\
v &\mapsto v  + 2b_4.
\end{align*}
We now apply Theorem \ref{extendexist} to define an integrable extension of $\sysI$.  In the theorem we will take $\mpt$ to be the origin,
and thus $\Slice \subset M$ is the 8-dimensional submanifold defined by $u=0,v=0$ and $u_z=0$.   We use $z,p,q,v_z$ as complex coordinates on $\Slice$, and define the submersion $\gam:\Slice \times K \to M$ by applying the group action to points on $\Slice$:
$$u = 2b_1 + \ri b_2 (z-\zbar) + b_3 (z+\zbar) + b_4 |z|^2,\ u_z = b_3 + \ri b_2  + b_4, \ v=2b_4,$$
Since $K$ is Abelian, the matrix $\Lambda$ on $K$ is equal to the identity matrix, and formula 
\eqref{pimomega} gives $\gam^*\omega^j = dc^j - \overline{dc^j} = 2\ri db^j$.

In this example, the 1-forms $\psi^i$ in equation \eqref{omegaform} are given in explicitly in equation \eqref{omegaBH} as
\begin{equation}
\psi^1 = -\frac{1}{2}\ri \, \,p\, z \dz, \quad \psi^2 = \tfrac12 (p + z \, v_z) dz, \quad \psi^3 = \frac{ 1}{2}\ri ( p - z\, v_z) dz, \quad \psi^4 = \ri v_z \dz.
\label{BHpsi}
\end{equation}
The right- (and left-)invariant 1-forms on $K$ are simply the differentials $dc^j$, so the extension $\sysE$ is the real Pfaffian system
corresponding to the holomorphic Pfaffian system defined in equation \eqref{Hext} using the forms in equation \eqref{BHpsi} given by
\begin{multline*} H = \{ \mu_j^R - \psi^j , \eta \} \\
= \{ dc_1 + \frac{\ri}{2} p z \dz, dc_2 - \tfrac12 (p + z v_z) dz, dc_3 - \frac{\ri}{2} ( p - z v_z) dz, dc_4 - \ri v_z \dz, dv_z - \tfrac12 q\dz \}.
\end{multline*}
The derived flag of $H$ indicates that $H$ can be put in Goursat normal form (see \cite{Yamma}).  In fact, if we make the change of complex coordinates on $S \times K$
\begin{multline*} k_0 = -(c_2+\ri c_3 +\ri c_4 z), \  k_1 = -\ri c_4, \  k_2 = v_z, \  k_3 = \tfrac12 q, \\
  \ell_0 = -2\ri c_1 + (c_2-\ri c_3)z,\  \ell_1 = c_2 - \ri c_3, \  \ell_2 = p,
\end{multline*}
then $H$ becomes a holomorphic contact system:
$$H = \{ dk_0 -k_1 \dz, dk_1 - k_2\dz, dk_2 - k_3\dz, d\ell_0 - \ell_1 \dz, d\ell_1 - \ell_2 \dz \}.$$
Of course, a solution to this system can be given by $k_0 = f(z)$, $\ell_0 = g(z)$, $k_1 = f'(z)$, $\ell_1 = g'(z)$, etc., 
for arbitrary holomorphic functions $f,g$.   Since, in terms of these coordinates,
$$u = \Re(\ell_0 + \zbar k_0),$$
then we have the solution formula $u= \Re( g(z) + \zbar f(z))$ for the biharmonic equation.

\end{example}
\fi

\section*{Remarks on Classical Darboux Integrability and Closed Form Solutions}\label{ClosedForm}

Darboux integrable equations often admit closed form general solutions as seen in the example of the  Liouville equation in \eqref{ELpm}. 
Both forms of the elliptic Liouville equation  admit the holomorphic Darboux invariant
$$
\xi = u_{zz} -\frac{1}{2} u_z^2.
$$
It is easy to check by taking the derivative $\partial_{\overline{z}} \xi$  which when evaluated on a solution to the Liouville equation $u_{xx}+u_{yy}=\pm 2 e^u$  
is zero and hence $\xi$ defines a holomorphic function when restricted to a solution of the elliptic Liouville equation. In particular this implies that if $h(z)$ is any holomorphic function, then
\begin{equation}
u_{xx}+u_{yy} = \pm 2 e^u, \qquad u_{zz} -\frac{1}{2} u_z^2 = h(z)
\label{TDE}
\end{equation}
define an  integrable system of total differential equations (where all second derivatives are known and all integrability conditions are satisfied) for $u(x,y)$.  This demonstrates the classical notion of Darboux integrability. If  we substitute say $u_{-}(x,y) =\ln \frac{ 4 f(z)' \overline {f(z)'} }{(1+f (z)\overline{f(z)})^2}$ from equation \eqref{ELpm} into the second equation \eqref{TDE} we get
\begin{equation}
\frac{f'''}{f'}-\frac{3}{2} \left(\frac{f''}{f'}\right)^2=h(z)
\label{Pcurve}
\end{equation}
where the left hand side is the Schwarzian derivative of $f(z)$. It is not possible to find $f(z)$  in closed form in terms of $h(z)$ using quadratures. This is a consequence of the Vessiot algebra for these equations being simple and that equation \eqref{Pcurve} is an equation of Lie type on $SL(2,\C)$. So while we have a closed form formula for the general solution given in equation \eqref{ELpm}, we have no closed form formula  for  $f(z)$  in terms of $h(z)$ in equation \eqref{Pcurve}. Also of note is that equation \eqref{Pcurve} is the prescribed curvature problem for holomorphic curves in $\C{\rm P}^1$ with the standard $SL(2,\C)$ action \cite{Flanders}.

If the Vessiot algebra is solvable as in equation 
\begin{equation*}
\dfrac{\di W}{\di \zbar} = \tfrac12 |W|^2,
\end{equation*}
then the closed form solution with prescribed holomorphic invariant can be obtained by quadrature. For this example the general solution is derived in Example \ref{benexample2} to be
$$
W = \dfrac{2f'(z)}{\overline{f(z)}-f(z)},
$$
while a holomorphic Darboux invariant is given in equation \eqref{QDI2} of Example \ref{benexample2} as
$$
\xi = \frac{W_z }{W} -\frac{W}{2} .
$$
In this case the prescribed Darboux invariant problem $\xi=h(z)$ can be solved by quadrature
$$
\frac{f''(z)}{f'(z)} = h(z)\qquad {\rm and} \qquad f(z) = \int {\rm e} ^{\int h(z) dz} dz.
$$



%




\end{document}